\documentclass[11pt]{article}

\usepackage{graphicx,,psfrag}

\usepackage{amssymb,amsmath,amsthm,enumerate}
\usepackage{fullpage}

\usepackage{amsmath,amsfonts,latexsym, amssymb,amsthm}
\usepackage{color,hyperref,psfrag,epsfig}

\newcommand{\ABS}[1]{{{\left| #1 \right|}}} 
\newcommand{\BRA}[1]{{{\left\{#1\right\}}}} 
\newcommand{\PAR}[1]{{{\left(#1\right)}}} 

\newcommand{\bgt}{\begin{itemize}}
\newcommand{\ent}{\end{itemize}}

 \newcommand{\IND}{\mathbf{1}}

\newcommand{\tr}{ \operatorname{tr}}
\newcommand{\Tr}{\operatorname{Tr}}

\newcommand{\E}{\mathbb{E}}
\newcommand{\dE}{\mathbb{E}}

\newcommand{\R}{\mathbb{R}}
\newcommand{\C}{\mathbb{C}}\newcommand{\dC}{\mathbb{C}}

\newcommand{\cH}{\mathcal{H}}

\newcommand{\veps}{\varepsilon}

\newcommand{\bbm}{\begin{bmatrix}}
\newcommand{\ebm}{\end{bmatrix}}
\newcommand{\bes}{\begin{equation*}}
\newcommand{\ees}{\end{equation*}}
\newcommand{\be}{\begin{equation}}
\newcommand{\ee}{\end{equation}}
\newcommand{\beqy}{\begin{eqnarray}}
\newcommand{\eeqy}{\end{eqnarray}}
\newcommand{\beq}{\begin{eqnarray*}}
\newcommand{\eeq}{\end{eqnarray*}}

\newcommand{\bpm}{\begin{pmatrix}}
\newcommand{\epm}{\end{pmatrix}}

\newtheorem{theorem}{Theorem}[section]
\newtheorem{proposition}[theorem]{Proposition} 
\newtheorem{lem}[theorem]{Lemma}

\newtheorem{cor}[theorem]{Corollary}

\long\def\symbolfootnote[#1]#2{\begingroup
\def\thefootnote{\fnsymbol{footnote}}\footnote[#1]{#2}\endgroup}

\def\eq{\begin{eqnarray*}}
\def\qe{\end{eqnarray*}}
\def\eqa{\begin{eqnarray}}
\def\qea{\end{eqnarray}}

\def\cA{{\mathcal A}}
\def\cE{{\mathcal E}}
\def\cK{{\mathcal K}}

\def\tr{\mbox{Tr}\;}

\def\e{{\varepsilon}}
\def\bR{{\mathbb R}}
\def\bC{{\mathbb C}}
\def\bE{{\mathbb E}}
\def\bP{{\mathbb P}} \def\dP{{\mathbb P}}

\newcommand{\EMAIL}[1]{E-mail:~\texttt{\href{mailto:#1}{#1}}}

\newcommand{\AND}{{\quad \hbox{ and } \quad}}
\def\stab{\mathrm{Stab}}
\def\cF{{\mathcal F}}
\def\rank{\mathrm{rank}}

\title{Delocalization at small energy for heavy-tailed
random matrices}
\author{Charles Bordenave\footnote{CNRS \& Universit\'e Toulouse III, France. \EMAIL{charles.bordenave@math.univ-toulouse.fr}.} \,  and Alice Guionnet\footnote{CNRS \& \'Ecole Normale Sup\'erieure de Lyon, France, and  MIT, Cambridge, USA. Partially supported by the Simons Foundation and by NSF Grant DMS-1307704. \EMAIL{aguionne@ens-lyon.fr}.}
}

\begin{document}
\maketitle

\begin{abstract}
We prove that the eigenvectors associated to small enough eigenvalues of an heavy-tailed symmetric random matrix  are delocalized with probability tending to one as the size of the matrix grows to infinity. The delocalization is measured thanks to a simple criterion related to the inverse participation ratio which computes an average ratio of $L^4$ and $L^2$ -norms of vectors. In contrast, as a consequence of a previous result, for random matrices with sufficiently heavy tails, the eigenvectors associated to large enough eigenvalues are localized according to the same criterion. The proof is based on a new analysis of the fixed point equation satisfied asymptotically by the law of a diagonal entry of the resolvent of  this matrix. 
\end{abstract}

\section{Introduction}

Anderson localization has attracted a lot of interest in both mathematical and physical communities recently. One of the most tractable  model to study this phenomenon is given by random Schr\"odinger operators on trees, see notably \cite{AT,AAT,MR1492789,AW}.
It was shown that at small energy the system displays delocalized waves whereas at large energy waves are localized. This phenomenon is  related  to the transition between a continuous spectrum and a discrete spectrum.  Even, a transition between these two phases at a given energy, the so-called mobility edge, could be proved \cite{AW}.  Such a
a transition is expected to happen in much more  general settings, see e.g. \cite{PH}.  In this article we shall prove the existence of a similar phenomenon for random matrices with heavy tails, as conjectured in \cite{BouchaudCizeau,Slanina,TBT}. This is in  contrast with the full delocalization observed for light tails Wigner matrices \cite{ESY09a, EKY11, ESY09b}.
Indeed, we shall prove that for L\'evy matrices with heavy enough tail, eigenvectors are delocalized for small enough energy whereas they are localized for large enough energy. We are not able to prove a sharp transition but the mobility edge value is predicted in \cite{TBT} based on the replica trick. In fact, we already proved in \cite{BG} that eigenvectors are delocalized provided the entries have roughly speaking  finite $L^1$-norm, whereas a localization phenomenon appears for sufficiently heavy tail and large energy.
However, we left open the question of proving delocalization at small energy and very heavy tails, or at least to exhibit a single criterion which would allow to distinguish these two phases. In this article we remedy this point. 

Let us first describe roughly our main results. Consider a symmetric  matrix $A$ of size $n \times n$ with independent equidistributed  real entries above the diagonal. Assume that the tail of $A_{ij}$ is such that, for some $0< \alpha < 2$, 
$$ n \mathbb P(|A_{ij}|\ge t)\simeq_{t\to\infty} t^{-\alpha}\,,$$
(in a sense which will be made precise later on). Then, for $z \in \C \backslash \R$, consider the following fractional moment of the resolvent:
$$y^n_z(\beta)=\frac{1}{n}\sum_{k=1}^n (\Im (A-z I)^{-1}_{kk})^\beta\,.$$
For $\beta=1$, as $n$ goes to infinity and then $z$ goes to $E \in \R$ on the real line, $y^n_z(1) / \pi $ converges towards the spectral density which turns out to be positive \cite{BAG1,BDG09,BCC11}.
However, we proved in \cite{BG} that for $\beta=\alpha/2$, and for sufficiently heavy tails ($0 < \alpha<2/3$),  as $n$ goes to infinity and then $z$ goes to $E\in\mathbb R$ large enough, $y^n_z(\alpha/2)$ goes to zero. This can be shown to imply a localization of eigenvectors with large enough eigenvalue or energy. On the other end, we prove in the present article that for $0 < \alpha<2$ (outside a countable subset of $(0,2)$), 
  as $n$ goes to infinity and then $z$ goes to $E\in\mathbb R$ small enough, $y^n_z(\alpha/2)$ is bounded below by a positive constant. Back to eigenvectors, this in turn allows to prove delocalization at small energies versus localization at high energies 
  according to the following criterion. Consider an orthonormal basis of $\R^n$ of eigenvectors of $A$. Let $I=[E-\eta_n,E+\eta_n]$ be an interval of the real line so that $\eta_n$ goes to zero as $n$ goes to infinity. Let $\Lambda_I$ denote the set of eigenvectors of $A$ with eigenvalues in $I$ and set, if $\Lambda_I$ is not empty, 
  $$Q_I =n\sum_{k=1}^n\left(\frac{1}{|\Lambda_I|}\sum_{u\in\Lambda_I} \langle u,e_k\rangle^2\right)^2\,.$$
 We will explain below why $Q_I$ is related to the nature of eigenvectors in $\Lambda_I$. Then, the main result of this article is that  for  $\eta_n$ going to zero more slowly than $n^{-\rho}$ for some $\rho>0$, for $E$ large enough, $Q_I$ goes to infinity (Theorem \ref{th:locvect}), whereas for $E$ small enough, it remains finite (Theorem \ref{th:main}).

Let us now describe our results more precisely. 
For integer $n \geq 1$, we consider an array $(A_{ij})_{1 \leq i \leq j \leq n}$ of  i.i.d. real random variables and set, for $i > j$, $A_{ij} = A_{ji}$. Then, we define the random symmetric matrix:
$$A = ( A_{ij} )_{1 \leq i , j \leq n}.$$
The eigenvalues of the matrix $A$ are real and are denoted by $\lambda_n \leq \cdots \leq \lambda_1$. We also consider an orthogonal basis $( u_1, \ldots, u_n ) $ of $\R^n$ of eigenvectors of $A$, for $1 \leq k \leq n$, $A u_k = \lambda_k u_k$.

If $X_{11} = \sqrt n A_{11}$ is a random variable independent of $n$ and with variance equal to $1$, then $A$ is a normalized Wigner matrix. In the large $n$ limit, the spectral properties of this matrix are now well understood, see e.g. \cite{BaiSil,BAG21,AGZ10,erdossurvey,MR2784665}. The starting point of this analysis is the Wigner's semicircle law, which asserts that for any interval $I \subset \R$, the proportion of eigenvalues in $I$ is asymptotically close to 
$\mu_{sc}(I) $ where $\mu_{sc}$ is the distribution with support $[-2,2]$ and density $f_{sc} ( x) = \frac 1 {2 \pi} \sqrt { 4 - x ^2}$. Many more properties of the spectrum are known. For example, if $X_{11}$ is centered and has a subexponential tail, then, from \cite{ESY09a, ESY09b}, for any $ p \in (2,\infty]$ and $\veps >0$, with high probability,
\begin{equation}\label{eq:delocWi}
\max_{1 \leq k \leq n}   \|u_k \|_p  \leq n^{ 1 / p - 1/2 + \veps },
\end{equation}
 where  for $u \in \R^n$, $\| u  \|_p =  \left( \sum_{i=1} ^n |u_i|^p  \right)^{1 /  p}$ and $\| u \|_\infty = \max |u_i|$. This implies that the eigenvectors are strongly delocalized with respect to the canonical basis.

In this paper, we are interested in heavy-tailed matrices, it corresponds to the assumption that the measure defined on $\R_+  = (0,\infty)$,
\begin{equation}\label{eq:convLambda}
L_n ( \cdot )  = n \dP ( |A_{11}|^2 \in \cdot) 
\end{equation}
converges vaguely as $n$ goes to infinity to a (non trivial) Radon measure $L$ on $\R_+$. For example, if $A_{11}$ is a Bernoulli $0$-$1$ variable with mean $c /n$, then $L$ is equal $c \delta_1$. In this case, (up to the irrelevant diagonal terms) $A$ is the adjacency matrix of an Erd\H{o}s-R\'enyi graph, where each edge is present independently with probability $c/n$.  In this paper, we will focus on L\'evy matrices introduced by Bouchaud and Cizeau \cite{BouchaudCizeau}. They can be defined as follows. We fix $0 < \alpha < 2$ and assume that $X_{11} = n^{1/\alpha} A_{11}$ is a random variable independent of $n$ such that 
\begin{equation}\label{eq:tailX11}
\bP ( |X_{11} |\geq t ) \sim_{t \to \infty} t^{-\alpha}.
\end{equation}
In the above setting, they correspond to the case $L= (\alpha/2) x^{-\alpha/2-1} dx$, where $dx$ is the Lebesgue measure on $\R_+$. For technical simplicity, we will further restrict to $X_{11}$ a symmetric $\alpha$-stable random variable such that for all $t \in \bR$, 
$$
\bE  \exp ( i t X_{11}  ) = \exp ( - \sigma^\alpha |t |^\alpha ),
$$
with $\sigma^{\alpha} = \pi /   ( 2 \sin(\pi \alpha / 2)\Gamma(\alpha)) $. With this choice, the random variable $X_{11}$ is normalized in the sense that \eqref{eq:tailX11} holds.

The spectrum of heavy-tailed matrices is far from being perfectly understood. It differs significantly from classical Wigner matrices. In the L\'evy case \eqref{eq:tailX11}, for any interval $I \subset \R$, the proportion of eigenvalues in $I$ is asymptotically close to 
$\mu_\star(I) $ where the probability measure $\mu_\star$ depends on $0 < \alpha <2$, it is symmetric, has support $\R$, a bounded density $f_\star$ which is analytic outside a finite set of points. Moreover, $f_\star (0)$ has an explicit expression and as $x$ goes to $\pm \infty$, $f_\star ( x) \sim (\alpha / 2) |x|^{-\alpha   -1}$, see \cite{PhysRevE.75.051126,BG08,BDG09,BCC11}.


The eigenvectors of L\'evy matrices have been rigorously studied in \cite{BG}. It is shown there that if $1 < \alpha < 2$, there is a finite set $\cE_\alpha$ such  if $K \subset \R\backslash \cE_\alpha$ is a compact set, for any $\veps > 0$, with high probability,
\begin{equation}\label{eq:delocLe12}
\max \left\{\|u_k \|_\infty : \lambda_k \in K \right\} \leq n^{- \rho + \e },
\end{equation}
where $\rho = (\alpha - 1 )  / ( (2 \alpha) \vee (8 - 3 \alpha) ). $ 
Since $\|u \|_p \leq \|u\|^{1 - 2/p}_{\infty} \|u \|_2 ^{2/p}$, it implies that the $L^p$-norm of most eigenvectors is $O(n^{2\rho/p - \rho + o(1)})$. Notice that when $\alpha \to 2$, then $\rho \to 1/4$ and it does not match with \eqref{eq:delocWi}. It is expected that the upper bound \eqref{eq:delocLe12} is pessimistic.

When $0 < \alpha < 1$, the situation turns out to be very different. In \cite{BouchaudCizeau}, Bouchaud and Cizeau have conjectured the existence of a mobility edge, $E_\alpha >0$, such that all eigenvectors $u_k$ with $|\lambda_k| < E_\alpha - o(1)$ are delocalized in a  sense similar to \eqref{eq:delocLe12} while eigenvectors $u_k$ with $|\lambda_k| > E_\alpha+ o(1)$ are localized, that is they have a sparse representation in the canonical basis. In \cite{BG}, the existence of this localized phase was established when $0 < \alpha < 2/3$. More precisely, for $I$ an interval of $\bR$, as above $\Lambda_I$ is the set of eigenvectors whose eigenvalues are in $I$. Then, if $\Lambda_I$ is not empty, for $ 1 \leq k \leq n$, we set
$$
P_I  (k) =  \frac { 1} {| \Lambda_I | } \sum_{u \in \Lambda_I} \langle u , e_k \rangle ^2,
$$
where $ 
|\Lambda_I |$ is the cardinal of $\Lambda_I$. In words, $P_I(k)$ is the average amplitude of the $k$-th coordinate of eigenvectors in $\Lambda_I$.  Theorem 1.1 in \cite{BG} asserts that $|\Lambda_I|$ is of order $n |I|$ for intervals $I$ of length $|I| \geq n ^{-\rho}$ for some $\rho > 0$ (depending on $\alpha$). By construction, $P_I$ is a probability vector:
$$
\sum_{k =1} ^ n  P_I  (k) =  1.
$$
Observe also that $P_{\R} (k) = 1/n$.  If the eigenvectors in $\Lambda_I$ are localized and $I$ contains few eigenvalues, then we might expect that for  some $k$, $P_I  (k) \gg 1/n$, while for most of the others $P_I(k) = o(1/n)$.  Alternatively, if the eigenvectors in $\Lambda_I$ are well delocalized, then $P_I(k) = O( 1/n)$ for all $k$. More quantitatively, we will measure the (de)localization of eigenvectors through 
\begin{equation}\label{eq:defQI}
Q_I  = n \sum_{k =1} ^ n  P_I  (k)^2\;  \in [1,n].
\end{equation}
The scalar $\log(Q_I)$ is proportional to the  R\'enyi divergence of order $2$ of $P_I$ with respect to the uniform measure $(1/n,\ldots, 1/n)$. If $Q_I \leq C$ then for any $t >0$, the number of $k$ such that $P_I(k) \geq t \sqrt C  / n$ is at most $ n / t^2$. The scalar $Q_I$ is also closely related to the {\em inverse participation ratio} which can be defined as
$$
\Pi_I =  \frac { n } {| \Lambda_I | } \sum_{u \in \Lambda_I} \sum_{k=1}^n\langle u , e_k \rangle ^4 = \frac { n } {| \Lambda_I | } \sum_{u \in \Lambda_I}  \| u \|_4^4  \; \in [1,n].
$$
Using $\sum_{u \in S}  x^4_u  \leq  (\sum_{u \in S}  x^2_u )^2 \leq |S| \sum_{u \in S}  x^4_u$, we find
$$
Q_I \leq \Pi_I \leq Q_I  |\Lambda_I|.
$$

We will write that a sequence of events $E_n$ defined on our underlying probability space holds with overwhelming probability, if for for any $ t > 0$, $n^t \dP ( E_n^c)$ goes to $0$ as $n$ goes to infinity. As we shall check, \cite[Theorem 1.3]{BG} implies the following localization statement. 
\begin{theorem}[Localization of eigenvectors of large eigenvalues \cite{BG}]\label{th:locvect}
Let $0 < \alpha < 2/3$, $0 < \kappa < \alpha /2$ and $\rho = \alpha / (2 + 3 \alpha)$. There exists $E_{\alpha,\kappa}$ such that for any compact $K \subset [-E_{\alpha,\kappa},E_{\alpha,\kappa}]^c$, there is a  constant $c_1 >0$ and, if $I\subset K$ is an interval of  length $ |I| \geq  n^{-\rho}  ( \log  n)^2$,
 then$$
Q_I \geq c_1 |I|^ { -\frac{ 2\kappa}{2 - \alpha}} ,
$$
with overwhelming probability.
\end{theorem}

In this work, we shall prove the converse of this statement and prove notably that there exists a neighborhood of $0$ where eigenvectors are delocalized. 
\begin{theorem}[Delocalization of eigenvectors of small eigenvalues] \label{th:main}
There exists a countable set $\cA \subset ( 0,2)$ with no accumulation point on $(0,2)$ such that the following holds. Let $\alpha \in (0,2) \backslash \cA$ and $\rho' =  \alpha / ( 4 + \alpha) \wedge 1/4$.  There is $E'_{\alpha} > 0$ and  constants $c_0,c_1 >0$ such that, if $I\subset [-E'_\alpha, E'_\alpha]$ is an interval of  length $ |I| \geq  n^{-\rho'}  ( \log  n)^{c_0}$, then
$$
Q_I \leq c_1,
$$
with overwhelming probability.
\end{theorem}

As we shall see in the course of the proof,  $Q_I$ is finite for $I=[E-\eta,E+\eta]$ with $\eta$  going to zero 
iff the fractional moment of the resolvent  $y_z (\alpha/2) $ is bounded below by a positive constant as $n$ goes to infinity. Our point will therefore be to provide such a bound.

The parameter  $Q_I$ could be replaced for any $p > 1$ by 
$$
 n^{p-1} \sum_{k=1}^n P_I (k)^p.
$$
Then, the statements of Theorem \ref{th:locvect} and Theorem \ref{th:main} are essentially unchanged (up to modifying the value of $\rho'$ in Theorem \ref{th:main} and the exponent $\delta >0$ in the lower bound $Q_I \geq c_1|I|^{-\delta}$ in Theorem \ref{th:locvect}). We have chosen to treat the case $p=2$ for its connection with the inverse participation ratio.

There are still many open problems concerning L\'evy  matrices. 
In the forthcoming Corollary \ref{cor:mainloc}, we prove a local law (i.e. a sharp quantitative estimate of $|\Lambda_I|$ for intervals $I$ of vanishing size) which improves for small value of $\alpha$ on Theorem 1.1 in \cite{BG}. We conjecture that such local law holds for all $\alpha \in (0,2)$ and for all intervals $I $ of length much larger that $1/ n$.

There is no rigorous results on the local eigenvalue statistics for L\'evy matrices, see \cite{TBT} for a recent account of the predictions in the physics literature. It is expected that for $1 < \alpha < 2$ the local eigenvalue statistics are similar to those of the Gaussian Orthogonal Ensemble (GOE), asymptotically described by a sine determinantal point process. For $0 < \alpha < 1$ and energies larger than some $E_\alpha$, we expect that the local eigenvalue statistics are asymptotically described by a Poisson point process. In the regime $0 < \alpha < 1$ and energies smaller than $E_\alpha$, \cite{TBT} also predicts a GOE statistics. Finally, \cite{BouchaudCizeau} conjectured the existence of yet another regime when $1 < \alpha < 2$ at energies larger than some $E_\alpha$.

For $0 < \alpha < 1$, proving the existence of such mobility edge $E_\alpha$ is already a very interesting open problem. The core of the difficulty is to better understand a fixed point equation described in its probabilistic form by \eqref{eq:RDE} which is satisfied by the weak limit of $(A- z I)^{-1}_{11}$ as $n$ goes to infinity. More generally, the L\'evy matrix is an example of a broader class of random matrices with heavy tails. The qualitative behavior of the spectrum depends on the Radon measure $L_n$ in \eqref{eq:convLambda} and its vague limit which we denoted by $L$.  It is a challenging question  to understand how $L$ influences the nature of the spectrum around a given energy (regularity of the limiting spectral measure, localization of eigenvectors, local eigenvalue statistics).

The  paper is organized  as follows. In Section \ref{sec:th1}, we prove that Theorem \ref{th:locvect} is a direct consequence of a result in \cite{BG}. Section \ref{sc:th2} gives an outline of the proof of Theorem \ref{th:main}. The actual proof is contained in Section  \ref{sec:FPI} and Section \ref{sec:rates}.

\section{Proof of Theorem \ref{th:locvect}}

\label{sec:th1}

For $0 < \alpha < 2/3$, $0 < \kappa < \alpha /2$ and $\rho = \alpha / (2 + 3 \alpha)$, it follows from \cite[Theorem 1.3]{BG} that there exists $E_{\alpha,\kappa}$ such that for any compact $K \subset [-E_{\alpha,\kappa},E_{\alpha,\kappa}]^c$, there are constants $c_0,c_1 >0$ and for all integers $n \geq 1$, if $I\subset K$ is an interval of  length $ |I| \geq  n^{-\rho}  ( \log  n)^2$,
 then
\begin{equation}\label{eq:thlocvectp}
n^{\alpha/ 2 -1} \sum_{k=1}^n P_I ( k)^{\alpha /2} \leq c_1 |I|^\kappa,
\end{equation}
with overwhelming probability. We may notice that the logarithm of the left hand side in the above expression is proportional to the R\'enyi divergence of order $\alpha /2$ of $P_I$ with respect to the uniform measure. The smaller it is, more localized is $P_I$ (for explicit bounds see \cite{BG}). 

We may use duality to obtain from \eqref{eq:thlocvectp} a lower bound on $Q_I$.  From H\"older inequality, we write for $0< \veps <1$ and $1/p + 1/q = 1$, 
\begin{eqnarray*}
 1 =   \sum_{k =1 }^n  P_I ( k) & = &  \frac 1 n   \sum_{k =1 }^n  ( n P_I ( k ) ) ^{\veps} (n P_I (k)) ^{1-\veps} \\
& \leq & \left(n^{\veps p -1} \sum_{k=1 }^n  P^{\veps p }_I ( k)\right)^{ 1 / p} \left( n^{(1-\veps)q -1} \sum_{k =1 }^n  P^{(1 - \veps)q}_I (k)\right)^{1 / q}.
\end{eqnarray*}
We choose $\veps = \alpha / ( 4 - \alpha)$ and $ p  = 2 - \alpha / 2$. We have $\veps p = \alpha / 2$, $(1- \veps)q = 2$ and $p/q = 1 - \alpha /2$.  Hence, if the event \eqref{eq:thlocvectp} holds, we deduce that 
\begin{equation*}\label{eq:WIloc}
( c_1 |I |^\kappa ) ^{-\frac{q}{p}} = c'_1 |I|^{-\frac{ 2 \kappa}{ 2 - \alpha}}  \leq  Q_I.
\end{equation*}
It completes the proof of Theorem \ref{th:locvect}. 

\section{Outline of proof of Theorem \ref{th:main}}
\label{sc:th2}
\subsection{Connection with the resolvent}

For $z \in \dC_+ = \{ z \in \C : \Im (z)  >0 \}$, the resolvent matrix of $A$ is defined as
$$R(z) = ( A - z  I) ^{-1}.$$ 

The next lemma shows that the quadratic mean of the diagonal coefficients of the resolvent upper bounds $Q_I$. 

\begin{lem}\label{le:IPRres}
Let $I = [ \lambda - \eta, \lambda + \eta]$ and $z  = \lambda + i \eta \in \C_+$. If $|\Lambda_I | \ne 0$, we have
$$
Q_I   \leq \left(  \frac{n |I|}{|\Lambda_I|} \right)^2  \left( \frac 1 n \sum_{k =1} ^n (\Im R_{kk}(z))^2 \right). 
$$
\end{lem}

\begin{proof}
We use the classical bound for $1 \leq k \leq n$,
$$
\sum_{u \in \Lambda_I} \langle u, e_k \rangle ^2 \le  \sum_{j = 1}^n \frac{ 2 \eta^2    \langle u_j, e_k \rangle ^2 }{ \eta^2 + ( \lambda_j - \lambda) ^2  } =  2 \eta \Im R_{kk}(z)
$$
We get,
$$
Q_I  \leq  \frac{ 4 n \eta^2}{|\Lambda_I|^2 }  \sum_{k =1} ^n (\Im R_{kk}(z))^2, 
$$ 
as requested. \end{proof}

Incidentally, we remark from \cite[Lemma 5.9]{BG} and the above proof of Theorem \ref{th:locvect} that there is a  converse lower bound of $Q_I$ involving the average of $(\Im  G_{kk} (z) ) ^{\beta}$ for any $0 < \beta < 1$.

We may now briefly describe the strategy behind Theorem \ref{th:main}. Take $I= [\lambda - \eta , \lambda + \eta] \subset K$ be an interval and $z = \lambda + i \eta$. First, from \cite[Theorem 1.1]{BG}, there exists a constant $c >0$, such that, with overwhelming probability, $|\Lambda_I| \geq c n |I|$.  Thanks to Lemma \ref{le:IPRres}, it is thus sufficient to prove that 
$$
\frac 1 n \sum_{k =1} ^n (\Im R_{kk}(z))^2 = O(1).
$$
From general concentration inequalities, it turns out that the above quantity is self averaging for $\eta \geq n ^{-\rho}$. Using the exchangeability of the coordinates, it remains to prove that
$$
\E (\Im R_{11}(z))^2  = O(1). 
$$
Now, the law of $R_{11}(z)$ converges as $n$ goes to infinity to a limit random variable, say $R_\star(z)$, whose law satisfies a fixed point equation. In subsection \ref{subsec:FPE}, in the spirit of \cite{MR1492789}, we will study this fixed point and prove, by an implicit function theorem,  that $\E (\Im R_\star(z))^2  = O(1)$. It will remain to establish an explicit convergence rate of $R_{11} (z)$ to $R_\star(z)$ to conclude the proof of Theorem \ref{th:main}. A careful choice of the norm for this convergence will be very important. We outline the content of these two sections in the next subsection.

\subsection{The fixed point equations}
\label{subsec:exfpe}
The starting point in our approach is Schur's complement formula, 
\begin{eqnarray}
R_{11} (z)&=&  - \PAR{ z   - n^{-\frac 1 \alpha} X_{11}  +  n^{-\frac 2 \alpha}  \sum_{2 \leq k,\ell \leq n} X_{1k} X_{1\ell} R^ {(1)}_{k\ell}(z)}^{-1} \nonumber\\
& = & - \PAR{ z    + n^{-\frac 2 \alpha}  \sum_{2 \leq k   \leq n} X^2_{1k}  R^ {(1)}_{kk}(z) + T_z }^{-1},  \label{eq:schurcf0}
\end{eqnarray}
where $R^{(1)}$ is the resolvent of the $(n-1)\times (n-1)$ matrix $(X_{k\ell})_{2\le k,\ell\le n}$ and we have set 
\begin{eqnarray*}
T_z & = &  n^{-\frac 1 \alpha}  X_{11} + n^{-\frac 2 \alpha} \sum_{2 \leq k \ne \ell \leq n} X_{1k} X_{1 \ell} R_{k\ell}^{(1)}(z).
\end{eqnarray*}
It turns out that $T_z$ is negligible, at least for $\Im z$ large enough. Assuming that, we observe that the moments of $\Im R_{11}(z)$ are governed by 
the order parameter 
$$y_z =\frac{1}{n}\sum_{k=1}^n (\Im R_{kk}(z))^{\frac{\alpha}{2}}.$$
Indeed, \begin{eqnarray}
\Im R_{11} (z)&\simeq &  -\Im  \PAR{ z    + n^{-\frac 2 \alpha}  \sum_{2 \leq k\le n} X_{1k}^2 R^ {(1)}_{kk}(z)}^{-1} \nonumber\\
& \le & \PAR{  n^{-\frac 2 \alpha}  \sum_{2 \leq k\le n} X_{1k}^2 \Im R^ {(1)}_{kk}(z)}^{-1} . \nonumber
\end{eqnarray}
The resolvent $R$ and $R^{(1)}$ being close, we can justify that $y_z \simeq \frac{1}{n}\sum_{k=2}^n (\Im R_{kk}^{(1)}(z))^{\frac{\alpha}{2}}$. Then, taking moments  and using the formula,  for $p > 0$, $\Re (x) >0$,
\begin{equation}\label{eq:feo}
\frac 1 {x^{p}}=\frac 1 {\Gamma(p)} \int_0^\infty t^{p-1}e^{-x t} dt,
\end{equation}
we deduce that
\begin{eqnarray}
\bE[\left(\Im R_{11}(z)\right)^p]&\le&  \frac{1}{\Gamma(p)}\int_0^\infty t^{p-1} \bE[ e^{- t(\Im R_{11}(z))^{-1}}] dt \nonumber \\
&  \lesssim  & \frac{1}{\Gamma(p)}\int_0^\infty t^{p-1} e^{-\Gamma(1-\alpha/2)
 t^{\alpha/2}  n^{-1} \sum_{k=2}^n (\Im R^{(1)}_{kk}(z))^{\frac{\alpha}{2}}}dt   \nonumber \\
&  \lesssim  & \frac{1}{\Gamma(p)}\int_0^\infty t^{p-1} e^{-\Gamma(1-\alpha/2) 
 t^{\alpha/2} y_z}dt  , \label{bot}\end{eqnarray}
 where, in the second step, we used that the variable $X_{1k}^2$ are in the domain of attraction of the non-negative $\alpha/2$-stable law (this approximation will be made more precise notably thanks to Lemma \ref{cor:formulealice}). The main point becomes to lower bound $y_z$.
 To this end, we shall extend it as a function on $\mathbb C$ and set
 $$
\gamma_z (u) = \Gamma \PAR{1 - \frac \alpha 2} \times \frac 1 n \sum_{k=1}^n \PAR{ -i R_{kk}(z) . u }^{\frac \alpha 2}
$$
where 
$$
h.u = \Re ( u )  h + \Im (u ) \bar h = (  \Re (u) + \Im (u) ) \Re (h) + i ( \Re(u) - \Im(u) ) \Im(h) .
$$
Observe that we wish to lower bound
 $$\gamma_z(e^{i\pi/4})= 2^{\frac \alpha  4}   \Gamma \PAR{1 - \frac \alpha 2} \times \frac 1 {n}\sum_{k=1}^n (\Im(R_{kk}))^{\frac{\alpha}{2}} = 2^{\frac \alpha  4}   \Gamma \PAR{1 - \frac \alpha 2} y_z\,.$$
We shall study the function $\gamma_z$ thanks to a fixed point argument.
We shall use that $\gamma_z$ is homogeneous. Also, we shall restrict ourselves to $u\in\cK_1=\{z\in\mathbb C: \arg(z)\in [-\frac{\pi}{2},\frac{\pi}{2}]\}$, or even in the first quadrant $\cK_1^+=\{z\in\mathbb C: \arg(z)\in [0,\frac{\pi}{2}]\}$.  Here and after, for $z \in \C$, we take the argument $\arg(z)$ in $(-\pi, \pi]$. 
We can see that $\gamma_z$ is approximately solution of a fixed point equation by using \eqref{eq:schurcf0}. To state this result, let us first   define the space $\cH_{\alpha/2,\kappa}$, $\kappa\in [0,1)$, in which we will consider $\gamma_z$.
For any $\beta \in \C$, we let $\cH_\beta$ denote the space of $C^1$ functions $g$ from $ \cK_1^+$ to  $\bC$ such that $g( \lambda u ) = \lambda^{\beta } g(u)$, for all $\lambda \geq 0$. For $\kappa\in [0,1)$, we  endow $\cH_\beta$  with the norms
\begin{eqnarray}
\nonumber \| g  \|_{\infty} & =& \sup_{ u \in S_1^+} | g(u) |,\\
\label{eq:defnorm} \| g  \|_{\kappa} &= & \| g \|_\infty + \sup_{ u \in S_1^+}  \sqrt{ |(i.u)^{\kappa}\partial_1 g (u) |^2 +|(i.u)^{\kappa}\partial_i g (u) |^2} , 
\end{eqnarray}
where $S_1^+=\{u\in  \cK_1^+, |u|=1\}$ and $\partial_\veps g(u) \in \dC$ is the partial derivative of $g$ at $u$ with respect to the real ($\veps = 1$) or imaginary part ($\veps = i$) of $u$. We denote $\cH_{\beta,\kappa}$ the closure of $\cH_\beta$ for $\|.\|_\kappa$. The space $\cH_{\beta,\kappa}$ is a Banach space. Notice also that $\cH_{\beta,0}$ and $\cH_{\beta}$ coincide. The norm $\|.\|_\kappa$ will turn out to be useful with $\kappa>0$ to obtain concentration estimates for $\gamma_z$ as well as to establish existence and good properties for its limit $\gamma^\star_z$ (there $\kappa\ge 0$ is sufficient). 

 We define formally the function  $F$ given  for  $h\in \mathcal K_1$, $ u \in S^1_+$  and $g \in \cH_{\alpha/2}$ by

\begin{align}
 F_h ( g ) (u) =  \int_0 ^ {\frac \pi 2}   \hspace{-5pt} d\theta (\sin 2 \theta  )^{\frac \alpha 2-1}  \int_{0}^\infty  \hspace{-5pt} dy\,  y^{-\frac{\alpha}{2}-1} \int_0^\infty  \hspace{-3pt} dr \, r^{\frac{\alpha}{2}-1} e^{- r h .e^{i \theta }   }\left( e^{-r^{\frac \alpha 2 } g(e^{i\theta})} - e^{- y  r h .  u }e^{-r^{\frac \alpha 2}g(e^{i \theta} + y   u )}\right).\label{defF}
\end{align}
$F_{-iz}$ is  related to a fixed point equation satisfied by $\gamma_z$.  
Namely  let 
\begin{equation}\label{defc}
c_\alpha =   \frac {  \alpha }  { 2^{ \frac \alpha 2 }  \Gamma ( \alpha / 2 ) ^2  } \quad \hbox{ and } \quad \check u = i \bar u =  \Im (u) + i \Re (u). \end{equation}
For $z \in \dC_+$, we introduce the map  $G_z$ on $\cH_{\alpha/2}$ given by
\begin{equation}\label{defG}
G_z(f)  (u)= c_\alpha F_{-iz} (f) (\check u)\,.
\end{equation}
Finally we let 
\begin{align*}
\bar\gamma_z (u) & =\dE \gamma_z(u) = \Gamma \PAR{1 - \frac \alpha 2} \dE \PAR{-i R_{11} (z).u}^{\frac \alpha 2}.
\end{align*}
Then, $ \bar\gamma_z (u) \in \cH_{\alpha/2}$ and we shall prove that

\begin{proposition}\label{prop:afpi} Let $\alpha \in (0,2)$ and $\delta\in (0,\alpha/2)$. There exists $c  > 0$ such that if for  $z \in \dC_+$ and some $\veps >0$, we have $|z| \leq \veps^{-1}$, $\Im z \geq n^{-\delta / \alpha + \veps} $, $\dE ( \Im  R_{11}(z) )^{\alpha/2} \geq \veps $ and $\dE | R_{11} (z)|\leq \veps^{-1}$ then for all $p \geq \alpha/2$ and all $n$ large enough (depending on $\alpha,\veps,p$),
\begin{eqnarray*}
\| \bar \gamma_z - G_z ( \bar \gamma_z) \|_{1-\alpha/2+\delta} & \leq  & (\log n )^{c}   \PAR{    \eta^{-\alpha/2} \bar M_z^{\alpha/2}  +   \eta^{-\alpha/2} n^{-\delta/2}  + \bar M_z^{1 - \alpha/2} \IND_{\alpha >1} } \\
|\dE |R_{11}|^p - r_{p,z} ( \bar \gamma_z) | & \leq & (\log n)^c \PAR{ \eta^{-p} \bar M_z^{\alpha/2} +  \eta^{-\alpha/2}n^{-\delta/2}  } \\
|\dE (-i R_{11})^p - s_{p,z}( \bar \gamma_z(1)) | & \leq & (\log n)^c \PAR{ \eta^{-p} \bar M_z^{\alpha/2} +  \eta^{-\alpha/2}n^{-\delta/2}  },
\end{eqnarray*}
where we have set $\eta = \Im z$, $\bar M_z  = \bE  \Im R^{(1)}_{22}(z) / (n \Im z)$ and,  for $f \in \cH_{\alpha/2}$
$$
r_{p,z} (f) =   \frac{2^{1-p/2}}{\Gamma( p/2)^2}  \int_0 ^ {\frac \pi  2} d\theta  \sin ( 2 \theta) ^{p/2 -1} \int_0 ^ \infty dr r^{p-1}  e^ { ir z. e^{ i \theta}  - r^{\alpha/2} f( e^{ i \theta})  },
$$
and for $x \in \cK_1$, 
$$
s_{p,z} (x) = \frac{1}{\Gamma( p)}  \int_0 ^ \infty dr r^{p-1}  e^ { - i r z   - r^{\alpha/2} x}\,.
$$

\end{proposition}
Using that $R^{(1)}$ is close to $R$, we can upper bound $\bar M_z$ if we can lower bound $\bar \gamma_z$ by \eqref{bot}. Similarly, we can upper bound $\E|R_{11}|^p$ if we can lower bound $\bar \gamma_z$ by \eqref{bot}.
Assuming for a moment that we can obtain such  bounds (by using a bootstrap argument)
the above proposition shows that $\bar \gamma_z$ is approximately a fixed point for $G_z$. 

It turns out that, for any $z \in \C_+$, $\bar \gamma_z$ converges to $\gamma^\star_z \in \cH_{\alpha/2}$ as $n$ goes to infinity, where $\gamma^\star_z$ is a solution of the equation $0 = f - G_z (f)$ (even we cannot prove that there is a unique solution of this equation).  We will check that this last equation has a unique explicit solution of interest for $z = it$, $t \geq 0$ with $\gamma_0^\star$ also  in $  \cH_{\alpha/2}$. We will study  the solutions of the equation $0 = f - G_z(f)$  for $z$ close to $0$ and $f$ close to $\gamma_0^\star$ thanks to the Implicit Function Theorem on Banach space. We  will show that for most $\alpha$ in $(0,2)$, if $|z| $ is small enough, $0 = f -G_z(f)$ has a unique solution in the neighborhood of $\gamma_0^\star$. Moreover, the real part of this solution  is lower bounded by a positive constant. 
Let us summarize these results in the following statement:
\begin{proposition}\label{fixlem}  \label{le:4IFTapp} There exists a countable subset $\cA \subset (0,2)$ with no accumulation point on $(0,2)$ such that the following holds. Let $\kappa \in [0,1)$ and $\alpha\in (0,2)\backslash \mathcal A$. There exists $\tau >0$ such that if $|z| \leq \tau$, then $\gamma^\star_z$ is the unique $f\in  \cH_{\alpha/2,\kappa}$ 
such that $f =G_z(f) $ and $\|f - \gamma_0^\star \|_\kappa \leq \tau$. Moreover, uniformly in $|z| \leq \tau$,  $\gamma_z^\star(e^{i\pi/4})$ is bounded from below and,  for any $p >0$, $r_{p,z} (  \gamma^\star_z) $ is bounded from above. 
\end{proposition}
The possible existence of the set $\mathcal A$  should be purely technical. Our proof requires 
 that $\mathcal A$ contains $\{1/2,1\}$, but it could be larger as our argument is based on the fact that some function, analytic in $\alpha$, does not vanish except possible on a set with no accumulation points.

We will also deduce the following result. 

\begin{proposition}\label{le:fploccont}
Let $\cA$ be as in Proposition \ref{fixlem}, $\alpha \in (0,2)\backslash \cA$, $\kappa \in [0,1)$ and $\gamma^\star_z$ be as in Proposition \ref{fixlem} for $|z|$ small enough. There exist $\tau >0$ and $c > 0$ such that if $|z| \leq \tau$ and $\| \gamma - \gamma_z^\star \|_\kappa \leq \tau$ then 
$$
\| \gamma - \gamma_z^\star \|_\kappa \leq c \, \| \gamma - G_z (\gamma )  \|_\kappa. 
$$
\end{proposition}
As a corollary of the above three propositions, we will prove that $\gamma_z(e^{i\pi/4})$
 is lower bounded for $n$ sufficiently large. In the next section, we study the fixed point equation for $G_z$ and establish Proposition \ref{fixlem}  and Proposition \ref{le:fploccont}. In Section \ref{sec:rates}, we will prove 
Proposition \ref{prop:afpi}  and complete the details of the proof of Theorem \ref{th:main}. 

\section{Analysis of the limiting fixed point equation}
\label{sec:FPI}
\subsection{Analysis of the function $F_h$}
\label{subsec:FPE}
In this first part, we show that the function 
$F_{h} $
is well defined as a   map  from $\cH_{\alpha/2,\kappa}^0=\cup_{\varepsilon>0}\cH^\varepsilon_{\alpha/2,\kappa}$ into $\cH_{\alpha/2,\kappa}$ with
$\cH^\varepsilon_{\alpha/2,\kappa}$ the set of functions in $f \in \cH_{\alpha/2,\kappa}$ such that $\Re f(u) > \veps$ for all  $u$ in $S_1^+$. We also check that $F_h$ has good regularity properties. Notice that $\cH^0_{\alpha/2,\kappa}$ is an open subset of $\cH_{\alpha/2,\kappa}$ for our $\| \cdot \|_\kappa$ norm. We set $\cH_{\alpha/2} ^\veps = \cH_{\alpha/2,0} ^\veps$. Finally, the closure $\bar \cH_{\alpha/2,\kappa}^0$ of $\cH_{\alpha/2,\kappa}^0$ is the set of functions in $\cH_{\alpha/2,\kappa}$ whose real part is non-negative on $\cK_1^+$.

\begin{lem}\label{le:deffp}
Let $h \in  \cK_1$ and $ 0 < \alpha  < 2$.  Let $\kappa\in [0,1)$. $F_h$ defines a map from $\cH^0_{\alpha/2, \kappa}$ to $\cH_{\alpha/2,\kappa}$.  Moreover, if $\Re(h) >0$,  $F_h$ defines a map from $\bar \cH^0_{\alpha/2, \kappa}$ to $\bar \cH^0_{\alpha/2,\kappa}$ and, for some constant $c = c(\alpha)>0$, \begin{equation}\label{le:deffp2}
\| F_h (g)\|_\kappa \leq \frac{c}{\Re (h)^{\alpha/2}} + \frac{c}{\Re (h)^{\alpha}} \| g \|_\kappa \,.\end{equation}
Finally, if $g \in \bar \cH^0_{\alpha/2,\kappa}$,  $| r_{p,ih} (g) | \leq c / \Re(h)^{p}$ and $| s_{p,ih} (g(1)) | \leq c / \Re(h)^{p}$ for some constant $c = c(\alpha,p)$. 
\end{lem}
We shall also prove that $F_h$ is Fr\'echet differentiable and more precisely

\begin{lem}\label{le:Frechetfp}
Let $h \in  \cK_1$, $ 0 < \alpha  < 2$,   $\kappa\in [ 0,1)$, and $g \in \cH^0 _{\alpha/2,\kappa}$. The Fr\'echet derivative of $F_h$ at $g$ is the bounded operator given, for any $f\in \cH_{\alpha/2,\kappa}$, by
\begin{align*}
& DF_h ( g ) (f) (u) =  \int_0 ^ {\frac \pi 2}   \hspace{-5pt} d\theta (\sin 2 \theta  )^{\frac \alpha 2 -1}\, \int_{0}^\infty  \hspace{-5pt} dy\,  y^{-\frac{\alpha}{2}-1}  \\
&\quad \quad \times \int_0^\infty  \hspace{-3pt} dr \, r^{\alpha-1} e^{- r h .e^{i \theta }   }\left( f(e^{i \theta} ) e^{-r^{\frac \alpha 2 } g(e^{i\theta})} -  f(e^{i \theta}+ yu ) e^{- y  r h .  u }e^{-r^{\frac \alpha 2}g(e^{i \theta} + y   u )}\right).
\end{align*}
Moreover, $(h,g)\mapsto F_h(g)$ is continuously differentiable on $\cK_1 \times \cH^0_{\alpha/2,\kappa}$ and $(h,g) \mapsto DF_h(g)$ is continuous in $\cK_1 \times \cH^0_{\alpha/2,\kappa} \to \mathcal B ( \cH_{\alpha/2,\kappa}, \cH_{\alpha/2,\kappa})$. \end{lem}

As a corollary we shall see that all the functions   defined in Proposition \ref{prop:afpi} are Lipschitz in some appropriate norm.

\begin{lem}\label{le:fplip} 
For any $\alpha \in (0,2)$,  $\kappa\in [0,1)$  and $a >0$, $G_z$ is Lipschitz on $\cH^a_{\frac \alpha 2,\kappa}$ : there exists $c >0$ such that for any $z \in  \dC$ and $f,g \in \cH^a_{\frac \alpha 2,\kappa}$,
$$
\| G_z(f)  - G_z(g) \|_\kappa \leq c \, \| f - g  \|_\kappa +(\|f\|_\kappa+\|g\|_\kappa)\|f-g\|_\infty \,. 
$$
Similarly, for any $z \in  \dC$ and $f,g \in \cH^a_{\frac \alpha 2,\kappa}$, any $x,y \in \cK_1$, $\Re(x) \wedge \Re(y) \geq a$, any $p>0$,
$$
| r_{p,z}(f)  - r_{p,z}(g) |  \leq c \, \| f - g  \|_\infty \AND | s_{p,z}(x)  - s_{p,z}(y) |  \leq c \, | x - y  |.
$$
\end{lem}

\begin{proof}[Proof of  Lemmas \ref{le:deffp} and \ref{le:Frechetfp}]
We treat simultaneously the case where  $g\in  \cH_{\alpha/2,\kappa}^a$ for some $a>0$ or 
 $g\in  \bar \cH^0_{\alpha/2,\kappa}$ but $\Re(h) \geq  b >0$. We set for $\zeta \geq 0$ (mainly $\zeta=0,\alpha,\alpha/2,1$) 
and $f\in \cH_{\zeta,\kappa}$,
\begin{align*}
& \varphi_{g,f}^{\zeta,h}(u)  =  \int_0 ^ {\frac \pi 2}   \hspace{-5pt} d\theta (\sin 2 \theta  )^{\frac \alpha 2 -1}\, \int_{0}^\infty  \hspace{-5pt} dy\,  y^{-\frac{\alpha}{2}-1}  \\
&\quad \quad \times \int_0^\infty  \hspace{-3pt} dr \, r^{\frac{\alpha}{2}+\zeta -1} \left( f(e^{i \theta} ) e^{- r h .e^{i \theta }   }e^{-r^{\frac \alpha 2 } g(e^{i\theta})} -  f(e^{i \theta}+ yu ) e^{- r h . (e^{i\theta}+y u) }e^{-r^{\frac \alpha 2}g(e^{i \theta} + y   u )}\right).
\end{align*}
This corresponds to the definition of $F_h$ when $f=1$ and $\zeta=0$, and to its derivative in the direction of $f$ when $\zeta=\alpha/2$. 
We need to check that $\varphi_{g,f}^{\zeta,h}$ has finite norm.  To take into account the singularity of the integration in $y$ at the origin, we cut the integral over $y$ in two pieces: one accounts for integration over $[0,1/2]$ and the other for the integration over $[1/2,+\infty)$.  We let $\varphi_{g,f }^{\zeta,h}(u)=\varphi_{g,f (1)}^{\zeta,h}(0)-\varphi_{g,f (1)}^{\zeta,h}(u)+\varphi_{g,f (2)}^{\zeta,h}(u)$ with 
\begin{eqnarray*}
\varphi_{g,f (1)}^{\zeta, h}(u)  &=&  \int_0 ^ {\frac \pi 2}   \hspace{-5pt} d\theta (\sin 2 \theta  )^{\frac \alpha 2 -1}\, \int_{\frac 1 2}^\infty  \hspace{-5pt} dy\,  y^{-1-\frac{\alpha}{2}}   \int_0^\infty  \hspace{-3pt} dr \, r^{\frac \alpha 2+\zeta-1} e^{- r h .(e^{i \theta }+yu)   }  f(e^{i \theta}+ yu ) e^{-r^{\frac \alpha 2}g(e^{i \theta} + y   u )},\\
\varphi_{g,f (2)}^{\zeta,h}(u)  &=&  \int_0 ^ {\frac \pi 2}   \hspace{-5pt} d\theta (\sin 2 \theta  )^{\frac \alpha 2 -1}\, \int_{0}^{\frac 1 2}  \hspace{-5pt} dy\,  y^{-1 - \frac{\alpha}{2}}  \\
&&\quad  \times \int_0^\infty  \hspace{-3pt} dr \, 
 r^{\frac{\alpha}{2}+\zeta -1} \left( f(e^{i \theta} ) e^{- r h .e^{i \theta }   }e^{-r^{\frac \alpha 2 } g(e^{i\theta})} -  f(e^{i \theta}+ yu ) e^{- r h . (e^{i\theta}+y u) }e^{-r^{\frac \alpha 2}g(e^{i \theta} + y   u )}\right).
\end{eqnarray*}
For the first expression,  we can bound the integral uniformly in $\theta$ by using $\Re(h.u )\ge \Re(h)|u|$ for all $u\in \mathcal{K}_1^+$.  We get a finite constant $C$ such that
 \begin{eqnarray*}
|\varphi_{g,f (1)}^{\zeta,h}(u)|  &\le &C \|f\|_\infty \sup_{\theta\in [0,\frac{\pi}{2}]}
\int_{\frac 1 2}^\infty  \hspace{-5pt} dy\,  y^{-1-\frac{\alpha}{2}} \int_0^\infty  \hspace{-3pt} dr \, r^{\frac{\alpha}{2}-1+\zeta} 
  | e^{i \theta}+ yu |^{\zeta}  e^{- a r^{\frac \alpha 2} | e^{i \theta} + y   u |^{\frac \alpha 2} }e^{-b r |e^{i\theta}+yu|}\\
  &\le&C\|f\|_\infty \sup_{\theta\in [0,\frac{\pi}{2}]}
  \int_{\frac 1 2}^\infty  \hspace{-5pt} dy\,  y^{-1-\frac{\alpha}{2}} | e^{i \theta} + y   u |^{  -\frac \alpha  2} 
\end{eqnarray*}
which are bounded uniformly in $u\in \cK^+_1$ (observe that $| e^{i \theta} + y   u | \geq 1$ for $e^{i \theta} \in S_1 ^+$ and $y \geq 0$). 
The constant $C$ depends on $a \vee b >0$. We used (and will use again repeatidly)  the following straightforward inequality :  for any $h\in \cK_1$, $x \in \cK_1$ and $\beta >0$, 
\begin{equation}\label{eq:gammaHB}
\int_0^\infty \left|  r ^{\beta -1} e ^{- rh} e ^{- r^{\frac \alpha 2} x } \right|d r  \leq\min\PAR{  \Gamma \left( \frac{2 \beta }{ \alpha}\right)  \Re (x) ^{-\frac{2 \beta }{ \alpha}}, \Gamma(\beta) \Re(h)^{-\beta} }\,.
\end{equation} 
Similarly for $ \varphi_{g,f (2)}^{\zeta,h}$, we have to bound, for $0 \leq y \leq 1 /2$ and $0 \leq \theta \leq \pi/2$,
\begin{align}
& L(r,y,\theta):= \left|  e^{- r h .e^{i \theta }   }\left( f(e^{i \theta} ) e^{-r^{\frac \alpha 2 } g(e^{i\theta})} -  f(e^{i \theta}+ yu ) e^{- y  r h .  u }e^{-r^{\frac \alpha 2}g(e^{i \theta} + y   u )}\right)\right| \label{nm} \\
& \leq  \left|  ( f(e^{i \theta} ) -  f(e^{i \theta}+ yu ) e^{-y r h . u}   ) e^{-r^{\frac \alpha 2 }   g(e^{i\theta})} \right| e^{-br} +  |f(e^{i \theta}+ yu ) |\left| e^{-r^{\frac \alpha 2 } g(e^{i\theta})} - e^{-r^{\frac \alpha 2}g(e^{i \theta} + y   u )}\right| e^{-b r}\,. \nonumber 
\end{align}
To bound increments of functions in terms of the $\kappa$-norm, let us use that the linearity of $x \mapsto i.x$ and that if $g \in \cH_{\alpha/2}$ then its derivative is homogeneous of order $\alpha/2 - 1$. We get for  $z,w   \in \cK_1^+$, 
\begin{eqnarray}
|g ( z ) - g (w) | & \leq  &  \| g \|_\kappa | z  - w  | \int_0^1 dt \frac{| z + t (w - z) |^{\kappa + \frac \alpha 2-1} }{|i.(z +t(w- z))|^\kappa} .\nonumber
\end{eqnarray}
A similar bound holds for $f$ with $\zeta$ in place of $\alpha/2$.
Notice that for $z=e^{i\theta}$ and $w=e^{i\theta}+yu$, $z + t (w - z) = e^{i \theta} + t y u$. Using that  $ | e^{-x} - e^{-y} | \leq |x - y | e^{ - \Re (x) \wedge \Re(y) }$ and $1 \leq  |e^{i\theta} +t yu|  \leq 3/2 $ for $t \in [0,1]$ and $y \in [0,1/2]$, we deduce that
 $$L(r,y,\theta)\le C y( \|  f \|_\kappa +\|f\|_\infty r +\|f\|_\infty\|g\|_\kappa r^{\frac{\alpha}{2}}) e^{- a r^{\frac \alpha 2 }-br }\phi(e^{i\theta},e^{i\theta}+yu)\,,$$
 where $C$ is a constant and
 $$
 \phi(z,w) = 1 + \int_0^1\frac{ dt }{|i.(z +t(w- z))|^\kappa}. 
 $$
For $z=e^{i\theta}$ and $w=e^{i\theta}+yu$, 
\begin{equation*}\label{pok}
| i. ( z + t( w - z) ) | = \ABS{ \cos (\theta) - \sin (\theta) + t y (\Re (u) - \Im(u))}. \end{equation*}
Now, the $[0,\pi/2] \to [-1,1]$ function $w(\theta) = \cos ( \theta) - \sin (\theta) $ is  decreasing and $|w'(\theta)|  = \cos \theta  + \sin \theta \in [1,\sqrt 2]$. Since $ | t y (\Re (u) - \Im(u)) | \leq 1/2$, it follows that $\theta \mapsto | i. ( z + t( w - z) ) |$ vanishes once at $\theta_0 \in ( \delta, \pi/2 - \delta)$ for some $\delta >0$.
As a consequence, we find that since $\kappa<1$,
$$\sup_{y\in [0,1/2], u\in S_1^+} \int_0^{\frac \pi 2} d\theta |(\sin 2 \theta  )^{\frac \alpha 2 -1}| \phi(e^{i\theta},e^{i\theta}+yu) | <+\infty\,.$$
Therefore, we can integrate $L(r,y,\theta)$ under $\theta$ and $y$ to find that
\begin{align*}
&\int_0^{\frac{\pi}{2}} d\theta (\sin 2 \theta  )^{\frac \alpha 2 -1}\int_{0}^{1/2} \hspace{-5pt} dy\,  y^{-\frac{\alpha}{2}-1} \int_0^\infty  \hspace{-3pt} dr \, r^{\frac{\alpha}{2}+\zeta-1}
 L(r,y,\theta)
 \\ 
 &\le  C \int_0^{\frac{\pi}{2} }d\theta (\sin 2 \theta  )^{\frac \alpha 2 -1}\int_{0}^{1/2}  \hspace{-5pt} dy\,  y^{-\frac{\alpha}{2}-1} \int_0^\infty  \hspace{-3pt} dr \, r^{\frac{\alpha}{2}+\zeta-1}y \left(
 \|f\|_\kappa +\|f\|_\infty r \|g\|_\kappa\|f\|_\infty r^{\frac{\alpha}{2}}\right)e^{-a r^{\frac \alpha 2}-br }\phi(e^{i\theta},e^{i\theta}+yu)\\
 &\le C  \left( 
 \|f\|_\kappa  
 +\|g\|_\kappa\|f\|_\infty \right),
\end{align*}
where the constant $C$ changes from line to line (and depends of $a \vee b$).
We thus obtain that $\|\varphi^{\zeta,h}_{g,f(k)}  \|_{\infty} < \infty$ for $k \in \{1,2\}$, and collecting all bounds  that $\|\varphi^{\zeta,h}_{g,f}  \|_{\infty} < \infty$ is finite. We now check that $\|(i.u)^\kappa \partial_\veps \varphi^{\zeta,h}_{g,f}  \|_{\infty} < \infty$ for $\veps \in \{1,i\}$.  To this end, notice that
by homogeneity, 
\begin{equation}\label{derhom}
  \partial_\veps  e^{ -r^{\frac \alpha 2}g(e^{i \theta} + y   u )}  =  y r^{\frac \alpha 2} |e^{i \theta} + y u |^{\frac \alpha 2 -1 }   \partial_\veps g (v)
e^{ -r^{\frac \alpha 2}g(e^{i \theta} + y   u )}  \end{equation}
where $v = ( e^{i \theta} + y   u ) / | e^{i \theta} + y   u| \in S_1^+$.
Therefore, we get for $g\in\cH_{\alpha/2,\kappa}$ and $f\in \cH_{\zeta,\kappa}$,
\begin{align*}
&  \left|  \partial_\varepsilon[ f(e^{i \theta}+ yu ) e^{- r h . (e^{i\theta}+yu)  }e^{-r^{\frac \alpha 2}g(e^{i \theta} + y   u )}]\right|  \leq  
 y r |h |  | e^{i \theta}+ yu |^{\zeta} \| f \|_\infty e^{-a r^{\frac \alpha 2}| e^{i \theta}+ yu |^{\frac \alpha 2 } -br|e^{i \theta}+ yu |} \\
 &\quad\quad +
y | e^{i \theta}+ yu |^{\zeta -1+\kappa} \|f \|_\kappa |i.(e^{i\theta}+yu)|^{-\kappa}  e^{-a r^{\frac \alpha 2}| e^{i \theta}+ yu |^{\frac \alpha 2 } -br|e^{i \theta}+ yu |} \\
& \quad \quad +  y r^{\frac \alpha   2}  |i.(e^{i\theta}+yu)|^{-\kappa} | e^{i \theta}+ yu |^{\frac{\alpha}{2}+\zeta -1+\kappa} \| f \|_\infty  \|  g \|_\kappa e^{-a r^{\frac \alpha 2}| e^{i \theta}+ yu |^{\frac \alpha 2 }-br|e^{i \theta}+ yu | }. 
\end{align*}
Using \eqref{eq:gammaHB}, we get for some constant $C$, depending on $(\alpha,a,b)$ 
such that

\begin{align}
& \int_0 ^\infty \hspace{-3pt} dr \, r^{\frac{\alpha}{2}+\zeta-1}  \left | \partial_\veps  [f(e^{i \theta}+ yu ) e^{- y  r h .  u }e^{-r^{\frac \alpha 2}g(e^{i \theta} + y   u )}]\right|  \leq   Cy \bigg ( |h |  | e^{i \theta}+ yu |^{-\frac{\alpha}{2}-1} \| f \|_\infty 
  \nonumber\\
& \quad \quad +    | e^{i \theta}+ yu |^{-\frac \alpha 2 -1+\kappa} \|f \|_\kappa |i.(e^{i\theta}+yu)|^{-\kappa} +     |i.(e^{i\theta}+yu)|^{-\kappa} | e^{i \theta}+ yu |^{-\frac \alpha 2 -1+\kappa} \| f \|_\infty  \|  g \|_\kappa\Bigg) . \label{lkj}
\end{align}
The above expression is integrable on $[0, \infty) \times [0,\pi/2]$ with respect to $y^{-\frac{\alpha}{2}-1}  (\sin 2 \theta  )^{\frac \alpha 2 -1} dy d\theta$ for any $\kappa\in (0,1)$. Indeed, take  $B\ge 0$.
Then, we claim that 
\begin{align}
J_{B,\alpha,\kappa}(u):=&|i.u|^\kappa \int_0^{\frac{\pi}{2} } d\theta (\sin 2 \theta  )^{\frac \alpha 2 -1}  \int_{0}^\infty  \hspace{-5pt} dy\,  y^{-\frac \alpha 2} \frac{ |e^{i \theta} + y u |^{-B}}{|i.(e^{i \theta} + y u)|^\kappa}\label{poi}
\end{align}
is uniformly bounded  if $B+\kappa > 1-\frac{\alpha}{2}$.  To see that, observe that if $i.u=0$ the integral is clearly finite, where as otherwise
$$J_{B,\alpha,\kappa}(u)= \int_0^{\frac{\pi}{2} } d\theta (\sin 2 \theta  )^{\frac \alpha 2 -1}  \int_{0}^\infty  \hspace{-5pt} dy\,  y^{-\frac \alpha 2} \frac{ |e^{i \theta} + y u |^{-B}}{|\frac{i.e^{i \theta}}{i.u} + y |^\kappa}\,.$$
Cut  this integral into two pieces: either $\theta$ is small or close to $\pi/2$, say in $[0,\pi/6]\cup [ \pi / 3,\pi/2]$, or it 
is in $(\pi/6,\pi/3)$, at distance greater than $\pi/6$ from $0$ and $\pi/2$.  
In the first case,  we integrate in $y$ a function which is uniformly bounded at a positive distance of $y_0= - i.e^{i\theta}/i.u$ which is away from the 
origin,
  blows up at most at  $y_0$ where it behaves like $|y-y_0|^{-\kappa}$ and 
behaves like $y^{-B-\kappa-\frac{\alpha}{2}}$ at infinity (independently of $i.u$), and hence  with uniformly bounded integral.
In the second case, we can integrate  first on $\theta$, bounding $(\sin 2\theta)^{-1}$  uniformly from above,
whereas we can use the  bound $|e^{i\theta}+yu|\ge 1 \vee y$ for $\theta\in[0,\frac{\pi}{2}]$, $u\in S_1^+$. 
Then,  observe that 
$$\psi(x)=\int_{\frac \pi 6}^{\frac{\pi}{3}} \frac{d\theta}{ |i.e^{i\theta}+x|^\kappa}$$
 is uniformly bounded, and goes to zero as $ |x|^{-\kappa}$ as 
$x$ goes to infinity. Hence, we bound $\psi$  from above uniformly by a constant times $(1+|x|^\kappa)^{-1}$.
This implies that the second part of the integral is bounded, up to a multiplicative constant, by 
$$
|i.u |^{\kappa} \int_{0}^\infty  \hspace{-5pt} dy\,  y^{-\frac \alpha 2} \frac{  (1 \vee y )^{-B} }{ 1 + |i.u|^\kappa  y^\kappa} = |i.u |^{\kappa+ \frac \alpha 2 -1} 
 \int_{0}^{|i.u|}  \hspace{-5pt} dx\, \frac{ x^{-\frac \alpha 2}}{ 1 + x^\kappa} + |i.u |^{\kappa+ \frac \alpha 2 + B -1} 
\int_{|i.u|}^\infty  \hspace{-5pt} dx\, \frac{ x^{-\frac \alpha 2-B}}{ 1 + x^\kappa},
$$
which is uniformly bounded if $B+\kappa > 1-\frac{\alpha}{2}$ (and even vanishes as $|i.u|$ goes to $0$ : the first term is of order $|i.u|^\kappa$ and the order of the second term  depends on whether $\alpha / 2 + B$ is less, equal or larger than $1$).

Applying \eqref{poi} in \eqref{lkj}  (with $B=1+\alpha/2-\kappa$ or $1+\alpha/2$) shows that $\varphi_{g,f}^{\zeta,h}\in \cH _{\alpha/2,\kappa}$   and in fact
collecting the previous bounds we get
\begin{equation}\label{boundvarphi}\|\varphi^{\zeta,h}_{g,f}\|_\kappa\le C( \|f\|_\kappa+\|f\|_\infty\|g\|_\kappa)\end{equation}
with a finite constant $C$ depending on $a \vee b>0$. 

 In particular $F_h$ is bounded from $\cH^a_{\alpha/2,\kappa}$ into
$\cH_{\alpha/2,\kappa}$ if $a >0$ and $F_h$ is bounded from $\bar \cH^0_{\alpha/2,\kappa}$ into
$\cH_{\alpha/2,\kappa}$ if $\Re(h) \geq b$. In this last case, our proof also shows  that 
$$
\| F_h (g) -F_h(0)\|_\kappa \leq C_{\Re(h)} \| g \|_\kappa. 
$$
for some constant $C_{\Re(h)}$ depending on $\Re(h)$. However, from the homogeneity relation, for $t >0$,  
$$
F_h ( t^{\alpha/2} g)  = t^{-\alpha/2} F_{ h / t} (g). 
$$
we get, with $t = \Re(h)$, 
\begin{eqnarray*}
\| F_h (g)\|_\kappa &=& t^{-\alpha/2}\| F_{ h / t} (t^{-\alpha/2}g)\|_\kappa \\
&\le& t^{-\alpha/2}\| F_{ h / t} (0)\|_\kappa +t^{-\alpha/2}\| F_{ h / t} (t^{-\alpha/2}g)-F_{h/t}(0)\|_\kappa\le Ct^{-\alpha/2} + 
  t^{-\alpha} C_1 \| g\|_\kappa
\end{eqnarray*}
where we noticed that 
 $\|F_h(0)\|_\kappa$ is uniformly bounded when $\Re(h)=1$.  This  completes the proof of Lemma \ref{le:deffp}.
The fact that $D F_h (g) (f) (u) = \varphi_{g,f}^{\alpha,h}(u)$ follows easily since $g \in \cH^a_{\alpha/2,\kappa}$ implies that $g + f \in \cH^{a'}_{\alpha/2,\kappa}$ with $a' = a - \|f \|_{\infty}$.  
Finally, from \eqref{boundvarphi}, $DF_h(g) $ is a bounded operator.

We  now check the continuity of $DF_h(g)$ in $h$ and $g$. It is sufficient to prove that for any $a >0$, there exists a constant $c = c(\alpha, a)$ such that for all $f \in \cH_{\alpha/2,\kappa}$, $g,g' \in \cH^a_{\alpha/2,\kappa}$, $h,h' \in \cK_1$,
\begin{equation}\label{eq:contDF}
\| D F_h (g) (f) -  D F_{h'} (g') (f) \|_\kappa  \leq c \| f \|_\kappa ( |h - h '| + \| g - g '\|_\kappa ).  
\end{equation}
To this end, we prove the same bound for $\varphi^{\alpha/2,h}_{g,f }$  instead of $DF_h(g)$.
By interpolation, we may write that 
$$\varphi^{\alpha/2,h}_{g,f (2)}-\varphi^{\alpha/2,h'}_{g',f(2)}=-\int_0^{1} ds [\varphi^{\alpha/2+1,h_s}_{g,(h-h').* f }
+\varphi^{\alpha, h'}_{g_s,f(g-g')}]$$
with $h_s=s h+(1-s)h', g_s=sg+(1-s)g'$ and $ [ (h-h').*f ](u)=(h-h').u f(u)$ is in $\cH_{\alpha/2+1}$.
 The bound then again follows from \eqref{boundvarphi}. 
It completes  the proof of \eqref{eq:contDF}. In particular, we see that $h\to DF_h(g)(f)$ is Lipschitz. Similarly, we can compute the partial derivative in $h$ of $(h,g) \mapsto F_h(g)$ with respect to the real ($\veps = 1$) or imaginary part ($\veps = i$) of $h$. We find
$$\partial_{\veps} F_h(g)= \varphi^{1,h}_{g,\veps.*1}$$
and see that it has finite $\|.\|_\kappa$ norm and it is continuous in $g$.

Finally, the last statement of Lemma \ref{le:deffp} is an immediate consequence of inequality \eqref{eq:gammaHB}. 
 \end{proof}

\begin{proof}[Proof of Lemma \ref{le:fplip}]
From $$G_z(f)-G_z(g)=\int_0^1d\theta  DG_z(\theta f+(1-\theta )g)(f-g) $$
we deduce from \eqref{boundvarphi} that
$$\|G_z(f)-G_z(g)\|_\kappa\le c \,   \| f - g  \|_\kappa + c \, (\|f\|_\kappa+\|g\|_\kappa)  \| f - g  \|_\infty. 
$$ 
 The proof of the second statement is straightforward. \end{proof}

\subsection{Implicit function theorem : proofs of Proposition \ref{fixlem} and Proposition \ref{le:fploccont}}
In this subsection, we analyze the solutions  for the fixed point equation governed by $G_z$ for $z$ small. 
 It is proved in \cite{BCC} that, for each integer $k$, as $n$ goes to infinity, the analytic function $z \mapsto R_{kk}(z)$ from $\dC_+$ to $\dC_+$ converges weakly for the finite dimensional convergence to the random analytic function $z \mapsto R_\star(z)$ from $\dC_+$ to $\dC_+$ which is the unique solution of the recursive distributional equation for all $z \in \bC_+$,
\begin{equation}\label{eq:RDE}
R_\star(z) \stackrel{d}{=}  - \left(z + \sum_{k \geq 1} \xi_k R_k (z)\right)^{-1},
\end{equation}
where $\{\xi_k\}_{k \geq 1}$ is a Poisson process on $\bR_+$ of intensity measure $\frac{\alpha}{2} x^{\frac \alpha  2 - 1} dx$, independent of $(R_k)_{k\geq 1}$, a sequence of independent copies of $R_\star$. In \cite{BCC}, $R_\star(z)$ is shown to be the resolvent at a vector of a random self-adjoint operator defined associated to Aldous' Poisson Weighted Infinite Tree. Its order parameter is given by
$$\gamma_z^\star(u):= \Gamma(1-\frac{\alpha}{2})\mathbb E[(-i R_\star(z).u)^{\frac{\alpha}{2}}] \in \bar \cH^0_{\alpha/2} \,.$$
Integrating \eqref{eq:RDE}, it turns out that $\gamma^\star_z$ satisfies the following fixed point equation.

\begin{lem}\label{le:fpgamma}
Let $z \in  \bC_+$ and $ 0 < \alpha  < 2$. Then for all $u \in \cK_1^+$, $$\gamma_z^\star (u )  = G_z(\gamma_z^\star)(u)= c_\alpha F_{-iz} ( \gamma_z^\star )  (\check u ),$$
where $c_\alpha$ and $\check u$ are  defined in \eqref{defc}. Moreover, for any $p >0$, $\E | R_\star (z)|^p = r_{p,z} (\gamma_z^\star)$ and $\dE (-i R_{\star})^p=  s_{p,z}(\gamma^\star_z(1))$. 
\end{lem}
\begin{proof}
The statements are a consequence of the L\'evy-Khintchine formula applied to the Poisson process $(\xi_k)$, see e.g. \cite[(4.5)]{BCC}: if $(w_k)$ are independent of $(\xi_k)$, i.i.d. complex random variables with $\Re(w_k) > 0$ we  have \begin{equation}\label{eq:LK}
\E \exp \PAR{ - \sum_{k} \xi_k w_k } = \exp\PAR{ -  \Gamma \PAR{ 1 - \frac \alpha 2}\E w_1^{\alpha/2} }.
\end{equation}
The first statement of the lemma is \cite[Lemma 5.3]{BG}. The second statement is a straightforward modification of the forthcoming Corollary \ref{cor:stg} and \eqref{eq:LK}. To avoid repetition we skip the details of the proof.  
\end{proof}

We know that for $z = i \eta$, $\eta > 0$, $R_\star(i \eta)$ is pure imaginary (see \cite{BCC}). It follows that $\gamma^\star_{i \eta} ( u ) = (1.u ) ^{\frac \alpha 2} a_{\eta}$, where $a_\eta$ satisfies an explicit fixed point equation. We know that $a_\eta \to a_0$ as $\eta \to 0$ and, by \cite[Lemma 4.3]{BCC11}, 
$$
a_0 =\left( \frac{ \Gamma \left ( 1- \frac \alpha 2 \right)}{ \Gamma \left ( 1+ \frac \alpha 2 \right)} \right)^{1/2}.
$$
In particular, $\gamma^\star_0 := (1.u ) ^{\frac \alpha 2} a_{0}$ is in $\cH^0_{\alpha/2}$, it is the limit of $\gamma^\star_{i\eta}$ as $\eta \to 0$ and, by Lemma \ref{le:Frechetfp}, it satisfies the fixed point equation $\gamma_0^\star  = G_0(\gamma_0^\star)$.  We are interested in $D G_0 ( \gamma^\star_0)$. First, by Lemma \ref{le:Frechetfp}, 
\begin{eqnarray}
 DF_0 ( \gamma^\star_0 ) (f) (u)  &=&  \int_0 ^ {\frac \pi 2}   \hspace{-5pt} d\theta (\sin 2 \theta  )^{\frac \alpha 2 -1}  \int_{0}^\infty  \hspace{-5pt} dy\,  y^{-\frac{\alpha}{2}-1} \,  \nonumber \\
&&\quad \quad \times\int_0^\infty  \hspace{-3pt} dr \, r^{\alpha-1} \left( f(e^{i \theta} ) e^{-r^{\frac \alpha 2 } a_0 ( 1 .e^{i\theta})^{\frac \alpha 2}} -  f(e^{i \theta}+ yu ) e^{-r^{\frac \alpha 2}a_0 (1 . ( e^{i \theta} + y   u ) )^{\frac \alpha 2}  }\right) \nonumber \\
&  =&  \frac {2}{\alpha a_0^2} \int_0 ^ {\frac \pi 2}   \hspace{-5pt} d\theta (\sin 2 \theta  )^{\frac \alpha 2 -1}  \int_{0}^\infty  \hspace{-5pt} dy\,  y^{-\frac{\alpha}{2}-1} \left( \frac{ f(e^{i \theta} )}{ ( 1 . e^{i \theta} ) ^\alpha}  -   \frac{ f(e^{i \theta}+ yu ) }{( 1 . ( e^{i \theta} + y u )  ) ^\alpha}  \right). \label{eq:DF0}
\end{eqnarray}

Observe that, for $z  \in \cK^+_1$, $(1.  z ) = \Re(z) + \Im(z) $ and $|z| \leq ( 1. z) \leq \sqrt 2 |z|$.  Define the unitary operator, $J(f)(u) = f(\check u)$.  and for $\alpha \in (0,2)$, we consider the operator on $\cH_{\alpha/2,\kappa}$ given by
$$
K_\alpha =  - c_\alpha  DF_0 (\gamma^\star_0) J,
$$
where, with an abuse of notation, $DF_0 (\gamma^\star_0)$ is the operator defined on the right hand side of \eqref{eq:DF0}.  Then $-K_\alpha$ is precisely equal to $D G_0 ( \gamma^\star_0)$. 
Our goal is to apply the implicit function theorem to $I + K_\alpha$. The main result of this section is the following result.

\begin{theorem}\label{th:4IFT}
Let $\kappa\in [0,1)$. Let $\cA$ be the subset of $\alpha$ in $(0,2)$ such that $I + K_\alpha$ is not an isomorphism of $\cH_{\alpha/2,\kappa}$.
 If $F$ is a closed subset of $(0,2) \backslash\{1/2,1\}$, then $\cA \cap F$ is finite.   
\end{theorem}

The proof of Theorem \ref{th:4IFT} will require a careful analysis of the operator $K_\alpha$. We postpone it to the next subsection. We first use it to prove Proposition \ref{fixlem} and Proposition \ref{le:fploccont}.

\begin{proof}[Proof of Proposition \ref{fixlem}.]
Since $z \mapsto R_\star (z)$ is analtyic and $|R_\star(z)| \leq 1/ \Im (z)$, it is easy to check that $z \mapsto \gamma^\star_z$ is continuous from $\dC_+$ to $\cH_{\alpha,\kappa}$ (see fortcoming Lemma \ref{le:gammaholder}). Also, as already pointed, $\gamma^\star_{z}$ converges to  $\gamma^\star_{0}$ as $z$ goes to $0$. Hence, from Lemma \ref{le:Frechetfp}, the definition of $\cA$, and  the Implicit Function Theorem in  Banach spaces (see \cite[Theorem 2.7.2]{MR1850453}), for any $t >0$, there exists $\tau > 0 $ such that for all $z \in \dC_+$, $|z| \leq \tau$, 
$$
\| \gamma^\star_z - \gamma^\star_0 \|_\kappa \leq t.
$$
We take $t = 2^{\alpha/4} a_0 /2$. We deduce that if $|z| \leq \tau$, 
$$
\Gamma( 1 - \alpha /2 ) 2^{\frac \alpha 4} \E (\Im  R_\star (z) )^{\frac \alpha 2}  = \gamma^\star_z  (e^{i \pi /4} ) \geq \gamma^\star_0 ( e^{i \pi/4} ) - t = t.
$$
Hence, with $c = a_0/2$, if $|z| \leq c_0$, we have 
$$
\Gamma( 1 - \alpha /2 )  \E (\Im  R_\star (z) )^{\frac \alpha 2} \geq c. 
$$
Now, from \eqref{eq:RDE}, we have the bound,
$
| R_\star(z) |\leq  1/ \left( \sum_{k \geq 1} \xi_k \Im R_k(z)\right). 
$
Using Lemma \ref{le:fpgamma} and the formula \eqref{eq:feo}, 
we get 
\begin{eqnarray*}
r_{p,z} (\gamma_z^\star) = \E |R_\star(z)|^ p& \leq & \frac 1 {\Gamma(p)} \int x^{p -1} \E e^{ - x \sum_{k} \xi_k \Im (R_k)} dx \\
&  = & \frac {1}{\Gamma (p)} \int x^{p -1}  e^{ - x^{\alpha/2} \Gamma ( 1 - \alpha/2) \E \Im (R_\star)^{\alpha/2} } dx \\
&  \leq & \frac {1}{\Gamma (p)} \int x^{p -1}  e^{ - c x^{\alpha/2}} dx,
\end{eqnarray*}
where at the second line we have used \eqref{eq:LK}.  \end{proof}

\begin{proof}[Proof of Proposition \ref{le:fploccont}.]
From Theorem \ref{th:4IFT}, Lemma \ref{le:deffp}, Lemma \ref{le:Frechetfp}, we may apply the inverse function Theorem to $T : (z,f) \mapsto (T_1(z),T_2(z))=(z, f - G_z (f))$ on $\dC \times \cH^0_{\alpha/2,\kappa}$. It follows that there exists an open neighborhood $U$ of $ (0,\gamma_0^\star)$, such that $T$ is an homeomorphism from $U$ to $V= T (U)$. Moreover, $T_{|U}^{-1}$ has Fr\'echet derivative at $y = (z,g) \in V$ given by $(D T ( T_{|U} ^{-1} y ))^{-1}$. If $\tau$ is taken small enough, $(z,\gamma_z^\star) \in U$ for all $|z| \leq \tau$. Also, since $\gamma_z^\star = G_z (\gamma^\star_z)$, if $\tau$ is taken small enough, we may further assume that $\{z\} \times B( \gamma_z^\star, 2 \tau)\subset U$ and $\{z \} \times B(0, 2\tau) \subset V$, where $B(f,\tau)$ is the open ball with the norm $\| \cdot \|_\kappa$. We then apply the fundamental theorem of calculus to $t \mapsto T_{|U}^{-1} ( z, f_t )$ where $f_t = t (\gamma - G_z(\gamma))$ to deduce
$$\|\gamma_z^\star-\gamma\|_\kappa=\| \PAR{ T^{-1}(z,0)-T^{-1}(z,\gamma -G_z(\gamma)}_2\|_\kappa\le c\|\gamma -G_z(\gamma)\|_\kappa\,.$$
 It implies the Proposition with $c = \sup \| (D T (x) ) ^{-1} \| $, where $\| \cdot \|$ is the operator norm and the supremum is over all $x = (z,f)$, $|z| \leq \tau$,  $f \in B( \gamma_z^\star, \tau)$. 
\end{proof}

\subsection{Proof of Theorem \ref{th:4IFT}}
The strategy of the proof of Theorem \ref{th:4IFT} is quite intricate. The road map is the following:
\begin{enumerate}
\item   Prove that the operator $K_\alpha$ is compact on the Banach  space of continuous functions for the $\|.\|_\kappa$ norm. Hence, $I+K_\alpha$ 
is not an isomorphism iff $-1$ is an eigenvalue of $K_\alpha$.
\item We could not study directly the spectrum of $K_\alpha$ but that of another compact operator  $H_\alpha$, an operator on a larger Hilbert space whose spectrum contains
the spectrum of $K_\alpha$. We prove that $H^m_\alpha$ is trace class for $m$ large enough (depending on $\alpha$). 
\item
To show that $-1$ is not an eigenvalue of $H_\alpha$,  we want toshow that for some large even number $m$, $\det(I-H_\alpha^{m})$ does not vanish
except  possibly on a discrete set of $\alpha$, where $\det(I + \cdot)$ denotes the Fredholm determinant. To this end, we first show that $\alpha\mapsto \det(I-H_\alpha^m)$ is analytic on $\Re(\alpha) \in (0,2)$. 

\item We finally check that $\alpha\mapsto \det(I-H_\alpha^m)$ does not vanish at infinity and concludes the whole argument. 

\end{enumerate}

\paragraph{Step 1 : Compactness. }
We start by the compactness of $K_\alpha=-c_\alpha DF_0(\gamma^\star_0) J$. 

\begin{lem}\label{le:compactness}
For any $\kappa\in (0,1)$,  any $\alpha \in U=\{z\in\mathbb C: |\Re(\alpha)\in (0,2)\}$, the operator $K_\alpha$ is compact in $\cH_{\alpha/2,\kappa}$. 
\end{lem}

\begin{proof}
We introduce the multiplication operator, $M (f)  (u) = f (u)  ( 1 . u  ) ^{-\alpha}$. We notice that $M$ is a bounded operator from $\cH_{\alpha/2,\kappa}$ to $\cH_{-\alpha/2,\kappa}$. We will now prove that the $\cH_{-\alpha/2,\kappa}$ to $\cH_{\alpha/2,\kappa}$ operator 
$$T f (u) = \int_0 ^ {\frac \pi 2}   \hspace{-5pt} d\theta (\sin 2 \theta  )^{\frac \alpha 2 -1}  \int_{0}^\infty  \hspace{-5pt} dy\,  y^{-\frac{\alpha}{2}-1} \left( f(e^{i \theta} )   -  f(e^{i \theta}+ yu )  \right)$$
is  compact. It will conclude the proof since $DF_0 (\gamma^\star_0) =  \left( 2 / (\alpha a_0^2) \right)T M$. Let $B$ be the set of $f \in \cH_{-\alpha/2,\kappa}$ such that $\|f \|_\kappa \leq 1$. We should prove that $T B$ is a compact set of $\cH_{\alpha/2,\kappa}$.  
Note that $M$ is bounded, invertible and  with bounded  inverse as $1.u$ is bounded above and below by a positive number. 
By Lemma \ref{le:Frechetfp}, $T= \left( 2 / (\alpha a_0^2) \right)^{-1} DF_0(\gamma^\star_0) M^{-1}$ 
  is a bounded operator  in $\mathcal H_{\alpha/2,\kappa}$. If $f \in B$ then arguments similar to the proof of Lemma \ref{le:deffp} easily show that $|T f(1) |\leq C \| f \|_\kappa$ is bounded. It is sufficient to prove that there exists a compact set $K$ of $C_\kappa(S_1^+)$ (the space of  continuous functions on $S_1^+$  equipped with  the norm $\| |i.u|^\kappa \cdot \|_\infty$) such that $u \mapsto \partial_\veps (T f)(u) $  is in $K$ for all $f \in B$ and $\veps \in \{1,i\}$.

To this end, we observe that, if $g (u)= \partial_\veps f (u)$, then $\partial_\veps (T f)(u) = - P g  (u)$ where 
\begin{equation}\label{eq:defP1}
P g (u) =   \int_0 ^ {\frac \pi 2}   \hspace{-5pt} d\theta (\sin 2 \theta  )^{\frac \alpha 2 -1}  \int_{0}^\infty  \hspace{-5pt} dy\,  y^{-\frac{\alpha}{2}}   g\left(\frac{ e^ { i \theta} + yu }{| e^ { i \theta} + yu|} \right)  | e^ { i \theta} + yu|^{-1  - \frac \alpha 2} . 
\end{equation}
We should  thus prove that $P$ is a compact operator on $C_\kappa(S_1^+)$.  
This amounts to prove that $P_\kappa(g)(u)= (i.u)^\kappa P((i.*)^{-\kappa} g)(u)$ is compact in $C_0(S_1^+) = C(S_1^+)$.
To this end, we write $P$ as a kernel operator. We assume that $u = e^{i \omega}$. We wish to perform the change of variable from $(\theta,y)$ to $(\theta, \psi)$ defined, for $\theta \ne \omega$, by
\begin{equation*}\label{eq:changevar}
e^{i \psi} = \frac{ e^ { i \theta} + yu }{| e^ { i \theta} + yu|}. 
\end{equation*}
The set $y \in (0,\infty)$ is mapped bijectively to  $\psi \in (\theta, \omega)$ if $\theta < \omega$ or $\psi \in (\omega, \theta)$ if $\theta > \omega$.  Since $| e^{i \theta}  + yu | = e^{ i (\theta - \psi) } + y e^{i (\omega - \psi)}$ is a real number, taking the imaginary part, we arrive at  
$$
y = \frac{ \sin( \theta - \psi) }{\sin ( \psi - \omega)}. 
$$
Taking the real part, we find
$$| e^{i \theta}  + yu | = \cos ( \theta - \psi) + \frac{ \sin( \theta - \psi) \cos ( \omega - \psi) }{\sin ( \psi - \omega)} = \frac{ \sin( \theta -  \omega) }{\sin ( \psi - \omega)}.$$
Also,
\begin{eqnarray*}
\frac{ d y }{d \psi } & =& -  \frac{ \cos( \theta - \psi)\sin( \psi - \omega) + \cos(\psi- \omega) \sin (\theta - \psi) }{\sin^2( \psi - \omega)} =- \frac{\sin( \theta - \omega)}{\sin^2 (\psi - \omega)}.
\end{eqnarray*}
 We may then express $P$ as 
\begin{equation}\label{eq:Pkernel}
P g (e^{i \omega} ) =   \int_0 ^ {\frac \pi 2}   \hspace{-5pt} d\psi   g(e^{i \psi} ) k (\omega, \psi),
\end{equation}
where $k$ is a kernel on $[0,\pi/2]^2$ defined by, for $\psi > \omega$,
$$
k(\omega, \psi)  =  \sin ( \psi - \omega) ^{ \alpha -1}\int_{\psi}^{\frac \pi 2} d \theta   (\sin 2 \theta  )^{\frac \alpha 2 -1}  \sin( \theta - \psi)^{-\frac \alpha 2}  \sin ( \theta - \omega)^{ - \frac \alpha  2}, 
$$
while if $\psi < \omega$, 
$$
k(\omega, \psi)  =  \sin (  \omega - \psi ) ^{ \alpha -1}\int_{0}^{\psi} d \theta   (\sin 2 \theta  )^{\frac \alpha 2 -1}   \sin( \theta - \psi)^{-\frac \alpha 2}  \sin ( \theta - \omega)^{ - \frac \alpha  2}.
$$
This implies that $P_\kappa$ has kernel
$$k_\kappa(\omega,\psi)= (i.e^{i\omega})^\kappa (i.e^{i\psi})^{-\kappa} k(\omega,\psi)\,.$$
With our alternative expression for $P_\kappa$, we can readily prove the compactness on $C(S_1^+)$. From \cite[Theorem 12.1]{MR647629}, it suffices to prove that for any $\omega \in [0,\pi/2]$, if $\omega_n \to \omega$, 
\begin{equation}\label{eq:compactQ}
\int_0^{\pi/2} | k_\kappa ( \omega,\psi) - k_\kappa ( \omega_n,\psi) | d \psi \to 0. 
\end{equation}
If $\alpha \in (0,2)$ is real, then $k_\kappa (\omega, \psi) \geq 0$ and we have that
$$
\int_0^{\pi/2}  | k_\kappa (\omega, \psi) | d \psi = |i.e^{i\omega}|^\kappa \int_0 ^ {\frac \pi 2}   \hspace{-5pt} d\theta (\sin 2 \theta  )^{\frac \alpha 2 -1}  \int_{0}^\infty  \hspace{-5pt} dy\,  y^{-\frac{\alpha}{2}}\frac{ | e^{i \theta} + y e^{i \omega}|^{-1-\frac \alpha 2+\kappa}}{|i.(e^{i \theta} + y e^{i \omega})|^\kappa}
$$
is continuous in $\omega$ by dominated convergence (recall $\kappa<1$ and $|e^{i\theta}+y e^{i\omega}|$ is bounded below). Consequently, from Scheff\'e's Lemma, to prove \eqref{eq:compactQ}, it is sufficient to check that for almost all $\psi \in [0,\pi/2]$, $k ( \omega_n,\psi)$ converges to $k ( \omega,\psi)$. It is however immediate that for any $\psi \ne \omega$, the above convergence holds. It completes the proof in the case $\alpha$ real. In the general case, writing explicitly the dependence of $k$ in $\alpha$, we find $| k^\alpha (\omega, \psi) |  \leq k^{\Re(\alpha)} (\omega,\psi)$.  We conclude similarly by dominated convergence that $\int_0^{\pi/2}  | k^\alpha_\kappa (\omega_n, \psi) | d \psi \to \int_0^{\pi/2}  | k^\alpha _\kappa(\omega, \psi) | d \psi$.  Almost sure convergence of $k^\alpha(\omega_n,\psi)$ is again clear from the formula.
\end{proof}

\paragraph{Step 2 : Trace class for an affiliated operator.  }
The next step would be to compute the spectrum of $K_\alpha$. Unfortunately, we have not been able to compute it explicitly. We  instead first define an operator $H_\alpha$ whose spectrum contains the spectrum of $K_\alpha$. For a suitable even power $m$, we will then prove that the Fredholm determinant of $I - H^m_\alpha$  depends analytically on $\alpha$. To this end, we set $X = S_1^+ \times \{0,1,i\}$ and consider the Hilbert spaces $L_\kappa^2 (S_1^+)$ and $L_\kappa ^2 (X)$ with respective norms
$$
\| g \|_ {L^2_\kappa (S_1^+)} = \sqrt{  \int_0^{\frac \pi 2 } |g( e^{i \psi})|^2  \frac{d \psi}{|i.e^{i\psi}|^\kappa} }  \quad \hbox{ and } \quad \| f \|_ {L^2_\kappa (X)} = \sqrt{ \|f_0\|_{L^2 _0 (S_1^+)}^2 +  \sum_{\veps \in \{1,i\} } \| f_\veps \|^2_{L^2_\kappa (S_1^+)} },
$$
where, for shorter notation, we have set $f_\veps : u \mapsto f(u,\veps)$. 
   We extend a function $f$ in $L_\kappa^2 (X)$ to a function on $ \cK^+_1 \times \{0,1,i\}$, by setting $f ( \lambda u , 0) = \lambda^{\alpha/2} f(u ,0)$, $f( \lambda u , \veps) = \lambda^{\alpha/2 - 1} f( u, \veps)$ for $\veps \in \{1,i\}$.  We define a new operator $S:C_\kappa(X)\to C_\kappa(X)$ by the formula, for $u = u_1 + i u_2$ and $\veps \in \{1,i\}$,  
\begin{eqnarray*}
S f (u , 0)  & = & \frac 2 \alpha  \int_0 ^ {\frac \pi 2}   \hspace{-5pt} d\theta (\sin 2 \theta  )^{\frac \alpha 2 -1}  \int_{0}^\infty \hspace{-5pt} dy \, y^{-\frac{\alpha}{2}} \left( \frac{u_1 f_1(e^{i \theta}+ yu ) + u_2 f_i (e^{i \theta}+ yu ) }{( 1 . ( e^{i \theta} + y u )  ) ^\alpha}   - \alpha \frac{ (1.u) f_0 (e^{i \theta}+ yu ) }{( 1 . ( e^{i \theta} + y u )  ) ^{1+\alpha}}  \right) \\
S f (u , \veps ) & =&  \int_0 ^ {\frac \pi 2}   \hspace{-5pt} d\theta (\sin 2 \theta  )^{\frac \alpha 2 -1}  \int_{0}^\infty  \hspace{-5pt} dy\,  y^{-\frac{\alpha}{2}} \left(   \frac{ f_\veps (e^{i \theta}+ yu ) }{( 1 . ( e^{i \theta} + y u )  ) ^\alpha} - \alpha   \frac{(1.u) f_0 (e^{i \theta}+ yu ) }{( 1 . ( e^{i \theta} + y u )  ) ^{1+\alpha}} \right).
\end{eqnarray*}
Let $P$ be defined by \eqref{eq:defP1}, we observe that in matrix form 
\begin{equation}\label{eq:S2P}
 S = \begin{pmatrix} M_{00} & M_{01} & M_{0i} \\
M_{10} & I & 0 \\
M_{i0} & 0 & I
\end{pmatrix}  \begin{pmatrix}  P & 0 & 0 \\
0 &   P & 0 \\
0 & 0 &  P 
\end{pmatrix} \begin{pmatrix}
N_0 & 0 & 0 \\
0 & N_1 & 0 \\
0 & 0 & N_i
\end{pmatrix}, 
\end{equation}
where $M_{\veps \veps'}$ and $N_\veps$ are the bounded multiplication operators in $L^2 (S_1^+)$, $M_{00} f (u) = - 2 (1.u) f(u)$, $M_{10} f (u) = M_{i0} f (u) = - \alpha (1.u) f(u)$, $M_{01} f (u) = \frac 2 \alpha  u_1 f(u)$,  $M_{0i} f (u) = \frac 2  \alpha u_2 f(u)$, $N_0 f (u) = (1.u)^{-\alpha-1} f(u)$ and $N_1 f (u)  =N_i f (u) =(1.u)^{-\alpha}f(u) $.  
The proof of Lemma \ref{le:compactness} shows that $S$ is a compact operator on $C_\kappa(X)$ since $P$ is compact and the $M_{\veps \veps'}$ and $N_\veps$ are bounded. 
We set 
\begin{equation}\label{eq:defHalpha}
H_\alpha = c'_\alpha S J , 
\end{equation}
where $c'_\alpha = c_\alpha  ( 2 / (\alpha a_0^2) ) $ and the operator $J$ is extended on $L_\kappa^2 (X)$ by setting $(J f )_\veps (u)  = f_\veps ( \check u)$.   If $f \in \cH_{\alpha/2,\kappa}$, we set $\partial_0 f (u)= f$  and $\bar f (u,\veps) = \partial_\veps f (u) \in L_\kappa^2 (X)$. We shall use the identity
\begin{align*}
   & \int_{0}^\infty  \hspace{-5pt} dy\,  y^{-\frac{\alpha}{2}-1} \left( \frac{ f(e^{i \theta} )}{ ( 1 . e^{i \theta} ) ^\alpha}  -   \frac{ f(e^{i \theta}+ yu ) }{( 1 . ( e^{i \theta} + y u )  ) ^\alpha}  \right) 
   =\, - \int_{0}^\infty  \hspace{-5pt} dy\,  y^{-\frac{\alpha}{2}-1} \int_0 ^y \hspace{-5pt} dt \,  \partial_u \left( \frac{ f(v ) }{( 1 . v   ) ^\alpha}  \right)_{v = e^{i\theta} + t u}  \nonumber \\
&\qquad\qquad \qquad \qquad\qquad =  \, - \frac 2 \alpha  \int_{0}^\infty \hspace{-5pt} dt \, t^{-\frac{\alpha}{2}} \left( \frac{ \partial_u f(e^{i \theta}+ tu ) }{( 1 . ( e^{i \theta} + t u )  ) ^\alpha}   - \alpha \frac{(1.u) f(e^{i \theta}+ tu )}{( 1 . ( e^{i \theta} + t u )  ) ^{1+\alpha}}  \right) .\label{eq:kokoko}
\end{align*}
Hence, we have
$$
\partial_\veps \left( K_\alpha f \right) (u) = ( H_\alpha \bar f)_\veps (u).  
$$
In particular, if $f \in \cH_{\alpha/2,\kappa}$ is an eigenvector of $K_\alpha$ then, we have $\bar f \in C_\kappa (X)$ and $\bar f$ is an eigenvector of $H_\alpha$ with the same eigenvalue.  It follows that the spectrum of $K_\alpha$ is included in the spectrum of $H_\alpha$. We may summarize this inclusion as
\begin{equation}\label{eq:inclusionspectrum}
\sigma_{\cH_{\alpha/2,\kappa}} (K_\alpha) \subseteq  \sigma_{C_\kappa (X)} (H_\alpha). 
\end{equation}
Since $C_\kappa(X)\subset L^2_\kappa(X)$,  we can see $H_\alpha$ as an operator on $L^2_\kappa(X)$.
The next lemma is the main technical ingredient for this part of the proof. It proves that if $\Re(\alpha)$ is not equal to $1$ or $1/2$, $H^m_\alpha$ is a trace class operator if $m$ is large enough (depending on $\alpha$).  
\begin{proposition}\label{prop:traceclass}
Let $\kappa\in (0,1)$. For integer $\ell \geq 0$, let $V_\ell = \{ \alpha \in \dC :  2^{-\ell} < \Re (\alpha) < 2^ {-\ell+1}\}$ and $W_\ell=V_\ell / 3 = \{\alpha\in \dC:   2^{-\ell} <3 \Re(\alpha)< 2^{-\ell+1}\}$. If $\alpha \in V_\ell$ then $H^{2^\ell}_\alpha$ is a Hilbert-Schmidt operator in $L^2_\kappa (X)$.  Finally, if $  \alpha \in W_\ell$ with $\ell\ge 1$, then $H_\alpha^{3 \cdot2^{\ell}}$  is a Hilbert-Schmidt operator in $L^2_\kappa (X)$. 

  \end{proposition}
\begin{proof}
Let $P$ be defined by \eqref{eq:defP1} and set  $(Q f ) ( u , \veps) = P ( f_\veps ) (u)$. From \eqref{eq:S2P}, and since the $M_{\veps \veps'}$ and the $N_\veps$ are uniformly bounded, the Hilbert Schmidt norm
of $H_\alpha^m$ in $L^2_\kappa(X)$ is always bounded by the Hilbert-Schmidt norm of $P^m$ in $L_\kappa^2(S_1^+)$.
We deduce that it is sufficient to prove that $P^{m}$ is Hilbert-Schmidt on $L^2_\kappa (S_1^+)$.

\noindent{\em Bound on the kernel of $P$. } It is simpler to work with $P$ in  its kernel form \eqref{eq:Pkernel}. We claim that,
\begin{equation}\label{eq:boundkernel}
\left| k( \omega, \psi) \right|  \leq \left\{ \begin{array}{ll}   C  | \psi - \omega |^{\Re (\alpha) - 1} \left( \sin ( 2 \psi) \wedge \sin ( 2 \omega) \right)^{-\Re(\alpha) /2} & \hbox{ if $0 < \Re (\alpha) < 1$}, \\
  C \left( \sin ( 2 \psi) \wedge \sin ( 2 \omega) \right)^{-1/2}  \left(1 \vee     \ln \left(\frac{\sin ( 2 \psi)} { \left| \omega - \psi \right|} \right)  \right)  & \hbox{ if $  \Re (\alpha) = 1$}, \\
C  \left( \sin ( 2 \psi) \wedge \sin ( 2 \omega) \right) ^{\Re(\alpha) /2 -1} & \hbox{ if $1 < \Re (\alpha) < 2$}, \end{array}\right.
\end{equation}
where the constant $C$ is uniform over all $\Re(\alpha)$ in a closed set of $(0,2)$.  Indeed, let us assume for example that $\psi < \omega$ and prove that \eqref{eq:boundkernel} holds. If  $\beta = \Re (\alpha)$, since $\sin (x) \geq 2 x / \pi$ for $x \in [0,\pi/2]$, we get 
\begin{equation}\label{eq:boundkernelk}
\left| k( \omega, \psi) \right| \leq C |\psi - \omega|^{\beta -1 } \int_0^{\psi} \left( \theta \wedge \left(\frac \pi 2 - \theta\right) \right)  ^{\frac \beta  2 - 1} |\theta - \psi|^{-\frac \beta  2} |\theta - \omega|^{- \frac \beta  2} d \theta. 
\end{equation}
Now, for $0 < a < b$, if $0 < \beta < 1$, 
\begin{eqnarray*}
I = \int_0 ^a x^{\frac \beta 2-1} |x - a|^{- \frac \beta  2 } |x - b |^{- \frac \beta 2} dx  & =&  a^{- \frac \beta 2} \int_0 ^1 x^{\frac \beta 2-1} |x - 1|^{- \frac \beta 2 } |x - b/a |^{-\frac  \beta 2} dx   \leq  C   b^{-\frac \beta 2}. 
\end{eqnarray*}
If $1 \leq \beta < 2$, then a singularity appears when $b/a$ is close to $1$. 
\begin{eqnarray*}
I & \leq &  C a^{- \frac \beta 2}  \int_0 ^1 |x - 1|^{- \frac \beta 2 } |x - b/a |^{-\frac  \beta 2} dx \\
&=  & C a^{- \frac \beta 2}  \int_0 ^1 x^{- \frac \beta 2 } (x + (b/a -1))^{-\frac  \beta 2} dx \\
& \leq & C a^{-\frac \beta 2} (b/a -1)^{-\beta + 1} \int_0^{(b/a -1)^{-1}} x^{- \frac \beta 2 } (x + 1)^{-\frac  \beta 2} dx. 
\end{eqnarray*}
Since $\sin (x) \leq x$, we deduce from \eqref{eq:boundkernelk} that \eqref{eq:boundkernel} holds if $\psi \leq \pi/4, \psi<\omega$. If $\pi / 4 < \psi < \omega < \pi /2$ then, we should upper bound for $0 < a < b < 1$, 
\begin{eqnarray*}
II = \int_0 ^a |1 - x|^{\frac \beta 2-1} |a -  x| ^{- \frac \beta  2 } | b - x |^{- \frac \beta 2} dx  & =&   \int_0 ^a  (x + ( 1 - a ) ) ^{\frac \beta 2-1} x^{- \frac \beta 2 } (x + b - a )^{-\frac  \beta 2} dx. 
\end{eqnarray*}
Now, with $s = b - a$ and $t = 1 - a > s$,
\begin{eqnarray*}
  \int_0 ^1  x^{- \frac \beta 2 }  (x + t ) ^{\frac \beta 2-1} (x + s )^{-\frac  \beta 2} dx &=& s^{-\beta/2} \int_0 ^{1/s}   x^{- \frac \beta 2 }  (x + t/s ) ^{\frac \beta 2-1} (x + 1 )^{-\frac  \beta 2} dx \\
& \leq & C s t^{-\beta/2}  +  s^{-\beta/2}   \int_1 ^{\infty}   x^{-  \beta  }  (x + t/s ) ^{\frac \beta 2-1} dx \\
& \leq & C s t^{-\beta/2}  +  t^{-\beta/2}   \int_{s/t} ^{\infty}   x^{-  \beta  }  (x + 1 ) ^{\frac \beta 2-1} dx.
\end{eqnarray*}
 If $0 < \beta < 1$, the above expression is $O ( t^{-\beta /2})$. If $\beta > 1$, it is of order $O(s^{1- \beta} t^{\beta/2 -1})$. Finally, if $\beta = 1$, it is $O ( t^{-1/2} ( 1 \vee \log (t / s))) $. It completes the proof of \eqref{eq:boundkernel}.


\noindent{\em Case $\Re(\alpha) > 1$. } We have that, if  $\delta, \kappa\in [0,1)$,
\begin{align}
\int_{[0,1]^2}\frac{ dx}{|x-1/2|^\kappa} \frac{dy}{|y-1/2|^\kappa} \,  (x \wedge y ) ^{- \delta }     \label{eq:HSlast} =2\int_{0}^1\frac{dy}{|x-1/2|^\kappa} \,    \int_0^y  \frac{dx}{|x-1/2|^\kappa}  x^{ -\delta}        < \infty  .
\end{align}
In particular, if $\beta = \Re(\alpha) > 1$, then, we deduce from \eqref{eq:boundkernel} applied to $\delta = 2 - \beta$  and the fact that 
 if  $u = e^{i \theta}$, $\theta \in (0,\pi/2)$,  
 \begin{equation}\label{jhk} |u - e^{i\pi/4}| \leq | \theta - \pi /4| \leq (\pi  / 2 ) | \sin ( \theta - \pi/4) |  =  (\pi  / (2\sqrt 2) ) |i . u|\,,\end{equation}
 that
$$
\int_{[0,\pi/2]^2} \left| k( \omega, \psi) \right|^2 \frac{d \omega}{|i.e^{i\omega}|^\kappa} \frac{d \psi}{|i.e^{i\psi}|^\kappa}
  \leq C    \left( \int_{[0,\frac{\pi}{2}]^2}   (\psi \wedge \omega ) ^{\beta -2 } \frac{d\omega}{|\omega-\frac{\pi}{4}|^\kappa} \frac{d\psi}{|\psi-\frac{\pi}{4}|^\kappa}  \right) ^2 
$$
is finite for $\beta>1$ and $\kappa<1$.
Hence $P$ and $H_\alpha$ are Hilbert-Schmidt operators in $L^2_\kappa(X)  $.

\noindent{\em Case $0 < \Re(\alpha) < 1$. } If $\beta = \Re(\alpha) < 1$, then, we need to take powers of $P$, to obtain a kernel operator in the Hilbert-Schmidt class. Let $\ell \geq 1$, we assume that 
\begin{equation}\label{eq:betaint}
2^{-\ell} < \beta < 2^{-\ell+1}.
\end{equation}
By recursion, we will upper bound the kernel of $P^{2^t}$  and show that $P^{2^{\ell}}$ is Hilbert-Schmidt in $L^2_\kappa(S_1^+)$.  To perform that, we will use the following inequality in the recursion. Let  $0 < \zeta ,\zeta' < 1$ and $0 < a < b$, we set
\begin{eqnarray}
I_{\zeta,\zeta'} (a,b) &=&  \int_0 ^1 |x - a |^{\zeta-1} | x - b |^{\zeta'-1} (x \wedge a) ^{-\zeta/2 } (x \wedge b)^{-\zeta'/2} dx \nonumber\\
&\leq&  \left\{ \begin{array}{ll}   C  a^ {-(\zeta+\zeta')/2} (b - a)^{  \zeta+\zeta' -1} & \hbox{ if $0 <  \zeta+\zeta' < 1$}, \\
  C  a^{ -  \zeta/2} b^{-\zeta'/2} & \hbox{ if $ 1 < \zeta+ \zeta'$}.
\end{array}\right.\label{eq:I4rec} 
\end{eqnarray}
and $I_\zeta=I_{\zeta,\zeta}$. Indeed, as in \eqref{eq:HSlast}, we decompose the integration interval over $[0,a]$, $[a,b]$ and $[b,1]$. It follows easily from the scaling arguments below \eqref{eq:boundkernelk} that the integral over $[0,a]$ is dominating if $2 \zeta < 1$ and the integral over $[b,1]$ is dominating if $2 \zeta >1$. 

From \eqref{eq:boundkernel}, up to a change a variable $x = \sin(\omega)$, it suffices to prove that the kernel operator  $\bar P$ in $L^2 ([0,1])$, 
$$
\bar k (x,y) = |x - y|^{\beta - 1}  ( x \wedge y)  ^{- \beta / 2} 
$$
satisfies that $\bar P^{2^{\ell}}$ is Hilbert-Schmidt in $L^2_\kappa ([0,1])=\BRA{f:\int_0^1|f(x)|^2 |x-\frac{1}{2}|^{-\kappa} dx<\infty}$.   
For $t \geq 1$, we set $\zeta_{t} = 2^{t-1} \beta$. If $\bar k_t$ is the kernel of $\bar P^{2^{t-1}}$, we have,
$$
\bar k_2 (x,y) \leq I_{ \zeta_1} ( x, y). 
$$
 Now,  if $\beta$ is an in \eqref{eq:betaint} with $\ell = 1$, then $2 \zeta_1  = 2 \beta > 1$ and from \eqref{eq:I4rec}, $\bar P^2$ is Hilbert-Schmidt since $\bar k_2 (x,y) \leq C  x^{-\beta/2} y ^{-\beta/2} $ and $ \beta < 1 $. If $\ell \geq 2$, we find from \eqref{eq:I4rec},
$$
\bar k_2 (x,y) \leq C  |x - y|^{\zeta_2 - 1} (x \vee y) ^{-\zeta_2/2}, 
$$
and $$
\bar k_3 (x,y) \leq  C    I_{ \zeta_2} ( x, y) . 
$$
We deduce easily by recursion that for some new constant $C > 0$, 
$$
\bar k_{\ell} (x,y) \leq C |x - y |^{\zeta_{\ell} -1} (x \wedge y )  ^{-\zeta_{\ell}/2}.  
$$
and 
\begin{equation}\label{eq:km+1}
\bar k_{\ell+1} (x,y) \leq  C     I_{ \zeta_{\ell}} ( x, y)  \leq C'  x^{-\zeta_{\ell}/2}  y^{-\zeta_{\ell}/2}  . 
\end{equation}
Now, since $\zeta_{\ell}/2 = 2^{\ell-2}\beta $ and  $2^{\ell-1} \beta < 1$, the operator $\bar P^{2^{\ell}}$ is Hilbert-Schmidt in $L^2_\kappa (S_1^+)$. It proves the first statement of the proposition.

For the second statement,  we proceed similarly but starting with the kernel $P^3$ instead of $P^2$. First, if $2\beta<1$ but $3\beta>1$, from \eqref{eq:I4rec} the kernel of $P^3$, which is bounded by $I_{\beta,2\beta}$, is in $L^2_\kappa(S_1^+)$.
 For $\beta\in (3^{-1} 2^{-\ell},3^{-1} 2^{-\ell+1}), \ell\ge 1$, we can proceed by recursion as above. \end{proof}


We recall that if an operator is Hilbert-Schmidt, its square is trace class. Hence, Proposition \ref{prop:traceclass} implies that 
\begin{cor}\label{cor:traceclass} Let $\kappa\in (0,1)$. Let
 $V=\{\alpha\in\mathbb C: \Re(\alpha)\in (0,2)\backslash\{1/2,1\}\}$, there exists  $m\in\mathbb N$ finite such that   $H_\alpha^m$ is trace class 
in the Hilbert space $L^2_\kappa( X)$. More precisely, let $V_\ell$ and $W_\ell$ be as in Proposition \ref{prop:traceclass}. If $\alpha \in V_\ell$, $\ell \geq 0$, $H_\alpha^{2^{\ell+1}}$ is trace class and if $\alpha \in W_\ell$, $\ell \geq 1$, $H_\alpha^{3 \cdot 2^{\ell+1}}$ is trace class. 
\end{cor}

\paragraph{Step 3 : Analyticity of the Fredholm determinant.  }

From Corollary \ref{cor:traceclass}, the Fredholm determinant of $H_\alpha^m$ is then properly defined for $\alpha \in V$ and $m$ large enough.  We now want to justify that $\alpha \mapsto \det ( I -  H_\alpha^m)$ is analytic on $V$. Consider a trace class operator $Q$ in $L_\kappa^2(X)$, written for some measurable kernel $q : X^2 \to \C$, 
$$
Q f (x) = \int_{X} q (x,y) f (y) d \rho(y),
$$
where $\rho$ is the underlying measure of our $L^2$ space, that is  $L_\kappa^2 (X) = L^2( X,  \rho)$ with
$$\int f d \rho  =\int_{[0,\pi/2]} f_0 (e^{i \psi} ) d \psi + \sum_{\veps \in \{1, i\}} \int_{[0,\pi/2]} f_\veps (e^{i \psi} ) \frac{d \psi}{|i.e^{i\psi}|^\kappa}\,.$$

There is a specific choice of $q$ which is especially relevant. If $x = (e^{i \omega},\veps)$ and $y = (e^{i \psi}, \delta)$, $0 < \omega, \psi < \pi/2$, we consider the average value around $(x,y)$, for $r >0$, 
$$ q (x,y,r) = \frac {1}{r^2} \int_{C_r}  q((e^{i \omega},\veps),(e^{i \psi},\delta)) d \psi d \omega, $$
 with $C_r = \left( [\omega-r/2,\omega+r/2] \times [\psi - r/2, \psi + r/2] \right) \cap [0,\pi/2]^2$.
We  define the Lebesgue value of $q$ as $\tilde q(x,y) = \lim_{r \downarrow 0} q(x,y,r) $.  From Lebesgue's Theorem, $\rho \otimes \rho$-a.e., $\tilde q (x,y) = q(x,y)$.  From Brislawn's Theorem \cite{MR929421},  $\tilde q(x,x)$ exists $\rho$-a.e. and we have $\Tr (Q) = \int_X \tilde q (x,x) d \rho(x)$.  As a consequence, Simon \cite[Theorem 3.7]{MR2154153} implies
\begin{equation}\label{eq:fredbris}
\det ( I + Q) = \sum_{n =0}^\infty \frac 1 {n !}\int_{X^n} \det \left( \tilde q( x_i, x_j) \right)_{1 \leq i , j \leq n} \prod_{i=1}^n d \rho (x_i).   
\end{equation}
This representation of the Fredholm determinant will be used in the proof of the following lemma. 
\begin{lem}\label{le:detholo}
Let $f \in L^ 2(X,\rho)$, $\Omega$ be an open connected set of $\dC$ and for $ z \in V$, let $Q_z$ be a trace class operator in $L^2(X,\rho)$ with a measurable kernel $q_z(x,y)$.  Assume that $\rho\otimes \rho$-a.e.  (i) for all $z \in \Omega$, $|q_z (x,y)| \leq f(x) f(y)$  and (ii) $z \mapsto q_z(x,y)$ is analytic on $\Omega$. Then 
$z \mapsto \det ( I + Q_z)$ is analytic on $\Omega$. 
\end{lem}

We will use repeatedly the following elementary consequence of Cauchy's formula and Lebesgue dominated convergence. 
\begin{lem}\label{le:intholo}
Let $\Omega$ be an open set of $\dC$ and for each $z \in \Omega$, let $x \mapsto f(x,z)$ be a measurable function in a measure space $(X,\rho)$.  Assume that there exists $g \in L^1 (X,\rho)$ such that $\rho$-a.e. (i) for all $z \in \Omega$,  $|f(x,z)| \leq g(x)$ and (ii) $z \mapsto f(x,z)$ is analytic on $\Omega$. Then $z \mapsto \int_X f (x,z) d\rho (x)$ is analytic on $\Omega$. 
\end{lem}

\begin{proof}[Proof of Lemma \ref{le:detholo}]
Since $f \in L^ 2(X,\rho)$, it also belongs to $L^1 (X,\rho)$. From (i)-(ii) and Lemma \ref{le:intholo} for any $r >0$, $\rho\otimes \rho$-a.e., $z \mapsto q_z(x,y,r)$ is analytic on $\Omega$. From Vitali's convergence theorem, it follows that, $\rho\otimes \rho$-a.e., $z \mapsto \tilde q_z(x,y)$  is analytic on $\Omega$. 

Now, we let $p_z(x,y) = \tilde q_z (x,y) / ( f(x) f(y))$. By assumption (i), $\rho\otimes \rho$-a.e., for all $z \in \dC$  $| p_z(x,y)| \leq 1$. From the multi-linearity of the determinant, 
 $$
D_z (x_1, \ldots, x_n) =  \det \left( \tilde q_z( x_i, x_j) \right)_{1 \leq i , j \leq n} =  \det \left( p_z( x_i, x_j) \right)_{1 \leq i , j \leq n}\prod_{i=1}^n  f(x_i)^2 . 
$$
However, from Hadamard's inequality, $\rho^{\otimes n}$-a.e.,
$$
\left| \det \left( p_z( x_i, x_j) \right)_{1 \leq i , j \leq n}\right| \leq n^{n/2}. 
$$
Hence, $\rho^{\otimes n}$-a.e., for all $z \in \Omega$, $|D_z (x_1, \ldots, x_n)| \leq n^{n/2} \prod_{i=1}^n  f(x_i)^2$. From (ii), $\rho^{\otimes n}$-a.e., $z \mapsto D_z (x_1, \ldots, x_n)$ is analytic. We deduce by a new application of Lemma \ref{le:intholo} that 
$$
z \mapsto \frac 1 {n!} \int_{X^n}  D_z (x_1, \ldots, x_n)  \prod_{i=1}^n  d \rho (x_i)
$$
is analytic on $\Omega$ and bounded by $\| f \|_{L^2 (X)}^{2n} n^{n/2}/ n! \leq n^{-n/2} C^n$, with $C = e \| f \|_{L^2 (X)}^2$. The series $\sum_n n^{-n/2} C^n$ being convergent, a new application of Vitali's convergence theorem proves  the analycity on $\Omega$ of 
$$
z \mapsto \sum_{n =0}^\infty \frac 1 {n !}\int_{X^n} D_z (x_1, \ldots, x_n)  \prod_{i=1}^n d \rho (x_i).   
$$
It remains to use \eqref{eq:fredbris}.   
\end{proof}

\begin{proposition}\label{prop:holodetH} For integer $\ell \geq 0$, let $V_\ell$ and $W_\ell$ be as in Proposition \ref{prop:traceclass}. The function $\alpha \mapsto  \det( I - H^{2^{\ell+1}}_\alpha)$ is analytic on $V_\ell$.  
For $\ell \geq 1$, $\alpha \mapsto  \det( I - H^{3\cdot 2^{\ell+1}}_\alpha)$ is analytic on $W_\ell$.  
\end{proposition}

\begin{proof} We treat the first case, the second being similar.
Let $\Omega \subset V_\ell$ be an open set with $\bar \Omega \subset V_\ell$. In particular for all $\alpha \in \Omega$, 
$$
2^{-\ell} <  \beta_0 <  \Re(\alpha)<\beta_1 < 2^{-\ell+1}. 
$$
Consider the operator $P$ defined in \eqref{eq:defP1} on $L^2 _\kappa (S_1^+)$. We denote the kernel of $P^{2^{s-1}}$ by $k_{s}$ (it depends implicitly on $\alpha \in V_\ell$). We will prove that  for all $\omega, \psi \in (0,\pi/2)^2$, $\alpha \mapsto k_{\ell+1} (\omega, \psi)$ is analytic on $\Omega$ and that $ | k_{\ell+1} (\omega, \psi) | \leq C f (\sin ( 2 \psi) )f ( \sin (2\omega)) $ with 
$$f(x) = \left\{\begin{array}{ll} x^{\beta_0/2-1} & \hbox{ if $\ell = 0$} \\
 x^{-2^{\ell-2} \beta_1} & \hbox{ if $\ell \geq 1$.}
\end{array} \right.
$$
We will then argue that we have the same properties for the kernels of $H^{2^{\ell}}_\alpha$ and  $H^{2^{\ell+1}}_\alpha$. It will remain to apply Corollary \ref{cor:traceclass} and Lemma \ref{le:detholo} to conclude the proof.  The case $\alpha\in W_\ell$ is similar.

\noindent{\em Iterated kernels of $P$, case $\ell = 0$.} 
From the explicit form of the kernel $k = k_1$ in \eqref{eq:Pkernel}, we see by Lemma \ref{le:intholo} that,  for all $\omega ,  \psi$,  $\alpha \mapsto k_1 ( \omega, \psi)$ is analytic on $ \Omega$.  From \eqref{eq:boundkernel}, $|k_1 ( \omega, \psi)| \leq C (\sin (2 \psi) \wedge \sin (2 \omega) ) ^{\beta_0/2 -1 } \leq C f( \sin (2 \psi) ) f (\sin (2 \omega) )$. 

\noindent{\em Iterated kernels of $P$, $\ell \geq 1$.} 
First, from the explicit form of the kernel $k = k_1$ in \eqref{eq:Pkernel} and Lemma \ref{le:intholo}, for all $\omega \ne  \psi$,  $\alpha \mapsto k_1 ( \omega, \psi)$ is analytic on $ \Omega$ and from \eqref{eq:boundkernel}, $|k_1 ( \omega, \psi)| \leq C  |\psi- \omega|^{\beta_0 -1}(\sin (2 \psi) \wedge \sin (2 \omega) ) ^{-\beta_1/2 }  $.

We set $\delta_\ell = 2^{\ell-1} \beta_1$ and $\zeta_\ell = 2^{\ell-1} \beta_0$. We use the following inequality (which generalizes \eqref{eq:I4rec}) for $0 < a, b < 1$, $0 < \delta , \zeta < 1$,
\begin{equation*}
J_{\zeta,\delta} (a,b) =  \int_0 ^1 |x - a |^{\zeta-1} | x - b |^{\zeta-1} (x \wedge a) ^{-\delta/2} (x \wedge b)^{-\delta/2} dx \leq  \left\{ \begin{array}{ll}   C  a^ {-\delta} (b - a)^{ 2 \zeta -1} & \hbox{ if $0 <  2 \zeta < 1$}, \\
  C  a^{ -  \delta/2} b^{-\delta/2} & \hbox{ if $ 1 < 2 \zeta$}.
\end{array}\right.\label{eq:I4rec2} 
\end{equation*}
We may argue as in the proof of Proposition \ref{prop:traceclass} and use Lemma \ref{le:intholo} by recursion on $1 \leq s \leq \ell$. We obtain that $| k_s (\omega, \psi) | \leq C_s |\psi - \omega|^{\zeta_s -1} (\sin (2 \psi) \wedge \sin (2 \omega) ) ^{-\delta_s/2 } $ and $\alpha \mapsto k_s( \omega, \psi)$ analytic on $\Omega$ for all $\omega \ne \psi$ in $(0,\pi/2)$. Since $2 \zeta_\ell = \zeta_{\ell+1} = 2^{\ell} \beta_0 > 1$, we find that 
$ k_{\ell+1} ( \omega, \psi) \leq C  f (\sin ( 2 \psi) )f ( \sin (2\omega))$ and is analytic in $\alpha \in \Omega$ for all $(\omega, \psi) \in (0,\pi/2)^2$.

\noindent{\em Iterated kernels of $H_\alpha$.} For $x = (e^{i \omega},\veps) ,y = (e^{i \psi},\delta) \in X$, let $h_s(x,y)$ be the kernel of $H^{2^{s-1}}_\alpha$. We now study the analyticity of the kernels of $\alpha \mapsto h_s(x,y)$ for $s \geq 1$. First, in \eqref{eq:defHalpha}, the function $\alpha \mapsto c'_\alpha$ is analytic on $U$. It thus  suffices to study  $\hat S = S J$.  If $\hat P = P J$, we have from \eqref{eq:S2P} in matrix form
$$
\hat S = \begin{pmatrix} M_{00} & M_{01} & M_{0i} \\
M_{10} & 1 & 0 \\
M_{i0} & 0 & 1 
\end{pmatrix}  \begin{pmatrix} \hat P & 0 & 0 \\
0 & \hat P & 0 \\
0 & 0 & \hat P\end{pmatrix}\begin{pmatrix} N_1 & 0 & 0 \\
0 & N_2 & 0 \\
0 & 0 & N_3
\end{pmatrix} 
J
 $$
 In particular, the kernel of $\hat S$ at $x = (e^{i \omega},\veps) ,y = (e^{i \psi},\delta) \in X$ is equal to $\sigma ( x, y) = g^\alpha_{\veps,\delta} ( \psi) k_1 ( \psi, \omega)f^\alpha_{\veps,\delta} ( \omega) $, where $g^\alpha_{\veps,\delta}( \psi), f^\alpha_{\veps,\delta}(\omega)$ are  bounded and analytic in $\alpha \in \Omega$. This factors $g^\alpha_{\veps,\delta}(\psi), f^\alpha_{\veps,\delta}(\psi)$ is harmless and  the above argument carries over easily to this more general situation. We find that the   kernel $\sigma_{\ell} (x,y)$ of $\hat S^{2^{\ell-1}}$ satisfies  for all $\omega, \psi \in (0,\pi/2)^2$, $\alpha \mapsto \sigma_{\ell+1} (x,y)$ is analytic on $\Omega$ and $ | \sigma_{\ell+1} (x, y) | \leq C f (\sin ( 2 \psi) )f ( \sin (2\omega)) $. By a new application Lemma \ref{le:intholo}, the same holds for $\sigma_{\ell+1}$ replaced by $\sigma_{\ell+2}$.  Finally, since $H_\alpha = c'_\alpha \hat S$, the same holds for $h_{\ell+2}$.  By Corollary \ref{cor:traceclass}, $H_\alpha^{2^{\ell+1}}$ is trace class.  Hence, the conditions of Lemma \ref{le:detholo} are fulfilled for $-H_\alpha^{2^{\ell+1}}$, we obtain that $\alpha \mapsto \det ( I- H_\alpha^{2^{\ell+1}})$ is analytic on $\Omega$. \end{proof}

\paragraph{Step 4 : Proof of Theorem \ref{th:4IFT}. }
We need a final lemma before proving Theorem \ref{th:4IFT}
\begin{lem}\label{le:detnonzero}
Let $\ell \geq 0$ be an integer and $ 2^{-\ell} < \beta < 2^{-\ell+1}$. If $m \geq 2^{\ell}$ then, 
$$
\lim_{t \to \infty} \| H^{m}_{\beta + it } \|= 0 ,
$$ 
where $\| \cdot \|$ denotes the operator norm in $L^2_\kappa (X)$.  Similarly, the same conclusion holds if $ 2^{-\ell} < 3\beta < 2^{-\ell+1}$ with $\ell \geq 1$ and $m \geq 3  . 2^{\ell}$. 
\end{lem}

\begin{proof}
We observe that 
$
H_{\beta + it} = f (t) \hat S, 
$
where $f(t) = \frac {2c_{\beta+it} }{\alpha a_0^2} $ goes to zero when $t$ goes to infinity  by definition \eqref{defc}. Moreover, the operator norm of $\hat S^m=\hat S^m_\alpha$ 
is bounded by the Hilbert-Schmidt norm of $\hat S_{\Re (\alpha)}^m$ which is finite by Proposition  \ref{prop:traceclass}  if $m$ is large enough. \end{proof}

We are now ready to prove Theorem \ref{th:4IFT}.  Let $\ell \geq 0$ be an integer, $V_\ell , W_\ell$ be as in Proposition \ref{prop:traceclass} and  $F$ be a closed subset of $V_\ell$.  From \eqref{eq:inclusionspectrum} and Lemma \ref{le:compactness}, if $I + K_\alpha$ is not an isomorphism of $\cH_{\alpha/2}$ then  $-1$ is an eigenvalue of $H_\alpha$ in $C_\kappa(X)$ and $1$ is an eigenvalue of $H_\alpha^m$ in $C_\kappa(X)$ for any $m$ even. If $\alpha \in V_\ell$, we take $m = 2^{\ell+1}$, then by Corollary \ref{cor:traceclass}, $H_\alpha^m$ is a trace class operator in $L^2_\kappa (X)$. Since $C_\kappa(X) \subset L^2_\kappa(X)$, from Theorem \cite[Theorem 3.7]{MR2154153}, we have $\det (I  - H_\alpha^m) = 0$. In summary, 
$$
\cA \cap F  \subset \{ \alpha \in F : \det (I  - H_\alpha^m) = 0 \}. 
$$
However, by Proposition \ref{prop:holodetH},  $\varphi : \alpha \mapsto \det (I  - H_\alpha^m)$ is analytic on $V_\ell$ and by Lemma \ref{le:detnonzero}, $\det (I  - H_\alpha^m)$ is non-zero for $\Im(\alpha)$ large enough. It follows that the level set $0$ of $\varphi$ cannot have accumulation points in $F$.  
Similarly, if $F\subset W_\ell$, $F$ cannot have accumulation points. This proves Theorem \ref{th:4IFT}. \qed

\section{Local law : proofs of Proposition \ref{prop:afpi}  and Theorem \ref{th:main}}
\label{sec:rates}

In this section, we conclude the proof of Theorem \ref{th:main}. We will first establish Proposition \ref{prop:afpi}, the main technical steps are an adaptation of  \cite{BG}. Then, using the result of the previous section, we will derive a quantitative estimate on the resolvent of $A$. 

\subsection{Concentration and deviation inequalities}
\label{subsec:conc}

For the reader convenience, we start by stating a lemma on the concentration for the diagonal of the resolvent of random matrices. The Lipschitz norm of $f:\bC \to\bC $ is $$
\|f \|_{L} = \sup_{ x \ne y }\frac{ |f(x) - f(y) | }{ |x - y |}.
$$
\begin{lem}[{\cite[Lemma C.3]{BG}}]\label{le:concres}
 Let $\beta \in (0,1)$ and $z = E +  i\eta \in \dC_+$. Let $B \in M_n(\dC)$ be a random Hermitian matrix and $G  = (B- z)^{-1} $. Let us assume that the
  vectors $(B_i)_{1 \leq i \leq n}$, where $B_i := (B_{ij})_{1 \leq j \leq i}
  \in \bC^i$, are independent. Then for any $f:\bC\to\bR$ and every
  $t\geq0$,
  \[
  \bP \left( \left|\frac 1 n  \sum_{k=1}^n  f (G _{kk} ) -\bE \frac 1 n  \sum_{k=1}^n  f (G_{kk} )  \right|  \geq t \right) %
  \leq 2 \exp\left({-\frac{n  \eta^2 t^2}{8 \| f\|_L ^2 }}\right).
  \]
\end{lem}

%

Our first claim checks that the norm $\| \cdot \|_\kappa$ has some H\"older regularity. 

\begin{lem}\label{le:gammaholder}
Let $\beta \in (0,1)$  and for $k \in \{1,2\}$, $x_k\in \cK_1$, $|x_k| \leq \eta^{-1}$, $\eta\le 1$.
\begin{enumerate}[(i)]
\item We define $\gamma_k \in \bar \cH^0_\beta$ by $$
\gamma_k (u) =  \PAR{ x_k . u }^{\beta}. 
$$
There exists a constant $c = c(\beta)$ such that for any $0 < \delta < \beta$, 
$$
\|\gamma_k\|_{1 - \beta + \delta} \leq c |x_k|^{\beta} \AND \| \gamma_1 - \gamma_2 \|_{1 - \beta + \delta}  \leq c \eta^{-\beta } \PAR{ |x_1 - x_2 |^\beta + \eta^{\delta} |x_1 - x_2|^{\delta}}. 
$$
\item Assume additionally that  $\Re (x_k) \geq s$. We set $\gamma_k (u) = (x_k^{-1}.u )^\beta$.  Then, for some $c = c(\beta)$, 
\begin{equation}\label{le:hihi}
 \| \gamma_1 - \gamma_2 \|_{1 - \beta + \delta} \leq c s^{\beta -2} \eta^ { 2 \beta - 1} | x_1 - x_2|.
\end{equation}
\end{enumerate}
\end{lem}
\begin{proof}
We prove the second statement of (i). The first statement is easy. We shall use two simple bounds. First, for any $\beta \in (0,1]$, for all $x,y \in \cK_1$
\begin{equation}\label{eq:holderbeta}
| x^\beta - y^\beta | \leq  | x - y|^{\beta}.  
\end{equation}
Secondly, for any $u \in \cK_1^+$, $h \in \cK_1$, 
\begin{equation}\label{eq:hubound}
|h|  |i.u|  \leq |h.u| \leq \sqrt{2} |h|  |u|,
\end{equation}
(see \cite[Equation (55)]{BG}). From \eqref{eq:holderbeta}-\eqref{eq:hubound}, we get 
\begin{equation}\label{eq:holdergamma1}
| \gamma_1 (u) - \gamma_2 (u) | \leq | x_1. u - x_2. u |^{\beta} \leq 2^{\beta/2} |x_1 - x_2|^{\beta}.  
 \end{equation}
Similarly, we get
 \begin{equation}\label{eq:holdergamma1b}
| \gamma_1 (u) - \gamma_1 (u') |  \leq 2^{\beta/2}\eta^{-\beta} |u - u'|^{\beta}.  
 \end{equation}
Hence, 
$$
\sup_{ u \in S_1^+} |  \gamma_1 (u) - \gamma_2 (u)  | \leq c |x_1 - x_2|^{\beta}.  
$$
We also need to control the derivative of $\gamma_1 - \gamma_2$. For $\veps \in \{1,i\}$, 
$$
\partial_\veps \gamma_k (u) = \beta (x_k.u)^{\beta - 1} (x_k.\veps). 
$$
We consider the function $g \in \cH_\beta$ defined for $x \in \cK_1$ by $ g(x) = (x. u) ^{\beta-1} (x.\veps)$. For $s >0$ to be fixed later, let $\phi : \dC \to \dC$ be equal to one for $|z| \geq 2 s$, equal to zero for $|z| \leq s$ and growing linearly with the modulus in between. Thus $\phi$ is Lipschitz with constant $1/s$. We write
$$
g(x) = ( 1 - \phi(x.u) ) g(x) +  \phi(x.u) g(x) = g_1 (x) + g_2(x). 
$$
From \eqref{eq:hubound}, 
\begin{equation}\label{eq:hubound2}
|x. \veps | \leq \sqrt 2 |x| \leq \sqrt{2} \frac{|x.u|}{ |i.u|}.
\end{equation}
If follows that the function $g_1$ is bounded by $c_0 s^{\beta} / |i.u|$. Moreover, the derivative with respect to $x$ of $g(x)$ is equal for $\veps' \in \{1,i\}$,
$$
\partial_{\veps'} g(x) = (\beta-1) (x. u) ^{\beta-2} (\veps'.u) (x.\veps) +  (x. u) ^{\beta-1} (\veps'.\veps) ,
$$
where $\partial_1$ is the derivative with respect to the real part of $x$ and $\partial_i$ is the derivative with respect to the imaginary part of $x$. Using \eqref{eq:hubound2}, we deduce that on the support of the function $x \mapsto \phi(x.u)$,   $g$ is Lipschitz with constant $c s^{\beta-1} / |i.u|$. Thus $g_2$ is Lipschitz with constant $c s^{\beta-1} / |i.u|$ (for some new constant $c >0$). It follows that 
$$
| \partial_\veps \gamma_1 (u) - \partial_\veps \gamma_2 (u) | \leq |g_1 (x_1)|  + |g_1 (x_2)|  + | g_2(x_1) - g_2(x_2)  | \leq c_1 s^{\beta} |i.u|^{-1} + c_1 s^{\beta-1} |x_1 - x_2 |   |i.u|^{-1} . 
$$
We choose $s = |x_1 - x_2|$, we find for some new constant $c >0$. 
\begin{equation}\label{eq:jdejdq}
| \partial_\veps \gamma_1 (u) - \partial_\veps \gamma_2 (u) | \leq  c  |x_1 - x_2 |^{\beta}   |i.u|^{-1} . 
\end{equation}
The above bound is not quite enough due to the factor $|i.u|^{-1}$.  Let $\beta' = \beta - \delta$ and $a = (1 - \beta)/(1 - \beta') < 1$ From \eqref{eq:holderbeta} for any $u,v \in S_1^+$, $x \in \cK_1$,
$$
\ABS{ \PAR{  \frac{(x.u)^a}{i.u} }^{\beta'-1}  - \PAR{  \frac{(x.v)^a }{i.v} }^{\beta'-1} } \leq \ABS{ \frac{i.u }{(x.u)^a} - \frac{i.v}{(x.v)^a} }^{1 - \beta'}. 
$$
If $|i .u| \geq |i.v|$, we find from \eqref{eq:hubound}
$$
\ABS{ \frac{i.u}{(x.u)^a} - \frac{i.v}{(x.v)^a} } \leq \ABS{ \frac{(i.u - i.v)}{(x.u)^a} } + \ABS{\frac{(i.v) ((x.v)^a - (x.u)^a)}{(x.u)^a (x.v)^a } } \leq  \frac{|u - v|}{|x|^a |i .u|^a} +  \frac{\sqrt{2}^a |u - v|^a }{|x|^a |i .u|^{2a-1}}  
$$
where at the last step, we have used   \eqref{eq:holderbeta} (for $\beta = a < 1$). Using that $|x_k| \leq \eta^{-1}$, we arrive at 
\begin{equation}\label{eq:holdergamma2}| (i.u)^{1- \beta'} \partial_\veps \gamma_k(u)  - (i.v)^{1- \beta'} \partial_\veps  \gamma_k(v) | \leq \eta^{-\beta} \PAR{\frac{|u - v|}{(|i.u| \vee |i.v|)^a } }^{1- \beta'} + \eta^{-\beta} \PAR{\frac{|u - v|}{(|i.u| \vee |i.v|)^{2 - 1/a} } }^{1- \beta}
\end{equation}
Now, we fix $s$ to be chosen later. For any $v \in S_1^+$, there exists $u \in S_1^+$, such that $|u-v| \leq s$ and $|i.u| \geq s/2$, see \eqref{jhk}.  Let $t = \eta^ {- \beta} (s / (s/2)^a )^{1-\beta'}  + \eta^{-\beta} (s / (s/2)^{2 - 1/a} )^{1-\beta} $, 
 we have from \eqref{eq:holdergamma2} that 
 $$
\ABS{ (i.v)^{1- \beta'} ( \partial_\veps \gamma_1 (v) - \partial_\veps \gamma_2 (v) ) } \leq 2 t + \ABS{ (i.u)^{1- \beta'} ( \partial_\veps \gamma_1 (u) - \partial_\veps \gamma_2 (u) ) }. 
$$
Observe that $s^{\beta - \beta'} = s^\delta = c t \eta^{\beta}$. Hence, from  \eqref{eq:jdejdq}, we get

$$
\ABS{ (i.v)^{1- \beta'} ( \partial_\veps \gamma_1 (v) - \partial_\veps \gamma_2 (v) ) } \leq c \eta^{ -\beta} s^\delta + c |x_1 - x_2 |^{\beta}  s^{-\beta'}. 
$$
We choose $s = \eta  | x_1 - x_2|$. The second statement of  (i)  follows.  

To prove the claim (ii), observe that
the derivative of $x \mapsto (x^{-1}.u )^\beta$ in the direction $\veps \in \{1,i\}$ is 
$$
-  \beta x^{-2
} (\veps.u) (x^{-1}.u )^{\beta-1}. 
$$
We observe that $\Re (x.u) \geq \Re (x) \geq s$ and $\Re (x^{-1}) = \Re (x) / |x|^2$. This derivative is bounded by $c (\Re x) ^{\beta-1} |x|^{-2\beta}$. Thus, we find from the intermediate value theorem that $|\gamma_1 (u) - \gamma_2 (u) |\leq c s^{\beta-1} \eta^{-2\beta} |x_1 - x_2|$.  

Similarly, the derivative of $x \mapsto (x^{-1}.u )^{\beta-1} (x^{-1}.\veps')$ in the direction $\veps \in \{1,i\}$ is 
$$
-  (\beta-1) x^{-2} (\veps.u) (x^{-1}.u )^{\beta-2} (x^{-1}.\veps') - x^{-2} (x^{-1}.u )^{\beta-1} (\veps.\veps'). 
$$
The derivative is bounded by  $c | x|^{1 - 2\beta} \Re ( x)^{\beta-2}  + c | x|^{- 2\beta} \Re ( x)^{\beta-1} \leq 2c   | x|^{1 - 2\beta} \Re ( x)^{\beta-2}$.  \end{proof}

\begin{lem}\label{le:concnorm}
Let $\beta \in (0,1)$ and $z = E +  i\eta \in \dC_+$. Let $B \in M_n(\dC)$ be a random Hermitian matrix,  $G= (B - z)^{-1}$,  and we define $\gamma \in \bar \cH^0_\beta$ by $$
\gamma (u) = \frac 1 n \sum_{k=1}^n \PAR{ -i G_{kk} . u }^{\beta}. 
$$
Let us assume that the vectors $(B_i)_{1 \leq i \leq n}$ are independent, where $B_i = (B_{ij})_{ j \leq i} \in \dC^i$.  Then,  there exists a constant $c = c(\beta)$ such that for any $0 < \delta < \beta$ and $t \geq 0$, 
$$
 \dP \PAR{ \| \gamma - \dE \gamma \|_{1 - \beta + \delta}  \geq t } \leq  c (\eta^\beta t )^{-\frac 1  \delta} \exp \PAR{ - c n (\eta^{\beta} t )^{\frac 2 \delta} }.
$$
\end{lem}

\begin{proof}
 For $1 \leq k \leq n$, we set $h_k = -i G_{kk}$. We also set $\underline \gamma(u) = \gamma(u) - \dE \gamma(u)$. The first step is to find a concentration inequality for the function $\gamma(u)$ and its derivatives for any fixed $u \in S_1^+$.  It will follow from  Lemma \ref{le:concres}. We consider the function $f \in \cH_\beta$ defined for $x \in \cK_1$ by $ f(x) = (x.u)^{\beta}$. For $s >0$ to be fixed later, let $\phi : \dC \to \dC$ be equal to one for $|z| \geq 2 s$, equal to zero for $|z| \leq s$ and growing linearly with the modulus in between. Thus $\phi$ is Lipschitz with constant $1/s$. We write, 
$$
f(x) = ( 1 - \phi(x.u) ) f(x) +  \phi(x.u) f(x) = f_1 (x) + f_2(x). 
$$
The function $ f_1$ is bounded by $(2s)^{\beta}$. Let $t >0$, we set $s$ such that $(2 s)^{\beta} = t/4$. We get 
\begin{eqnarray}
 \dP \PAR{\ABS{ \underline \gamma (u) } \geq t } & = & \dP \PAR{ \ABS{\frac 1 n \sum_{k=1}^n f (h_k)  - \dE \frac 1 n \sum_{k=1}^n f (h_k)   } \geq t  }  \nonumber   \\
 & \leq & \dP \PAR{ \ABS{\frac 1 n \sum_{k=1}^n f_2 (h_k)  - \dE \frac 1 n \sum_{k=1}^n f_2 (h_k)   } \geq t / 2   }   \nonumber \\
 &\leq & 2 \exp \PAR{ - c n \eta^2 t^{2/\beta} } \label{eq:cgamma},
\end{eqnarray}
where the last line follows from Lemma \ref{le:concres} and that $f_2$ is Lipschitz with constant $c s^{\beta-1} = c' t^{1-1/\beta}$. 

Similarly,  for $\veps \in \{1,i\}$, we have 
$$
\partial_\veps \gamma(u) = \frac{1}{n} \sum_{k=1}^n \beta (h_k.u)^{\beta -1} (h_k. \veps). 
$$
We consider the function $g \in \cH_\beta$ defined for $x \in \cK_1$ by $ g(x) = (x. u) ^{\beta-1} (x.\veps)$. As in the proof of Lemma \ref{le:gammaholder}, we decompose
$
g(x) = g_1 (x) + g_2(x)
$
with  $g_1$  bounded by $c_0 s^{\beta} / |i.u|$ and $g_2$  Lipschitz with constant $c s^{\beta-1} / |i.u|$.  Now, for fixed $t >0$, we choose $s>0$ such that $c_0 s^{\beta} / |i.u| = t/4$. Using Lemma \ref{le:concres}, we find \begin{eqnarray*}
 \dP \PAR{\ABS{ \partial_\veps \underline \gamma(u)  } \geq t } & \leq & \dP \PAR{ \ABS{\frac 1 n \sum_{k=1}^n g_2 (h_k)  - \dE \frac 1 n \sum_{k=1}^n g_2 (h_k)   } \geq t / 2   } \\
 &\leq & 2 \exp \PAR{ - c n \eta^2 t^{2/\beta} |i.u|^{2/\beta}}.
\end{eqnarray*}
Hence, for any $t \geq 0$,
\begin{equation} \label{eq:cpartgamma}
 \dP \PAR{\ABS{ (i.u)^{1 - \beta + \delta} \partial_\veps \underline \gamma (u)  } \geq t } \leq  2 \exp \PAR{ - c n \eta^2 t^{2/\beta}   |i.u|^{2 - 2 \delta / \beta}}. 
\end{equation}

In the second and final step of the proof, we use a net argument to obtain the concentration for the norm. If $F$ is a $s$-net of $S^1_+$ with $2^{\beta/2} \eta^{-\beta} s^{\beta}  = t /4$,  from \eqref{eq:holdergamma1b}, we deduce that for any $t\geq 0$,  
$$
 \dP \PAR{\sup_{u \in S^+_1}\ABS{   \underline \gamma (u)} \geq t } \leq  \dP \PAR{\sup_{u \in F} \ABS{  \underline \gamma (u)  } \geq t /2}.
$$
There exists a net $F$ of cardinal bounded by $1/s$. From the union bound and \eqref{eq:cgamma}, we deduce that 
$$
 \dP \PAR{  \sup_{u \in S^+_1} \ABS{\underline \gamma (u)} \geq t } \leq  \frac{c}{\eta t ^{1/\beta} }  \exp \PAR{ - c n \eta^2 t^{2/\beta} }.
$$
Similarly, consider $F$ is a $s$-net of $S^1_+$  of cardinal at most $1/s$ such that $|i.u| \geq s/2$ for all $u \in F$. Set $\beta' = \beta - \delta$. From \eqref{eq:holdergamma2},  if  $\eta^ {- \beta} (s / (s/2)^a )^{1-\beta'}  + \eta^{-\beta} (s / (s/2)^{2 - 1/a} )^{1-\beta}    = t /4$, 
 we have that for any $t\geq 0$,  
$$
 \dP \PAR{\sup_{u \in S^+_1}\ABS{   (i.u)^{1 - \beta'} \partial_\veps \underline \gamma (u)} \geq t } \leq  \dP \PAR{\sup_{u \in F} \ABS{ (i.u)^{1 - \beta'} \partial_\veps \underline \gamma (u) } \geq t /2}.
$$
Observe that $s^{\beta - \beta'} = s^\delta = c t \eta^{\beta}$. From the union bound and \eqref{eq:cpartgamma}, we deduce that 
$$
 \dP \PAR{  \sup_{u \in S^+_1} \ABS{ (i.u)^{1 - \beta'} \partial_\veps \underline \gamma (u)} \geq t } \leq  c( \eta t ^{1/\beta} )^{-\beta / \delta} \exp \PAR{ - c n \eta^{2 \beta / \delta}  t^{2 / \delta} }.
$$
It concludes the proof. 
\end{proof}

The following lemma is a variant of the preceding statement. 
\begin{lem}\label{le:concres2}
Let $0 < \alpha < 2$,  $\beta \in (0,1)$, $(h_k)_{1 \leq k \leq n} \in \cK_1^n$ and $(g_k)_{1 \leq k \leq n}$ be iid standard normal variables. We define $\gamma \in \bar \cH^0_\beta$ by 
$$
\gamma (u) = \frac 1 n \sum_{k=1}^n \PAR{ h_k . u }^{\beta} |g_k|^{\alpha}. 
$$
Assume that for all $k\geq 1$, $|h_k| \leq \eta^{-1}$. Then,  there exists a constant $c = c(\alpha,\beta)$ such that for any $0 < \delta < \beta $ and $0 \leq t \leq \eta^{-\beta} $, 
$$
 \dP \PAR{ \| \gamma - \dE \gamma \|_{1 - \beta + \delta}  \geq t } \leq   c (\eta^\beta t  )^{-\frac 1   \delta} \exp \PAR{ - c n (\eta^{\beta} t) ^{\frac 2  \delta}  } .
$$
\end{lem}
\begin{proof}
We start by a preliminary concentration inequality. For $0 < \alpha < 2$, the variable $|g_k|^\alpha$ is sub-exponential. From Bernstein's inequality, for any $t \geq 0$, we have 
$$
\dP   \PAR{ \ABS{ \sum_k x_k |g_k|^\alpha  - \dE  \sum_k x_k |g_k|^\alpha  }  \geq t } \leq 2 \exp \PAR{ - c \frac{t^2}{\|x \|_ 2^2} \wedge  \frac{t}{\|x \|_ \infty }}. 
$$
Since $\| x \|_2 \leq \sqrt n \| x \|_\infty$,  if $\varphi(t) = t \wedge t^2$, 
\begin{equation}\label{eq:bernstein}
\dP   \PAR{ \ABS{ \sum_k x_k |g_k|^\alpha  - \dE  \sum_k x_k |g_k|^\alpha  }  \geq  n t } \leq 2 \exp \PAR{ - c  n \varphi \PAR{\frac{t}{\|x \|_ \infty}} }. 
\end{equation}
In particular, for any $u \in S_1 ^+$, 
$$
 \dP \PAR{ \ABS{ \gamma (u)  - \dE \gamma (u)} \geq t } \leq 2 \exp \PAR{ - c n \varphi(t \eta^\beta)},
$$
and for $\veps \in \{1, i\}$, from \eqref{eq:hubound2},
$$
 \dP \PAR{ \ABS{ \partial_\veps \gamma (u)  - \dE \partial_ \veps \gamma (u)} \geq t } \leq 2 \exp \PAR{ - c n \varphi (t \eta^\beta |i.u| ) }\,.
$$
We  then repeat the net argument used in Lemma \ref{le:concres}. Setting, $L = \frac 1 n \sum_k |g_k|^\alpha$, the inequalities \eqref{eq:holdergamma1}- \eqref{eq:holdergamma2} are replaced respectively by 
$$
|\gamma(u) - \gamma (v)| \leq L 2^{\beta/2} \eta^{-\beta} |u - v|^{\beta}. 
$$
and 
$$
| (i.u)^{1- \beta'} \partial_\veps \gamma(u)  - (i.v)^{1- \beta'} \partial_\veps  \gamma(v) | \leq L  \eta^{-\beta} \PAR{\frac{|u - v|}{(|i.u| \vee |i.v|)^a } }^{1- \beta'} + L \eta^{-\beta} \PAR{\frac{|u - v|}{(|i.u| \vee |i.v|)^{2 - 1/a} } }^{1- \beta}
$$
Arguing as in Lemma \ref{le:concres}, we find that for any  $t \leq  \eta^{-\beta}$, 
$$
 \dP \PAR{ \| \gamma - \dE \gamma \|_{1 - \beta + \delta}  \geq t \ell } \leq  c (\eta^\beta t  )^{-\frac 1   \delta} \exp \PAR{ - c n (\eta^{\beta} t) ^{\frac 2  \delta}  } + \dP ( L \geq \ell). 
$$
From Bernstein's inequality, for $\ell = 2 \dE |g_1|^\alpha$, $\dP ( L \geq \ell) \leq \exp ( - c n)$. 
\end{proof}

We conclude this subsection with a  perturbation inequality for the resolvent. 
\begin{lem}\label{le:pertnorm}
Let $\beta \in (0,1]$, $B,B' \in M_n(\dC)$ be hermitian matrices, $z = E + i \eta$ and $R = (B-z)^{-1}$ and $R'= (B'-z)^{-1}$.  Then,  
$$
\sum_{k =1}^n  |   R_{kk} -    R'_{kk}  |^ \beta \leq 2  n ^{1- \beta} \eta^{-\beta} \rank(B-B'). 
$$
\end{lem}
\begin{proof}
It is a variant of the proof of {\cite[Equation (91)]{BG}}. The resolvent identity asserts that 
  $
 M =  R - R' = R ( B' - B ) R'.
  $
 It follows that  $r = \rank(M)\leq \mathrm{rank}(B - B')$. We notice also that $ \|M \| \leq 2 \eta^{-1}$. Hence, in the singular value decomposition of $M = U D V$, at most $r$ entries of $D = \mathrm{diag} ( s_1, \cdots, s_n)$ are non zero and they are bounded by $\| M \|$.  We denote by $u_1, \cdots, u_r$ and $v_1, \cdots, v_r$ the associated orthonormal vectors so that 
  $$
  M = \sum_{i = 1}^r s_i u_i v_i ^ * ,
  $$
and
  \begin{eqnarray*}
  \left| R _{kk}  -    R'_{kk}  \right|  = |M_{kk}|  =     \left| \sum_{i=1}^r s_i \langle u_i , e_k \rangle  \langle v_i , e_k \rangle   \right|  \leq \| M \|   \sum_{i=1}^r |\langle u_i , e_k \rangle | | \langle v_i , e_k \rangle |. 
  \end{eqnarray*}
 From the subadditivity of $x \mapsto |x|^ {\beta}$, we get 
  \begin{eqnarray*}
  \left| R _{kk}  -    R'_{kk}  \right|^ \beta   \leq \| M \| ^ \beta  \sum_{i=1}^r |\langle u_i , e_k \rangle |^ \beta | \langle v_i , e_k \rangle |^ \beta. 
  \end{eqnarray*}
  Finally,  from H\"older inequality, 
   \begin{eqnarray*}
    \sum_{k=1}^n    \left| R _{kk}  -    R'_{kk}  \right| ^\beta \leq \|M \|^\beta \sum_{i= 1} ^r  n^{1 - \beta} \PAR{   \sum_{k=1}^n   |\langle u_i , e_k \rangle |^{2 } }^{\beta /2} \PAR{  \sum_{k=1}^n | \langle v_i , e_k \rangle |^{2 }  }^{\beta / 2} =  r \|M \|^\beta n^{1- \beta}. 
    \end{eqnarray*}
  It completes the proof. \end{proof}

\subsection{Properties of $\alpha$-stable variables}

In this subsection, we let $(X_k)_{1 \leq k \leq n}$ be iid  symmetric $\alpha$-stable random variables with distribution $\stab_\alpha(0,\sigma)$ for some $\sigma >0$ and $0 \leq \alpha < 2$. More precisely, for all $t \in \bR$,
\begin{equation}
\label{eq:stablefourier}
\bE \exp (it X)  = \exp\left[  - \sigma^\alpha | t|^\alpha ~\right]
\end{equation}

Our first lemma is a deviation inequality for quadratic forms with heavy tails (in a slightly modified form). 
\begin{lem}[{\cite[Lemma 3.2]{BG}}]\label{le:offdiag}
Let $B \in M_n (\dC)$. For any $0 < \alpha < 2$, there exists a constant $c = c(\alpha) >0$ such that for $n \geq 2$ and $t \geq 0$,
$$
\bP \left( \left| n ^{-\frac 2 \alpha}  \sum_{1 \leq k \ne \ell \leq n } X_{k} X_{\ell}  B_{k\ell}\right| \geq  \sigma^2\sqrt{ \frac {\Tr  ( BB^* ) } { n^2 } }  t \right)  \leq c t^{ - \alpha }   \log \left(n  ( t  \vee 2)  \right)\log      (t \vee 2)   . 
$$
\end{lem}

We also use the following identity (which is special case of a more general distributional identity).

\begin{lem}[{\cite[Corollary B.2]{BG}}]\label{cor:formulealice}
Let $(h_k)_{1 \leq k \leq n} \in \cK^n_1$. Then 
$$
\bE \exp \left(  - \sum_{k=1}^n h_k X_k^2 \right) = \bE \exp \left( -2^{\frac \alpha 2} \sigma^\alpha \sum_{k=1}^n h_k ^{\frac \alpha 2} |g_k|^\alpha \right),
$$ 
where $ (g_1, \cdots, g_n)$ is a standard gaussian vector $N( 0 , I ) $. \end{lem}

In  the proof of Proposition \ref{prop:afpi}, we shall use the following key  corollary.  

\begin{cor}\label{cor:stg}
Let $(h_k)_{1 \leq k \leq n} \in \cK^n_1$ and $z \in \dC_+$. Define the function $f \in \bar \cH^0_{\alpha/2}$,   
\begin{eqnarray*}
f(u) & =&\Gamma \left( 1 - \frac \alpha 2 \right) \bE \left(    \left( -iz +n^{-\frac 2 \alpha} \sum_{k=1}^n X_{k}^2 h_k \right)^{-1} . u  \right)^{\frac \alpha 2 }.
\end{eqnarray*}
and for $p\ge 0$,  the scalars
$$\zeta  = \bE \left|    -iz +n^{-\frac 2 \alpha}  \sum_{k=1}^n X_{k}^2 h_k  \right|^{-p} \AND \xi  =  \bE \left(    -iz +n^{-\frac 2 \alpha}   \sum_{k=1}^n X_{k}^2 h_k  \right)^{-p}.
$$
We have $f = \dE G_z ( Z)$, $\zeta = \dE r_{p,z} (Z)$ and $\xi = \dE s_{p,z} (Z')$ with 
$$Z  (u) = 2^{\frac \alpha 2} \sigma^{\alpha} \times \frac{1}{n} \sum_{k=1}^n (h_k.u)^{\frac \alpha 2} |g_k|^{\alpha}  \AND Z'= 2^{\frac \alpha 2} \sigma^{\alpha} \times \frac{1}{n} \sum_{k=1}^n h_k^{\frac \alpha 2} |g_k|^{\alpha} ,$$
where $(g_1, \cdots, g_n)$ is a standard Gaussian vector $N(0,I)$.  
\end{cor}

\begin{proof}
The proof is again essentially contained in \cite{BG}, we reproduce it. We set $h = -iz $, $\Re(h) > 0$. We have 
$$
f(u) = 
 \Gamma \left( 1 - \frac \alpha 2 \right) \bE \left(  \frac{ h. \check u + n^{-\frac 2 \alpha}  \sum_{k=2}^n X^2_{1k} h_k. \check u  }{ \left| h +n^{-\frac 2 \alpha}  \sum_{k=2}^n X_{1k}^2 h_k \right|^2 } \right)^{\frac \alpha 2 } .
$$
We use the formulas, for all $w \in \cK_1$, $\beta > 0$,
\begin{eqnarray*}
| w |^{-2 \beta} =   ( w )^{-\beta}  ( \bar w )^{-\beta}   & = & \Gamma ( \beta ) ^{-2} \int _{ [ 0,\infty) ^2} dx dy \, x^{\beta - 1}  y^{\beta - 1}  e^{ - x w - y \bar w} \\
& = & \Gamma ( \beta ) ^{-2} 2^{ 1- \beta}  \int_0 ^ {\frac \pi  2} d\theta  \sin ( 2 \theta) ^{ \beta - 1}  \int_0 ^ \infty dr  \, r^{2\beta - 1}  e^{ -r  w. e^{ i \theta}} .
\end{eqnarray*}
and for $0 < \beta <  1$,
\begin{eqnarray*}
 w ^{\beta}  & = & \beta  \Gamma (  1- \beta )^{-1} \int _{0}^\infty dx  \, x^{-\beta - 1}  (1 - e^{ -x w} ). 
 \end{eqnarray*}
With $h = -iz$, we find that $f( u) $ is equal to 
 \begin{align*}
& c_\alpha   \int_0 ^ {\frac \pi  2} d\theta  \sin ( 2 \theta) ^{ \frac \alpha 2 - 1} \int _{0}^\infty dx  \,  x^{-\frac \alpha 2 - 1} \\
& \quad \quad \times \int_0 ^ \infty dr  \, r^{\alpha - 1}  \bE \left( e^ { - r h. e^{ i \theta}  - n^{-\frac 2 \alpha}  \sum X_k^2  r  h_k . e^{ i \theta}  } - e^ { -   h. ( r e^{ i \theta} + x \check u ) -  n^{-\frac 2 \alpha}  \sum X^ 2_k h_k. ( r  e^{ i \theta} + x  \check u ).  } \right). 
\end{align*}
The above integrals are absolutely integrable since  $\Re(h) >0$. It remains to perform the change of variable $x = r y$ and use Corollary \ref{cor:formulealice}.

Similarly, with $h = -iz$,the above formula for $|w|^{-2 \beta}$ with $\beta  = p/2$ asserts that 
$$
\zeta  =  \Gamma ( p/2 ) ^{-2} 2^{ 1 - p/2}  \int_0 ^ {\frac \pi  2} d\theta   \sin ( 2 \theta) ^{p/2 -1} \int_0 ^ \infty dr  r^{p-1} \dE e^ { - r h. e^{ i \theta}  - n^{-\frac 2 \alpha}  \sum X_k^2  r  h_k . e^{ i \theta}  }  .
$$
We then use Corollary \ref{cor:formulealice} and that $\Gamma(1/2)  = \sqrt \pi$. The proof of the last statement is identical. \end{proof}

The inverse of a non-negative $\alpha/2$-stable random variable has a light tail.
\begin{lem}[{\cite[Lemma B.3]{BG}}]\label{le:inversetailS}
Let $S$ be a non-negative $\alpha/2$-stable variable with Laplace transform, $t \geq 0$, $\dE \exp ( - t S) = \exp ( - t^{\alpha/2} )$. Then, there exists $c >0$ such that $\dE \exp ( c S^{ - \alpha/ (2 - \alpha)}) < \infty$. In particular, for some $C >0$, for any $t \geq 1$, $\dE S^{-t} < (C t^{\alpha/(2-\alpha)} )^t$.  
\end{lem}

\subsection{Proof of Proposition  \ref{prop:afpi}}

We first introduce the variable  which depend  on $n$ and $z \in \dC_+$,
$$
M_z = \frac {\tr R^{(1)} (z){R^{(1)}}(z)^*} {n^2}   = \frac{ \tr \Im (R^{(1)}(z) )}{n^2 \Im z}.
$$
We recall that $\gamma_z \in \bar \cH^0_{\alpha/2}$ was defined for $u \in S_1^+$ by
$$
\gamma_z (u) = \Gamma \PAR{1 - \frac \alpha 2} \times \frac 1 n \sum_{k=1}^n \PAR{ -i R_{kk}(z) . u }^{\frac \alpha 2}. 
$$
For short notation, as in the statement of Proposition \ref{prop:afpi}, their expectations are denoted by 
$$
\bar M_z = \dE M_z \;  \AND  \bar \gamma_z = \dE \gamma_z. 
$$ 
Finally, we set $\eta = \Im(z)$ and
$$
\widetilde \gamma_z =   \bar \gamma_z (1)  =  \Gamma ( 1 - \frac \alpha 2) \dE  ( -i R_{kk}(z) )^{\frac \alpha 2}.
$$

We will use that the resolvent $R$ and $R^{(1)}$ are close. For example, 
we observe that 
\begin{equation}\label{eq:pertM}
\ABS{ M_z - \frac {  \tr \Im (R(z) )} {n^2 \eta}} \leq \frac{3}{n^2 \eta^2} \AND \ABS{ \bar M_z - \frac {\dE \Im R_{11}(z) } {n \eta}  } \leq \frac{3}{n^2 \eta^2},
\end{equation}
(we apply Lemma  \ref{le:pertnorm} to the matrix $B'$ given by $B'_{ij} = B_{ij} \IND ( i , j \geq 2)$. Its resolvent $R'_{ij}$ coincides with $R^{(1)}_{ij}$ for all $i,j \geq 2$ and $R'_{11} = -1/z$).   We will also need a technical lemma. 

\begin{lem}\label{le:momo}
If $h \in \cK_1$ and $0 < \beta < 1$, 
$
 \Re (h^\beta) \geq   (\Re h)^\beta. 
$
\end{lem}
\begin{proof}
We may assume without loss of generality that $|h| = 1$ and $\Im (h) \geq 0$. Then the lemma is equivalent to the inequality, for any $x \in [0,\pi/2]$,   
\begin{equation*}\label{eq:momo}
\cos ( \beta x) \geq (\cos x)^{\beta}. 
\end{equation*}
which can easily be checked by showing that $f(x)=x^{-1}\log \cos(x)$ is decreasing (for instance by showing that $f'(0)<0$ and $f''(x)\le 0$).
\end{proof}

As explained in subsection \ref{subsec:exfpe}, the approximate fixed point for $\gamma$ will come from Schur's complement formula \eqref{eq:schurcf0}. We recall that
\begin{eqnarray}
R_{11}(z) & = & - \PAR{ ih   + i Q_z + T_z }^{-1},  \label{eq:schurcf}
\end{eqnarray}
where we have set $h  = -i z \in \cK_1$, $H_{k}=-i R^{(1)}_{kk}(z)$ and
\begin{eqnarray*}
Q_z & = & n^{-\frac 2 \alpha}  \sum_{2 \leq k   \leq n} X^2_{1k}  H_k  \\
T_z & = &  n^{-\frac 1 \alpha}  X_{11} + n^{-\frac 2 \alpha}  \sum_{2 \leq k \ne \ell \leq n} X_{1k} X_{1 \ell} R_{k\ell}^{(1)} (z).
\end{eqnarray*}
We  introduce the function in $\cH_{\alpha/2}$, 
\begin{eqnarray*}
I_z(u)  & = & \Gamma \left( 1 - \frac \alpha 2 \right) \bE \left(    \left( h + Q_z \right)^{-1} . u  \right)^{\frac \alpha 2 } 
\end{eqnarray*}
We will often drop the explicit dependence in $z$. 
Finally, $\cF$ is the $\sigma$-algebra generated by the variables $(X_{ij})$, $i, j \geq 2$. Note that $Q$ and $T$ are $\cF$-measurable and $(X_{1k})$, $k \geq 1$, is independent of $\cF$.

The proof of Proposition \ref{prop:afpi} is divided in four steps. 
\paragraph{Step 1 :  from $I_z$ to $G_z ( \bar \gamma)$. }  From Corollary \ref{cor:stg}, we have 
\begin{equation}\label{eq:ItoZ}
I_z  = \dE G_z (Z).
\end{equation}
with $Z$ given by
$$Z  (u) = \Gamma \PAR{ 1 - \frac \alpha 2} \times \frac{1}{n} \sum_{k=2}^n (H_k.u)^{\frac \alpha 2} \frac{|g_k|^{\alpha}}{\dE |g_k|^\alpha} ,$$
where $(g_1, \cdots, g_n)$ is a standard Gaussian vector $N(0,I)$ independent of $\cF$. Indeed, with our choice $\sigma^\alpha =  \pi /   ( 2 \sin(\pi \alpha / 2)\Gamma(\alpha))$ in \eqref{eq:stablefourier}, we have 
$$
2^{ \alpha / 2} \sigma^{\alpha}  =  2^{\alpha/2-1}  \Gamma ( \alpha/2) \Gamma ( 1 - \alpha/2 ) / \Gamma( \alpha)   = \Gamma ( 1 - \alpha /2) / \dE |g_1|^\alpha, 
$$
where we used the classical identities $\dE |g_1|^\alpha = 2^{\alpha/2} \Gamma ( 1 / 2 + \alpha / 2) / \sqrt \pi$ and, for $0 < \beta < 1$, $\Gamma ( \beta ) \Gamma ( 1 - \beta ) = \pi / \sin ( \pi \beta)$, $\Gamma(\beta) \Gamma ( 1/2 + \beta) = \sqrt \pi 2^{1 - 2\beta} \Gamma ( 2 \beta)$.  
We set 
$$
\bar Z(u) = \dE_\cF Z(u) = \Gamma \PAR{ 1 - \frac \alpha 2} \times \frac{1}{n} \sum_{k=2}^n (H_k.u)^{\frac \alpha 2}. 
$$
From Lemma \ref{le:concres2}, with overwhelming probability,  we get
$$\| Z - \bar Z  \|_{1 - \alpha/2+ \delta} \leq (\log n)^c \eta^{-\alpha/2} n^{- \delta /2}.$$   
Similarly, from Lemma \ref{le:concnorm}, with overwhelming probability, we have 
$$
\|  \gamma  - \bar \gamma \|_{1 - \alpha/2 + \delta} \leq  (\log n)^c  \eta^{-\alpha/2} n^{- \delta /2}.  
$$
We consider the matrix $B'$ given by $B'_{ij} = B_{ij} \IND ( i , j \geq 2)$. Its resolvent $R'_{ij}$ coincides with $R^{(1)}_{ij}$ for all $i,j \geq 2$ and $R'_{11} = -1/z$. Using Lemma \ref{le:gammaholder},  as $\bar Z$ equals $\gamma$ up to replacing $R$ by $R^{(1)}$,
$$
\|  \bar Z   - \gamma \|_{1 - \alpha/2+ \delta} \leq    c \eta^{- \alpha/2} n^{-1}  \PAR{ 1 + \sum_{k=1}^n |R_{kk}(z) - R'_{kk}(z)|^{\alpha/2}  + \eta^{\delta} \sum_{k=1}^n |R_{kk} - R'_{kk}|^\delta}.
$$
We apply Lemma  \ref{le:pertnorm} to the matrix $B'$ and use the triangle inequality, it gives
$$
\|  \bar Z - \gamma  \|_{1 - \alpha/2 + \delta} \leq   c \eta^{-\alpha/2} n ^{-1} + c \eta^{-\alpha} n^{-\alpha/2}  + c  \eta^{-\alpha/2} n^{- \delta} \leq c \eta^{-\alpha/2} n^{- \delta  }. 
$$
So finally,  with overwhelming probability, we deduce
\begin{equation}\label{eq:koko}
\|   Z - \bar \gamma \|_{1 - \alpha/2 + \delta}  \leq  (\log n)^c \eta^{-\alpha/2} n^{- \delta /2 }. 
\end{equation}
Observe that the right hand side is $o(1)$ for our range of $\eta$ (we can take $\delta$ as close from $\alpha/2$ as wished).  The assumption  $\dE | R_{11} (z)| \leq \veps^{-1}$ and Lemma \ref{le:gammaholder} imply that for some $c>0$, $\| \bar \gamma \|_{1 - \alpha/2 + \delta} \leq c$ .  On the other end, Lemma \ref{le:gammaholder} also implies the rough bound $\|Z\|_{1 - \alpha/2 + \delta} \leq L \eta^{-\alpha/2}$ with $L = (c/n) \sum_{k=2}^n |g_k|^{\alpha}$.  For any $p \geq 0$,   $\dE L^p $ being bounded, we deduce from \eqref{eq:koko} that
\begin{equation}\label{eq:koko2}
 \PAR{  \dE \|   Z - \bar \gamma \|^p _{1 - \alpha/2 + \delta} }^{1/p} \leq (\log n)^c \eta^{-\alpha/2} n^{- \delta /2 }  \AND \dE \|   Z  \|^p _{1 - \alpha/2 + \delta}  \leq  c . 
\end{equation}
 We now claim for that for some $\veps' >0$, with overwhelming probability, for all $u \in S_1^+$
\begin{equation}\label{eq:minreal}
\Re (Z.u) > \veps' \AND \Re (\bar \gamma.u) > \veps'
\end{equation}
(that is $Z, \bar \gamma \in \cH_{\alpha/2,\delta} ^{\veps'}$). First, since  $\| Z - \bar \gamma \|_\infty \leq \|   Z - \bar \gamma \|_{1 - \alpha/2 + \delta} = o(1)$, it suffices to prove that $\bar \gamma \in \cH_{\alpha/2,\delta} ^{\veps'}$ up to modifying the value of $\veps'$. We should thus check that (for some new $\veps' >0$), for all $u \in S_1^+$
$$
\Re  \dE (-iR_{11} .u)^{\alpha/2}  > \veps' .
$$  
The assumption $\dE  ( \Re ( - i  R_{11} )) ^{\alpha/2} \geq \veps$ implies that for any $u \in S_1^+$, $\dE  ( \Re ( - i  R_{11} . u )) ^{\alpha/2} \geq \veps$ (since $\Re(h.u) \geq \Re(h)$). Using Lemma \ref{le:momo}, we deduce that $\Re \dE   ( - i  R_{11} . u ) ^{\alpha/2} \geq \veps$. It proves \eqref{eq:minreal}.

We now want to apply \eqref{eq:ItoZ}. Let $V$ be the event of overwhelming probability such that \eqref{eq:minreal} holds. On $V$, we apply Lemma \ref{le:fplip}, while on $V^c$, we use Lemma \ref{le:deffp}\eqref{le:deffp2}. It gives  
\begin{eqnarray}\label{eq:fstep1}
\| I_z - G_z (\bar \gamma) \|_{1 - \alpha/2+ \delta} & \leq & \dE \| G_z(Z) - G_z (\bar \gamma) \|_{1 - \alpha/2 + \delta} \nonumber\\
&\leq & c  \dE  ( 1 + \| \bar \gamma \|_{1 - \alpha/2+ \delta}  + \| Z \|_{1 - \alpha/2+ \delta} )  \| Z - \bar \gamma \|_{1 - \alpha/2 + \delta}  \nonumber \\
&& \quad \quad   + c  \eta^{-\alpha}   \dE \IND_{V^c} \PAR{ \| Z \|_{1 - \alpha/2 + \delta} + \| \bar \gamma \|_{1 - \alpha/2 + \delta} } \nonumber\\
& \leq & (\log n)^c \eta^{-\alpha/2} n^{- \delta /2}\,,
\end{eqnarray}
where the last inequality follows from \eqref{eq:koko2} and Cauchy-Schwartz inequality.

\paragraph{Step 2 :  bounds on $Q$ and $T$. } Recall that $Z (e^{i \pi/4}) =  (c / n) \sum_{k=2}^n \Re (H_k)^{\frac \alpha 2}|g_k|^{\alpha}$.  From Lemma \ref{cor:formulealice}, given $\cF$, with $Q=Q_{e^{i\pi/4}}$, we have
\begin{equation}\label{eq:eddis2}
\Re(Q) \stackrel{d}{=}  c S [ Z(e^{i\pi/4})]^{2/\alpha}, 
\end{equation}
where $S$ is a non-negative $\alpha/2$-stable variable independent of $Z$. By Lemma \ref{le:inversetailS}, for $a > 0$ and $b > 0$ small enough,
$$
\dP ( S \leq (\log n)^{-a} ) = \dP \PAR{ e^{  b S^{-\frac{\alpha}{2-\alpha}}} \geq e^{ b (\log n) ^{\frac{a \alpha}{2-\alpha}}} }\leq  c  e^{ - b (\log n) ^{\frac{a \alpha}{2-\alpha}}}.
$$
Hence, if $a \alpha / (2 - \alpha)  > 1$, the event $\{ S \leq (\log n)^{-a} \}$ holds with overwhelming probability. Putting together  \eqref{eq:koko} and \eqref{eq:eddis2}, we deduce that for some $c_0 > 0$, with overwhelming probability, 
\begin{equation}\label{eq:reqreq}
 \Re (Q) \geq 2 (\log n)^{-c_0}.  
\end{equation}
Similarly,  from Lemma \ref{cor:formulealice}, given $\cF$,
\begin{equation}\label{eq:coupSQ}
|Q| \leq c n^{-\frac 2 \alpha}  \sum_{k=2}^n |H_k| X_{1k}^{2}   \stackrel{d}{=}  c S Y^{2/\alpha}, 
\end{equation}
where, given $\cF$, $S$ is as above and independent of $Y$ given by  
$$
Y  =    \frac{1}{n} \sum_{k=2}^n |H_k|^{\frac \alpha 2} \frac{|g_k|^{\alpha}}{\dE |g_k|^\alpha},
$$
with $(g_1, \cdots, g_n)$ iid standard normal variables. We set 
$$
\bar Y = \dE_{\cF} Y =  \frac{1}{n} \sum_{k=2}^n |H_k|^{\frac \alpha 2}.
$$
From \eqref{eq:bernstein}, with overwhelming probability, 
$$
|Y - \bar Y| \leq (\log n) ^c n^{-1/2} \eta^{-\alpha/4} = o(1). 
$$
By Lemma \ref{le:pertnorm}, 
$$
|\bar Y - \rho | \leq c n^{-\alpha/2} \eta^{-\alpha/2} = o(1). 
$$
where 
$$
\rho = \frac 1 n \sum_{k=1}^n |R_{kk}|^{\alpha/2}. 
$$
Let $\bar \rho =\dE \rho =  \dE |R_{11}|^{\alpha/2}$. The proof of  \eqref{eq:cgamma} gives for any $t \geq 0$, 
$$
\dP (| \rho - \bar \rho| \geq t ) \leq 2 \exp ( - c n \eta^{2} t^{4/\alpha} ). 
$$
Hence, with overwhelming probability, 
$$
|\rho - \bar \rho | \leq (\log n)^c n^{-\alpha/4} \eta^{-\alpha/2} = o(1).
$$
By assumption $\bar \rho \leq \veps^{-1}$. It follows that with overwhelming probability,
$$
|Y| \leq 2 \veps^{-1}.
$$
We denote by $F$ the event 
$$
\{ |Y| \leq c_0 \hbox{ and } \Re (Q) \geq 2  (\log n)^{-c_0} \}.
$$
From what precedes, if $c_0$ is large enough, the event $F$ holds with overwhelming probability.  We also consider the event : 
$$
E = \BRA{ |T| \leq (\log n)^{-c_0}  } 
$$ 
Since $M_z \leq 1 /( n \eta^{2}) \leq 1$, by Lemma \ref{le:offdiag}, if $0 \leq t \leq 1$
\begin{equation}\label{eq:Ttail}
\dP_\cF  ( |T| \geq t ) \leq c M^{\alpha/2} t^{-\alpha} ( \log n )  ,
\end{equation}
and, consequently, for $\beta > \alpha$ and $0 \leq t \leq 1$, for some constant $c = c(\alpha)$, 
\begin{equation}\label{eq:Ttail2}
\dE_\cF  |T|^\beta \IND ( |T| \leq t ) \leq \frac{c}{ \beta - \alpha} M^{\alpha/2}   t^ {\beta - \alpha} ( \log n  ).
\end{equation}
Thus by Chebychev's inequality for some $c>0$,
\begin{equation}\label{eq:Ecomp}
\dP_\cF  ( E^c ) \leq M^{\alpha/2} (\log n )^{c}.
\end{equation}

\paragraph{Step 3 :  from Schur's formula \eqref{eq:schurcf} to $I_z $. }  In \eqref{eq:coupSQ}, up to enlarge our probability space, we may assume that the variables $(Y,S,Q,T)$ are defined on the same probabilisty space. From  \eqref{eq:schurcf}, Jensen's inequality and  Lemma \ref{le:gammaholder}, for any $s > 0$,
\begin{eqnarray}
\| \bar \gamma - I_z \|_{1 - \alpha/2 + \delta} & \leq& c\dE \| ( \PAR{ h   + Q  - i T }^{-1}.u)^{\alpha/2} -  (\PAR{ h   +  Q }^{-1}.u )^{\alpha/2} \|_{1 - \alpha/2+ \delta} \nonumber \\
& \leq &  (\log n)^c \dE S^{\alpha -1} \IND_{ S \leq s} |T| \IND_{E}  +  c \eta^{-\alpha/2} \PAR{  \dP ( S \geq s  ) +  \dP ( E^c \cup F^c)},\label{eq:gege}
\end{eqnarray}
where we have used that on $E \cap F$, $|h+ Q - i T|$ and $| h + Q|$ are $O(S)$ and $\Re (h+ Q - i T) \geq \Re (Q) - |T| \geq (\log n)^{-c_0}$. 

We first consider the case $0 < \alpha \leq 1$ in \eqref{eq:gege} where we take $s = \infty$.  Let $p = \log n$ and $1/ q + 1 / p = 1$. We have $q = 1 + 1/\log n + o(1/\log n) $. From H\"{o}lder  inequality and Lemma \ref{le:inversetailS}
$$
\dE S^{\alpha -1} |T| \IND_{E} \leq  \PAR{ \dE S^{p(\alpha -1)}}^{ 1 / p} \PAR{\dE |T|^q \IND_{E}}^{1 / q}  \leq (\log n)^c  \PAR{\dE |T|^q \IND_{E}}^{1 / q}. 
$$
Recall that $E = \{ |T| \leq (\log n)^{-c_0} \}$. Also Jensen's inequality implies that $\dE M^{\alpha/2} \leq \bar M^{\alpha/2}$. We deduce from \eqref{eq:Ttail2} that 
$$\PAR{\dE |T|^q \IND_{E}}^{1 / q} \leq \PAR{ (\log n)^c  \bar M^{\alpha/2} }^{1/q} \leq  (\log n)^{c} \bar M^{\alpha/2},$$
where the last inequality comes from $\bar M \geq c n^{-1}$ which follows from \eqref{eq:pertM} and the assumptions of Proposition \ref{prop:afpi}. 
Hence, from \eqref{eq:Ecomp}-\eqref{eq:gege}, 
$$
\| \bar \gamma - I_z \|_{1 - \alpha/2 + \delta} \leq  (\log n)^c \bar M^{\alpha/2} + (\log n )^{c} \eta^{-\alpha/2} \bar M^{\alpha/2}\leq (\log n )^{c'} \eta^{-\alpha/2} \bar M^{\alpha/2}.
$$
Similarly, for $1 < \alpha < 2$, we choose $s = M^{-1}$ in \eqref{eq:gege}. Let $p = \alpha / (2 (\alpha -1) )> 1$ and $q= \alpha / ( 2 - \alpha ) > \alpha$. Since $1/ q + 1 / p = 1$, from H\"{o}lder  inequality, we find for $n$ large enough,
$$
\dE S^{\alpha -1} \IND_{ S \leq s}  |T| \IND_{E} \leq \PAR{ \dE S^{\alpha/2} \IND_{ S \leq s}  }^{1/p} \PAR{ \dE |T|^q \IND_{E}} ^{1/q} \leq (\log n) \PAR{ \dE |T|^q \IND_{E}} ^{1/q},
$$
where we have used that $\dE S^{\alpha/2} \IND ( S \leq s) \leq c \log ( s\vee 2)$ and $\bar M \geq c n^{-1}$.  Since $\alpha / (2q) = 1 - \alpha/2$, we get from \eqref{eq:Ttail2}-\eqref{eq:Ecomp}, 
$$
\| \bar \gamma - I_z \|_{1 - \alpha/2 + \delta} \leq (\log n)^c \bar M^{1 - \alpha/2} + (\log n )^{c}  \eta^{-\alpha/2} \bar M^{\alpha/2}.
$$
We deduce from \eqref{eq:fstep1} that the first statement of Proposition \ref{prop:afpi} holds. 

\paragraph{Step 4 :  from Schur's formula \eqref{eq:schurcf} to $\dE |R_{11}|^p$ and $\dE R_{11}^p$. }  
We only treat the case $\dE |R_{11}|^p$. The case of $\dE R_{11}^p$ is identical. Let $J_z = \dE | h + Q |^{-p}$. From Corollary \ref{cor:stg}
\begin{equation*}\label{eq:ABStoZ}
J_z = \dE r_{p,z} (Z),
\end{equation*}
where $Z$ is as above defined in \eqref{eq:ItoZ}.  We drop the parameters $(p,z)$. From Jensen's inequality, 
$$
| J - r (\bar \gamma) | \leq \dE | r (Z) - r (\bar \gamma) |. 
$$
We then argue as in \eqref{eq:fstep1} : when the event $V$ holds, we apply Lemma \ref{le:fplip} and when $V$ fails to hold, we use instead that $| r (g) | \leq c \eta^{-p}$ for $g \in \bar \cH_{\alpha}^0$ (Lemma \ref{le:deffp}).  It gives in conjunction with \eqref{eq:koko2}
\begin{equation}\label{eq:fstep2}
| J - r (\bar \gamma) | \leq c \dE \| Z - r  (\bar \gamma) \|_\infty + c \eta^{-p} \dP (E^c) \leq (\log n)^c \eta^{-\alpha/2} n^{- \delta /2}. 
\end{equation}

We may now repeat the third step. Recall the variable $S$ defined in \eqref{eq:coupSQ} and the events $E,F$.  We observe that for any $x,y \in \dC$, 
$$\ABS{ \frac{1}{|x|^p} - \frac{1}{|y|^p}} \leq \frac{|x- y|}{ |x| ^p |y|^p}  \sum_{k=0}^{p-1} |x|^{k} |y|^{p-k-1}  \leq  p \frac{ |x - y |}{ s ^{p+1}},$$ where $s = \Re (x) \wedge \Re(y)$. Hence, from \eqref{eq:schurcf}, Jensen's inequality,
\begin{eqnarray*}
| \dE |R_{11}|^p - J | & \leq& \dE \ABS{  \ABS{ h   + Q  - i T }^{-p}  -  \ABS{ h   +  Q }^{-p} }\nonumber \\
& \leq &  (\log n)^ {c} \dE  |T| \IND_{E}  +  2 \eta^{-p}   \dP ( E^c \cup F^c),\label{eq:gegen}
\end{eqnarray*}
where we have used that on $E \cap F$, $\Re (h+ Q - i T) \geq \Re (Q) - |T| \geq (\log n)^{-c_0}$. Using  \eqref{eq:Ttail2}-\eqref{eq:Ecomp}, we arrive at 
$$
| \dE |R_{11}|^p - J | \leq (\log n)^c \eta^{-p} \bar M^{\alpha/2}. 
$$
Together with \eqref{eq:fstep2}, it implies the second statement of the proposition.\qed

\subsection{Local law in neighborhood of the origin}

We denote by $$\widetilde \gamma^\star_z =  \gamma^\star_z(1)  =   \Gamma (1 - \alpha/2) \dE (-iR_\star) ^{\alpha/2}, $$ where $R_\star$ was defined by \eqref{eq:RDE}.  Lemma \ref{le:fpgamma} asserts that
$
\widetilde \gamma^\star_z  =  \Gamma (1 - \alpha/2) s_{\alpha/2,z} (  \widetilde \gamma^\star_z ). 
$
We prove the following theorem :
\begin{theorem}\label{th:mainloc}
For $\alpha \in (0,2) \backslash \cA$, there exist $\tau >0$ and $c > 0$ and $0 < \delta < \alpha/2$ such that if $|z| \leq \tau$ and $\Im (z) \geq (\log n)^c n^{ - \alpha / ( 2 + \alpha)}$, then 
$$
\| \bar \gamma_z -  \gamma^\star_z \|_{1 - \alpha/2+ \delta}  = o(1) \AND | \dE |R_{11}| -  r_{1,z}(\gamma^\star_z) |  = o(1).
$$
Finally, if $ p \geq \alpha/2$ is such that $\Im (z) \geq (\log n)^c n^{ - \alpha / ( 2 p + \alpha)}$, 
\begin{eqnarray*}
|\dE |R_{11}(z)|^p - r_{p,z} ( \gamma^\star_z) | & = & o(1) \\
|\dE (-i R_{11}(z))^p - s_{p,z}( \widetilde \gamma^{\star}_z) | & = & o(1).
\end{eqnarray*}
\end{theorem}

We first check that the assumptions of Proposition \ref{prop:afpi} are met for any fixed $z \in \dC_+$.
\begin{lem}\label{le:afpi}
If $ z\in \dC_+$, there exists $c = c(z) >0$ such that $|R_{11}(z)| \leq c^{-1}$ and $\dE \Im ( R_{11} (z))^{\alpha/2} \geq c$. 
\end{lem}
\begin{proof}
Recall that $|R_{11}| \leq 1/ \Im (z)$. For the lower bound, let $X_1  = (X_{1k})_{2 \leq k \leq n}$. Since $\Im ( R^{(1)} ) \geq 0$, we have from \eqref{eq:schurcf} 
$$
\Im ( R_{11} ) =  \frac{\Im (z) + \langle  X_1 ,  \Im R^ {(1)} X_1 \rangle}{| z  + i Q   + T |^2} \geq  \frac{1}{3} \frac{\Im (z)}{| z|^2  + |Q|^2 + |T|^2}. 
$$
By \eqref{eq:Ttail}, since $M \leq n^{-1} \Im (z)^{-2} \leq c/ n$, we have $\dP (|T| \geq 1  ) \leq c n^{-\alpha/2} (\log n) \leq 1/ 4$ for $n$ large enough. Finally, from \eqref{eq:coupSQ}, $\dE |Q|^{\delta} < c$ if $\delta < \alpha/2$. We deduce that $\dP ( |Q| \geq t) \leq 1/4$ for $t$ large enough.  Then, with probability at least $1/2$, $\Im ( R_{11} ) \geq  \Im (z) / ( 3 |z|^2 + 3t + 3)$.  \end{proof}

\begin{proof}[Proof of Theorem \ref{th:mainloc}] Set $r_{z} = r_{1,z}$ and $\eta_{\min}  = (\log n)^{c_0} n^{-\alpha/( 2 + \alpha)}$ for some $c_0>0$ to be chosen later on. Since $\alpha / (2 + \alpha) <1/2$, we may choose $0 < \delta < \alpha/2$ such that for all $\eta \geq \eta_{\min}$ and fixed $c >0$, 
\begin{equation}\label{eq:choixd}
(\log n)^c \eta^{-\alpha/2}  n^{-\delta/2} = o(1).
\end{equation}
We fix such parameter $\delta$.  We use a continuity argument.  Assume that for some $z_1 = E + i \eta_1 \in \dC_+$, with $|z_1| \leq \tau$ and $\eta_1 \geq \eta_{\min}$, we have that
\begin{equation}\label{eq:boots}
\| \bar \gamma_{z_1} - \gamma_{z_1}^\star \|_{1 - \alpha/2 + \delta} \leq \veps \AND | \dE |R_{11}(z_1)| - r_{z_1}( \gamma_{z_1}^\star ) | \leq \veps 
\end{equation}
where $\veps \leq \tau/2$ is an arbitrarily small constant and $\tau$ is as in Proposition \ref{le:fploccont}. Let $z = E + i \eta$ with $\eta_1  - n^{-3/\delta} \leq \eta \leq \eta_1$. Using $\eta \geq 1/ \sqrt n$, we have $| R_{kk} (z)  - R_{kk} (z_1)| \leq |\eta - \eta_1| / \eta^2 \leq 2 n^{-3/\delta +1} \leq 2 n^{-2/\delta}$.  Hence,  by Lemma \ref{le:gammaholder}, 
$$
\| \bar \gamma_{z_1} -  \bar \gamma_{z} \|_{1 - \beta + \delta} \leq c n^{\beta/2} n^{-2 \beta /\delta} + c n^{\beta/2 - \delta/2}  n^{-2} = o(1). 
$$
The same holds for $\|   \gamma^\star_{z_1} -  \gamma^ \star_{z} \|_{1 - \beta + \delta}$. From the triangle inequality, we get 
\begin{equation}\label{eq:huhu}
\| \bar \gamma_{z} - \gamma_{z}^\star \|_{1 - \beta + \delta} \leq 2 \veps \leq \tau.
\end{equation}
Similarly, using the continuity of $z \mapsto r_z (f)$ for $f \in \cH^0_{\alpha/2}$, we get
$$
| \dE |R_{11}(z)| - r_{z}( \gamma_{z}^\star ) | \leq 2 \veps.
$$

By Proposition \ref{le:4IFTapp}, there exists a constant $c_1 >1$ such that for all $z \in \dC$ with $|z| \leq \tau$, 
$$
|r_{z}( \gamma_{z}^\star )| \leq c_1 \AND \gamma^ \star_{z} (e^{i \pi /4} ) \geq c_1^{-1}.
$$
Hence, if $\veps$ is small enough, we deduce that 
$$
 \dE |R_{11}(z)| \leq 2 c_1 \AND  \dE (\Im R_{11}(z))^{\alpha/2} \geq \frac{1}{2c_1 \Gamma (1 - \alpha /2)}. 
$$
We may thus apply Proposition \ref{prop:afpi}. The first inequality of  Proposition \ref{prop:afpi}, in conjunction with Corollary \ref{le:fploccont} and \eqref{eq:huhu}, implies that 
\begin{equation}\label{eq:totom}
\| \bar \gamma_z -  \gamma^\star_z \|_{1 - \alpha/2 + \delta} \leq   (\log n )^{c}   \PAR{    \eta^{-\alpha/2} \bar M^{\alpha/2}  +   \eta^{-\alpha/2} n^{-\delta/2}  + \bar M^{1 -\alpha/2}\IND_{\alpha >1}} .
\end{equation}
The second inequality of  Proposition \ref{prop:afpi} gives thanks to Lemma \ref{le:fplip} and the above inequality gives 
\begin{equation}\label{eq:sosom}
|\dE |R_{11}| - r_z ( \gamma^\star_z) |  \leq  (\log n)^c \PAR{ \eta^{-1} \bar M^{\alpha/2} +  \eta^{-\alpha/2}n^{-\delta/2}  }. 
\end{equation}
From \eqref{eq:pertM}, we have
$$
\bar M  = \frac{\dE \Im (R^{(1)}_{22})}{n \eta}  \leq \frac{ \dE |R_{11}|}{n \eta} + \frac{3}{n^2 \eta^2} \leq \frac{ 2 c_1 +3 }{ n \eta}. 
$$
We deduce from  \eqref{eq:totom}-\eqref{eq:sosom} and \eqref{eq:choixd} that 
$$
\| \bar \gamma_z -  \gamma^\star_z \|_{1 - \alpha/2 + \delta} \leq   (\log n )^{c}     \eta^{-\alpha}  n^{-\alpha/2}   + o(1) \AND  |\dE |R_{11}| - r_z ( \gamma^\star_z) |  \leq   (\log n )^{c}     \eta^{-1-\alpha/2}  n^{-\alpha/2}   + o(1).
$$
It is easy to check that if $c_0$ in choice of $\eta_{\min}$ is small enough, the above terms are $o(1)$. It follows that $z$ satisfies \eqref{eq:boots} with the same $\veps$.

We may thus use of continuity argument, for any fixed $z$ with $|z| \leq \tau$, \eqref{eq:boots} holds by Lemma \ref{le:afpi}, Proposition \ref{prop:afpi} and Lemma \ref{le:fploccont}. By iteration, it proves that \eqref{eq:boots} holds for all $z \in \dC_+$ with $|z| \leq \tau$ and $\Im (z) \geq \eta_{\min}$.  It proves the first statement of Theorem \ref{th:mainloc}. The second statement is then a direct consequence of a  new application of Proposition \ref{prop:afpi}. 
\end{proof}

Theorem \ref{th:mainloc} implies a control of the number of eigenvalues in an interval.  Recall that $|\Lambda_I|$ is the number of eigenvalues of $A$ in the interval $I$. We denote by $\mu_\star$ the probability measure on $\R$ whose Cauchy-Stieltjes transform is $i s_{1,z} (\widetilde \gamma^\star_z)= \dE R_\star$ (defined by \eqref{eq:RDE}) : for any $z \in \dC_+$,
$$
\int \frac{d \mu_\star (\lambda)} { \lambda -z } = i s_{1,z} (\widetilde \gamma^\star_z). 
$$
It is notably shown in \cite{BAG1,BDG09,BCC} that $\mu_{\star}$ has a bounded positive continuous density which is explicit at $0$. 
\begin{cor}
\label{cor:mainloc}
For $\alpha \in (0,2) \backslash \cA$, there exist $\tau >0$ and $c > 0$ such that if $I \subset [-\tau,\tau]$ is an interval of length at  least $(\log n)^c n^{ - \alpha / ( 2 + \alpha)}$, then with overwhelming probability,
$$
\ABS{\frac{|\Lambda_I|}{n}  -   \mu_\star (I) } = o(1).
$$
\end{cor}

\begin{proof}
Define the probability measure $\mu(I) = |\Lambda_I| / n$. By construction, the Cauchy-Stieltjes transform of $\mu$ is $\frac 1 n \tr R$. By \cite[Lemma 3.7]{BG}, up to modifying the value of $c$ and $\tau$, it is sufficient to prove that with overwhelming probability,
$$
\frac 1 n \tr (-i  R_z) = s_{1,z} + o(1)
$$ 
for all $z$ such that $|z| \leq \tau$ and $\Im(z) \geq (\log n)^c n^{ - \alpha / ( 2 + \alpha)}$. By Lemma \ref{le:concres}, with overwhelming probability 
$$
\ABS{ \frac 1 n \tr  (-i R_z) - \dE \frac 1 n \tr ( -i R_z) } \leq (\log n) n ^{-1/2} \eta^{-1} = o(1). 
$$
Since $\dE \frac 1 n \tr ( -i R)  = \dE (-i R_{11} )$, the conclusion follows from  Theorem \ref{th:mainloc}. 
\end{proof}

\begin{cor}
\label{cor:mainloc2}
For $\alpha \in (0,2) \backslash \cA$, there exist $\tau >0$ and $c > 0$ such that if $|z| \leq \tau$,   $\Im(z) \geq (\log n)^c ( n^{ - \alpha / ( 4 + \alpha)} \vee n^{-1/4} ) $, then with overwhelming probability,
$$
\frac{1}{n} \sum_{k=1}^n  |R_{kk} (z) |^{2} \leq c.
$$
\end{cor}

\begin{proof}
Set $z = E + i \eta$. We apply Lemma \ref{le:concres} with $f (x)= |x|^2 \wedge \eta^{-2}$. It is Lispchitz with constant $2 \eta^{-1}$.   We deduce that with overwhelming probability, 
$$
\ABS{ \frac{1}{n} \sum_{k=1}^n  |R_{kk} (z) |^{2}  - \dE |R_{11} (z)|^2 } \leq (\log n) n ^{-1/2} \eta^{-2} = o(1).   
$$
It remains to apply Theorem \ref{th:mainloc}.
\end{proof}

\subsection{Proof of Theorem \ref{th:main}}

Since $\mu_{ \star}$ has  a positive continuous density, there exists $c >0$ such that 
$$
\mu_{\star} (I) \geq c |I|
$$
for any interval $I \subset [-\tau,\tau]$.  Theorem \ref{th:main} is an immediate consequence of Lemma \ref {le:IPRres}, Corollary \ref{cor:mainloc} and Corollary \ref{cor:mainloc2}. \qed


 \bibliographystyle{amsplain}

\bibliography{bib}

\providecommand{\bysame}{\leavevmode\hbox to3em{\hrulefill}\thinspace}
\providecommand{\MR}{\relax\ifhmode\unskip\space\fi MR }
\providecommand{\MRhref}[2]{%
  \href{http://www.ams.org/mathscinet-getitem?mr=#1}{#2}
}
\providecommand{\href}[2]{#2}
\begin{thebibliography}{10}

\bibitem{AT}
Anderson Abou-Chacra, R., P.~W., and Thouless~D. J., \emph{A selfconsistent
  theory of localization.}, J. Phys. C: Solid State Phys. \textbf{6} (1973),
  1734?1752.

\bibitem{AAT}
R.~Abou-Chacra and Thouless~D. J., \emph{A selfconsistent theory of
  localization ii. localization near the band edges.}, J. Phys. C: Solid State
  Phys. \textbf{7} (1974), 65--75.

\bibitem{AW}
Michael Aizenman and Simone Warzel, \emph{Disorder-induced delocalization on
  tree graphs}, Mathematical results in quantum physics, World Sci. Publ.,
  Hackensack, NJ, 2011, pp.~107--109. \MR{2885163}

\bibitem{AGZ10}
Greg~W. Anderson, Alice Guionnet, and Ofer Zeitouni, \emph{An introduction to
  random matrices}, Cambridge Studies in Advanced Mathematics, vol. 118,
  Cambridge University Press, Cambridge, 2010. \MR{2760897 (2011m:60016)}

\bibitem{BaiSil}
Zhidong Bai and Jack~W. Silverstein, \emph{Spectral analysis of large
  dimensional random matrices}, second ed., Springer Series in Statistics,
  Springer, New York, 2010.

\bibitem{BDG09}
Serban Belinschi, Amir Dembo, and Alice Guionnet, \emph{Spectral measure of
  heavy tailed band and covariance random matrices}, Comm. Math. Phys.
  \textbf{289} (2009), no.~3, 1023--1055.

\bibitem{BAG1}
G{\'e}rard Ben~Arous and Alice Guionnet, \emph{The spectrum of heavy tailed
  random matrices}, Comm. Math. Phys. \textbf{278} (2008), no.~3, 715--751.
  \MR{2373441 (2008j:60015)}

\bibitem{BG08}
\bysame, \emph{The spectrum of heavy tailed random matrices}, Comm. Math. Phys.
  \textbf{278} (2008), no.~3, 715--751.

\bibitem{BAG21}
Gerard Ben~Arous and Alice Guionnet, \emph{Wigner matrices}, Handbook in Random
  matrix theory, editors: G. Akemann, J. Baik and P.Di Francesco, vol. Chapter
  21, Oxford University Press, 2010.

\bibitem{BCC}
Charles Bordenave, Pietro Caputo, and Djalil Chafai, \emph{Spectrum of large
  random reversible markov chains - heavy-tailed weights on the complete
  graph}, Ann. Probab. \textbf{39} (2011), no.~4, 1544--1590.

\bibitem{BCC11}
Charles Bordenave, Pietro Caputo, and Djalil Chafa{\"{\i}}, \emph{Spectrum of
  large random reversible {M}arkov chains: heavy-tailed weights on the complete
  graph}, Ann. Probab. \textbf{39} (2011), no.~4, 1544--1590.

\bibitem{BG}
Charles Bordenave and Alice Guionnet, \emph{Localization and delocalization of
  eigenvectors for heavy-tailed random matrices}, Probab. Theory Related Fields
  \textbf{157} (2013), no.~3-4, 885--953. \MR{3129806}

\bibitem{BouchaudCizeau}
Jean-Philippe Bouchaud and Pierre Cizeau, \emph{Theory of {L}\'evy matrices},
  Phys. Rev. E \textbf{3} (1994), 1810--1822.

\bibitem{MR929421}
Chris Brislawn, \emph{Kernels of trace class operators}, Proc. Amer. Math. Soc.
  \textbf{104} (1988), no.~4, 1181--1190. \MR{929421 (89d:47059)}

\bibitem{PhysRevE.75.051126}
Zdzis\l{}aw Burda, Jerzy Jurkiewicz, Maciej~A. Nowak, Gabor Papp, and Ismail
  Zahed, \emph{Free random l\'evy and wigner-l\'evy matrices}, Phys. Rev. E
  \textbf{75} (2007), 051126.

\bibitem{erdossurvey}
L{\'a}szl{\'o} Erd{\H{o}}s, \emph{Universality of {W}igner random matrices: a
  survey of recent results}, Uspekhi Mat. Nauk \textbf{66} (2011), no.~3(399),
  67--198. \MR{2859190}

\bibitem{EKY11}
L{\'a}szl{\'o} Erd{\H{o}}s, Antti Knowles, and Horng-Tzer Yau, \emph{Spectral
  statistics of {E}rd{\H{o}}s-{R}{\'e}nyi graphs {I}: Local semicircle law},
  arXiv:1103.1919 (2011).

\bibitem{ESY09a}
L{\'a}szl{\'o} Erd{\H{o}}s, Benjamin Schlein, and Horng-Tzer Yau, \emph{Local
  semicircle law and complete delocalization for {W}igner random matrices},
  Comm. Math. Phys. \textbf{287} (2009), no.~2, 641--655.

\bibitem{ESY09b}
\bysame, \emph{Semicircle law on short scales and delocalization of
  eigenvectors for {W}igner random matrices}, Ann. Probab. \textbf{37} (2009),
  no.~3, 815--852.

\bibitem{PH}
David~A. Huse and Arijeet Pal, \emph{The many-body localization phase
  transition}, Phys.Rev. B \textbf{82} (2010), 174411.

\bibitem{MR647629}
Konrad J{\"o}rgens, \emph{Linear integral operators}, Surveys and Reference
  Works in Mathematics, vol.~7, Pitman (Advanced Publishing Program), Boston,
  Mass.-London, 1982, Translated from the German by G. F. Roach. \MR{647629
  (83j:45001)}

\bibitem{MR1492789}
Abel Klein, \emph{Extended states in the {A}nderson model on the {B}ethe
  lattice}, Adv. Math. \textbf{133} (1998), no.~1, 163--184.

\bibitem{MR1850453}
Louis Nirenberg, \emph{Topics in nonlinear functional analysis}, Courant
  Lecture Notes in Mathematics, vol.~6, New York University, Courant Institute
  of Mathematical Sciences, New York; American Mathematical Society,
  Providence, RI, 2001, Chapter 6 by E. Zehnder, Notes by R. A. Artino, Revised
  reprint of the 1974 original. \MR{1850453 (2002j:47085)}

\bibitem{MR2154153}
Barry Simon, \emph{Trace ideals and their applications}, second ed.,
  Mathematical Surveys and Monographs, vol. 120, American Mathematical Society,
  Providence, RI, 2005. \MR{2154153 (2006f:47086)}

\bibitem{Slanina}
F.~Slanina, \emph{Localization of eigenvectors in random graphs}, Eur. Phys. B
  (2012), 85:361.

\bibitem{MR2784665}
Terence Tao and Van Vu, \emph{Random matrices: universality of local eigenvalue
  statistics}, Acta Math. \textbf{206} (2011), no.~1, 127--204. \MR{2784665
  (2012d:60016)}

\bibitem{TBT}
E.~Tarquini, G.~Biroli, and M.~Tarzia, \emph{Level statistics and localization
  transitions of levy matrices}, Phys.Rev. Lett \textbf{116} (2016), 010601.

\end{thebibliography}

\paragraph{Acknowledgments}
This work has benefited from insightful discussions  with Serban Belinschi and Benjamin Schlein. It is a pleasure to thank them.

\end{document}